\documentclass[11pt,leqno]{article}

\usepackage[english]{babel}

\usepackage{amsmath}

\usepackage{amssymb}

\usepackage{mathrsfs,upref}

\usepackage[all]{xy}
\UseComputerModernTips
\newcommand{\dsim}{\text{\smash{\lower0.9ex
\hbox{\normalsize $\sim$}}}}
\usepackage{latexsym}

\usepackage{amsfonts}

\usepackage[dvips]{graphicx}

\usepackage{mathtools,amsthm,enumerate}
\addto\captionsenglish{}
\newcommand{\earrow}
{\mathrel{\text{\smash{\lower.477ex\hbox{$\overset{\text{\normalsize \smash{\lower.55ex\hbox{$\sim\,$}}}}{\to}$}}}}}
\newcommand{\ftimes}{\mathop{\times}\limits}
\newcommand{\tens}{\mathop{\otimes}\limits}
\newcommand{\MS}{\operatorname{SS}}
\newcommand{\bD}{\boldsymbol{\mathsf{D}}}
\newcommand{\bDb}{D^b}
\newcommand{\ch}{\operatorname{Ch}}

\numberwithin{equation}{section}

\newcommand{\B}{\mathcal{B}}

\newcommand{\C}{\mathcal{C}}

\newcommand{\M}{\mathcal{M}}

\newcommand{\W}{\mathcal{W}}

\renewcommand{\mod}{\mathrm{Mod}}

\newcommand{\op}{\mathrm{Op}}

\newcommand{\CC}{\mathbb{C}}

\newcommand{\R}{\mathbb{R}}

\newcommand{\N}{\mathbb{N}}

\newcommand{\JJ}{\hat{J}}

\newcommand{\OO}{\mathcal{O}}

\newcommand{\E}{\mathcal{E}}

\newcommand{\iso}{\stackrel{\sim}{\to}}

\newcommand{\db}{\mathcal{D}\mathit{b}}

\newcommand{\OW}{\OO^\mathrm{w}}

\newcommand{\wtens}{\overset{\mathrm{w}}{\otimes}}

\newcommand{\CHZ}{{\operatorname{H}^0}}

\newcommand{\rh}{\mathit{R}\mathcal{H}\mathit{om}}

\newcommand{\supp}{\mathrm{supp}}

\newcommand{\RP}{\mathbb{R}^{{\scriptscriptstyle{+}}}}

\newcommand{\imin}[1]{#1^{-1}}

\newcommand{\lind}[1]{\underset{#1}{\underrightarrow{\lim}}}

\newcommand{\dt}[3]{{#1} \to {#2} \to {#3} \stackrel{+}{\to}}

\newcommand{\proddTM}{T^*_{M_1}\iota(M_1) \underset{X}{\times} \dots \underset{X}{\times} T^*_{M_{\ell}}\iota (M_{\ell})}

\theoremstyle{plain}
\newtheorem{teo}{Theorem}[section]
\newtheorem{cor}[teo]{Corollary}
\newtheorem{prop}[teo]{Proposition}
\newtheorem{lem}[teo]{Lemma}
\theoremstyle{definition}
\newtheorem{es}[teo]{Example}
\newtheorem{oss}[teo]{Remark}
\newtheorem{df}[teo]{Definition}
\newtheorem{con}[teo]{Condition}

\usepackage{latexsym}

\author{Naofumi Honda\ \ \ \ Luca Prelli\ \ \ \ Susumu Yamazaki}

\title{\bf{\sc{Multi-microlocalization and microsupport}}}

\date{}

\begin{document}

\maketitle
\ensuremath{
\fontdimen13 \textfont 2=.43em   
\fontdimen13 \scriptfont 2=0.26em 
\fontdimen17 \textfont 2=.3em 
\fontdimen17 \scriptfont 2= .25em  
\fontdimen14\textfont2= \fontdimen13\textfont2 
\fontdimen14 \scriptfont 2= \fontdimen13 \scriptfont 2
\fontdimen14 \scriptscriptfont 2= \fontdimen13 \scriptscriptfont 2
\fontdimen16\textfont2= \fontdimen17\textfont2 
\fontdimen16 \scriptfont 2= \fontdimen17 \scriptfont 2
\fontdimen16 \scriptscriptfont 2= \fontdimen17 \scriptscriptfont 2
}

\thispagestyle{empty}
\begin{abstract}
The purpose of this paper is to establish the foundations of multi-microlocalization,
in particular, to give
the fiber formula for the multi-microlocalization functor
and estimate of microsupport of a multi-microlocalized object. We also give some applications of these results.
\end{abstract}

\tableofcontents

\addcontentsline{toc}{section}{\textbf{Introduction}}

\section*{Introduction}

A microlocalized object of a sheaf $F$ along a closed submanifold $M$
is
locally described by, roughly speaking,
local cohomology groups of $F$
with support in a dual cone of the edge $M$.
As a result it can be tightly related with, via \v{C}ech cohomology groups,
boundary values of local sections of $F$
defined on open cones of the edge $M$.
It is well known that, for example, Sato's microfunctions, which are obtained by
applying the microlocalization functor
along a real analytic manifold $M$ to the sheaf of holomorphic functions,
can be regarded as boundary values of holomorphic functions locally
defined on wedges with the edge $M$.

We sometimes, in study of partial differential equations, need to
consider a boundary value of a function defined on a cone along
a family $\chi$ of several closed submanifolds.
In such a study, J.~M.~Delort \cite{De96} had introduced
{\it{simultaneous microlocalization}} along a normal crossing divisor $\chi$
which gives a boundary value of a function defined on a dual poly-sector.
{\it{Bi-microlocalization}} along submanifolds $\chi=\{M_1,\,M_2\}$
with $M_1 \subset M_2$ was also introduced by P.~Schapira and K.~Takeuchi
\cite{ST94} and \cite{Ta96} which defines a different kind of a boundary value.

On the other hand, in the paper \cite{HP},
the first and the second authors of this article established
the notion of {\it{the multi-normal cone}}
for a family $\chi$ of closed submanifolds with a suitable configuration,
and they also constructed
{\it{the multi-specialization functor}} along $\chi$.
We can observe that
cones appearing in simultaneous microlocalization and
bi-microlocalization are characterized by using the multi-normal cone
and that the both microlocalization functors coincide with
the {\it{multi-microlocalization functor}}
along $\chi$ where the latter functor is obtained by repeated application
of Sato's Fourier transformation
to the multi-specialization functor. Hence multi-microlocalization
gives us a uniform machinery for the both simultaneous microlocalization
and bi-microlocalization.
The purpose of this paper is to establish the foundations of multi-microlocalization,
in particular, to give
the fiber formula for the multi-microlocalization functor
and estimate of microsupport of a multi-microlocalized object.
We briefly explain, in what follows, these two important results and their meanings.

The most fundamental question for the multi-microlocalization functor $\mu_\chi$
along closed submanifolds $\chi = \{M_1, \dots, M_\ell\}$
is a shape of a cone on which a boundary value given by $\mu_\chi$
is defined.  The fiber formula gives us an explicit answer:
A germ of $H^k(\mu_\chi(F))$ is isomorphic to local cohomology groups
$\lind{G} H^k_G(F)$ where $G$ is a vector sum of closed cones $G_i$'s and
each $G_i$ is defined in the similar way as that in
the fiber formula of the usual microlocalization functor along $M_i$.
Therefore the multi-microlocalization functor can be understood as
a natural extension of the usual microlocalization functor.
Once we have grasped a geometrical aspect of multi-microlocalization, then
the next fundamental problem to be considered is
estimate of microsupport of $\mu_\chi(F)$
by that of $F$, for which the answer is quite simple and beautiful:
The $SS(\mu_\chi(F))$ is contained in the multi-normal cone of $SS(F)$ along $\chi^*$.
Here $\chi^*$ is a family of
Lagrangian submanifolds $\{T^*_{M_1}X, \dots, T^*_{M_\ell}X\}$.
This shows, in particular, soundness of our framework in the sense that
the sharp estimate can be achieved by a geometrical tool (the multi-normal cone)
already prepared in our framework.
These two results have many applications, and some of them will be given in the
last two sections of this paper.

\

The paper is organized as follows: We briefly recall, in Section 1, the theory of
the multi-specialization developed in \cite{HP}.
Then, in Section 2, we define the multi-microlocalization functor by repeatedly applying
Sato's Fourier transformation to the multi-specialization functor. After showing
several basic properties of the functor, we establish
a fiber formula which explicitly describes a stalk of a multi-microlocalized
object. By the fiber formula and an edge of the wedge theorem with bounds shown
in Section 4, we can construct the sheaf of microfunctions along $\chi$ that is
a natural extension of the sheaf of Sato's microfunctions.
In section 3, after some geometrical preparations, we give estimate of microsupport
of a microlocalized object, that is our main result.
Several applications of this result to $\mathcal{D}$ modules are studied in Section 5.

\section{Multi-specialization: a review}

In this section we recall some results of \cite{HP}. We first fix some notations, then we recall the notion of multi-normal deformation and the definition of the functor of multi-specialization with some basic properties.

\subsection{Notations}

Let $X$ be a real analytic manifold with $\operatorname{dim}{X} = n$, and let
$\chi = \{M_1,\dots,M_\ell\}$ be
a family of closed submanifolds in $X$ ($\ell \ge 1$).
Throughout the paper all the manifolds are always
assumed to be countable at infinity.

We set, for $N \in \chi$ and $p \in N$,
\[
	\operatorname{NR}_p(N) := \{M_j \in \chi;\,
p \in M_j,\,N \nsubseteq M_j \text{ and } M_j \nsubseteq N\}.
\]
Let us consider the following conditions for $\chi$.
\begin{itemize}
\item[H1] Each $M_j \in \chi$ is connected and the
submanifolds are mutually distinct, i.e. $M_j \neq M_{j'}$ for $j \neq {j'}$.
\item[H2] For any $N \in \chi$ and $p \in N$ with $\operatorname{NR}_p(N) \ne \emptyset$,
we have
\begin{equation}{\label{eq:geometric-condition}}
\left(
\underset{
M_j \in \operatorname{NR}_p(N)
}{\bigcap} T_pM_j\right) + T_pN = T_pX.
\end{equation}
\end{itemize}
Note that, if $\chi$ satisfies the condition H2, 
the configuration of two submanifolds
must be either 1.~or 2.~below.
\begin{enumerate}
\item Both submanifolds intersect transversely.
\item One of them contains the other.
\end{enumerate}
If $\chi$ satisfies the condition H2, then
for any $p \in X$, there exist a neighborhood $V$ of $p$ in $X$, a system
of local coordinates $(x_1,\dots,x_n)$ in $V$ and a family of subsets $\{I_j\}_{j=1}^\ell$
of the set $\{1,2,\dots,n\}$ for which the following conditions hold.
\begin{enumerate}
\item Either $I_k \subset I_j$, $I_j \subset I_k$ or $I_k \cap I_j = \emptyset$ holds
$(k,j \in \{1,2,\dots,\ell\})$.
\item A submanifold $M_j \in \chi$ with $p \in M_j$ $(j = 1,2,\dots,\ell)$
is defined by $\{x_i = 0; i \in I_j\}$ in $V$.
\end{enumerate}

We set, for $N \in \chi$,
\begin{equation}{\label{eq:def-chi}}
\iota_\chi(N) := \underset{N \varsubsetneqq M_j}{\bigcap} M_j.
\end{equation}
Here $\iota_{\chi}(N) := X$ for convention
 if there exists no $j$ with $N \varsubsetneqq M_j$.
When there is no risk of confusion, we write for short $\iota(N)$ instead of $\iota_\chi(N)$. We also assume the condition H3 below for simplicity.

\begin{itemize}
\item[H3] $M_j \ne \iota(M_j)$ for any $j \in \{1,2,\dots, \ell\}$.
\end{itemize}

In local coordinates let $I_1,\ldots ,I_\ell \subseteq \{1,\ldots,n\}$ such that $M_i=\{x_k=0\;;\;k\in I_i\}$.
Note that the family $\chi$ satisfies the conditions H1, H2 and H3 if and only if
$I_1,\dots,I_\ell$
satisfy
the corresponding conditions
\begin{equation}\label{eq:cond-H}
\begin{aligned}
	&\text{(i) either }  I_j \subsetneqq I_k,\, I_k \subsetneqq I_j \text{ or }
	I_j \cap I_k = \emptyset \text{ holds for any $j \ne k$}, \\
	&\text{(ii) }  \left(\underset{I_k \subsetneqq I_j}{\bigcup}I_k \right)
	\subsetneqq I_j
	 \text{ for any $j$}.
\end{aligned}
\end{equation}
Hence, for any $j \in \{1,2,\dots,\ell\}$,
the set
\begin{equation}{\label{eq:def-hat-I}}
\hat{I}_j := I_j \setminus \left(\underset{I_k \varsubsetneqq I_j}{\bigcup} I_k\right)
\end{equation}
is not empty by the condition H3. Further it follows from the
conditions H1, H2 and H3 that, for each $j \in \{1,\dots,\ell\}$, there exist
unique $j_1,\dots,j_p \in \{1,\dots,\ell\}$ such that
\[
I_j=\hat{I}_{j_1} \sqcup \cdots \sqcup \hat{I}_{j_p}
\]
and $I_k \subseteq I_j$ for all $k \in \{j_1,\dots,j_p\}$ (equality holds only if $k=j$). In particular
\begin{equation}{\label{eq:formula-union-hat-I}}
\underset{1 \le j \le \ell}{\bigcup} I_j = \hat{I}_1 \sqcup \cdots \sqcup \hat{I}_\ell.
\end{equation}
Set, for $i \in \{1,\dots,n\}$
\begin{equation}{\label{eq:def-J_i}}
J_i =\{j\in\{1,\ldots,\ell\}\;;\;i \in I_j\}.
\end{equation}
It follows from the proof of Proposition 1.3 of \cite{HP} that
\begin{equation}\label{eq:equality of J}
J_\alpha=J_\beta \ \ \Leftrightarrow \ \ \alpha,\beta \in \hat{I}_j
\end{equation}
for some $j \in \{1,\dots,\ell\}$.

\begin{lem}\label{lem:inclusions of J} Let $I_i \supseteq I_j$. Then
 for each $\alpha \in \hat{I}_i$ and $\beta \in \hat{I}_j$ we have $J_\alpha \subseteq J_\beta$.
\end{lem}
\begin{proof} Let $\alpha \in \hat{I}_i$ and $\beta \in \hat{I}_j$. By definition of $\hat{I}_i$ and condition H2 we have
\[
k \in J_\alpha \ \ \Leftrightarrow \ \ I_k \supseteq I_i.
\]
Since $I_i \supseteq I_j$ we have
\[
 I_k \supseteq I_j \ \ \Leftrightarrow \ \ k \in J_\beta.
\]
Then $J_\alpha \subseteq J_\beta$.
\end{proof}

Thanks to the previous result we can introduce the following notation.
Set for convenience
\begin{equation}{\label{eq:def-I_0}}
I_0 = \hat{I}_0 := \{1,\dots,n\} \setminus \bigcup_{j=1}^\ell I_j.
\end{equation}
Then, in local coordinates, we can write the coordinates
$(x_1,\dots,x_n)$ by
\begin{equation}
(x^{(0)},x^{(1)},\dots,x^{(\ell)}),
\end{equation}
where $x^{(j)}$ denotes the coordinates $(x_i)_{i \in \hat{I}_j}$ ($j=0,\dots,\ell$).
By taking \eqref{eq:equality of J} into account, we can also define, for $j \in \{0,1,\dots,\ell\}$
\begin{equation}{\label{eq:def-hat-J_j}}
\JJ_j=\{k\in\{1,\ldots,\ell\}\;;\;\hat{I}_j \subseteq I_k\}=\{k\in\{1,\ldots,\ell\}\;;\;I_j \subseteq I_k\}.
\end{equation}
Note that, with this notation, we have $\JJ_0=\emptyset$. Moreover, by \eqref{eq:equality of J} we have $J_\alpha=J_\beta=\JJ_j$ for each $j \in \{1,\dots,\ell\}$ and each $\alpha,\beta \in \hat{I}_j$. In particular, by Lemma  \ref{lem:inclusions of J}
\begin{equation}\label{eq: inclusion of JJ}
I_i \subseteq I_j \ \ \Rightarrow \ \ \JJ_j \subseteq \JJ_i.
\end{equation}

\subsection{Multi-normal deformation}

In \cite{HP} the notion of multi-normal deformation was introduced.
Here we consider a slight generalization where we replace the condition H2 with the weaker one.
Let $\chi = \{M_1,\dots,M_\ell\}$ be a family of closed submanifolds of $X$.
We say that $\chi$ is {\it{simultaneously linearizable}} on $M = M_1 \cap \dots \cap M_\ell$
if  for every $x \in M$ there exist a neighborhood $V$ of $x$ and a system of local coordinates
$(x_1, \dots, x_n)$ there for which we can find subsets $I_j$'s of $\{1, \dots, n\}$ such
that each $M_j \cap V$ is defined by equations $x_i = 0$ ($i \in I_j$). Note that
if $\chi$ satisfies the condition H2, then it is simultaneously linearizable.
Now, through the section, we assume that $\chi$ is simultaneously linearizable on $M$.

First recall the classical construction of \cite{KS90} of the normal deformation
of $X$ along $M_1$. We denote it by $\widetilde{X}_{M_1}$ and
we denote by $t_1\in\R$ the deformation parameter.
Let $\Omega_{M_1}=\{t_1>0\}$ and let us identify $\imin s(0)$ with $T_{M_1}X$.
We have the commutative diagram
\begin{equation}\label{eq:normal-deformation}
\xymatrix{T_{M_1}X \ar[r]^{s_{M_1}} \ar[d]^{\tau_{M_1}} & \widetilde{X}_{M_1} \ar[d]^{p_{M_1}} &
\Omega_{M_1} \ar[l]_{\ \ i_{\Omega_{M_1}}} \ar[dl]^{\widetilde{p}_{M_1}} \\
M \ar[r]^{i_{M_1}} & X. & }
\end{equation}
Set $\widetilde{\Omega}_{M_1}=\{(x;\,t_1)\;;\;t_1\neq 0\}$ and define
\[
\widetilde{M}_2:=\overline{(p_{M_1}|_{\widetilde{\Omega}_{M_1}})^{-1}M_2}.
\]
Then $\widetilde{M}_2$ is a closed smooth submanifold of $\widetilde{X}_{M_1}$.
\begin{oss}
 One cannot expect the smoothness of $\widetilde{M}_2$ without the simultaneously
linearizable condition. For example, let $X={\mathbb{R}^2}$ with $(x_1, x_2)$ and
let $M_1 = \{x_1 = 0\}$, $M_2 = \{x_1 - x_2^2 = 0\}$. Then in $\widetilde{X}_{M_1}$
with coordinates $(x_1,x_2;\, t_1)$ we have
\[
\widetilde{M}_2 :=\overline{\{t_1x_1 - x_2^2 = 0,\, t_1 \ne 0\}}
= \{t_1x_1 - x_2^2 = 0\}
\]
which is singular at $(0,0;\,0)$.
\end{oss}

Now we can define the normal deformation along $M_1,M_2$ as
\[
\widetilde{X}_{M_1,M_2}:=(\widetilde{X}_{M_1})^\sim_{\widetilde{M}_2}.
\]
Then we can define recursively the normal deformation along
$\chi$ as
\[
\widetilde{X}=\widetilde{X}_{M_1,\ldots,M_\ell}:=(\widetilde{X}_{M_1,\ldots,M_{\ell-1}})^\sim_{\widetilde{M}_\ell}.
\]
Set $S_\chi=\{t_1,\ldots,t_\ell=0\}$, $M=\bigcap_{i=1}^\ell M_i$ and $\Omega_\chi=\{t_1,\ldots,t_\ell>0\}$. Then we have the commutative diagram
\begin{equation}\label{eq:multi-normal-deformation}
\xymatrix{S_\chi \ar[r]^s \ar[d]^\tau & \widetilde{X} \ar[d]^p &
\Omega_\chi \ar[l]_{\ \ i_\Omega} \ar[dl]^{\widetilde{p}} \\
M \ar[r]^{i_M} & X. & }
\end{equation}

Let us consider the diagram \eqref{eq:multi-normal-deformation}.
In local coordinates let $I_1,\ldots ,I_\ell \subseteq \{1,\ldots,n\}$ such that $M_i=\{x_k=0\;;\;k\in I_i\}$. For $j \in \{0,\dots,\ell\}$ set
\[
\JJ_j=\{k\in\{1,\ldots,\ell\}\;;\;\hat{I}_j \subseteq I_k\}, \ \ \ \ t_{\JJ_j}=\prod_{k\in \JJ_j}t_k,
\]
where $t_1,\ldots,t_\ell\in\R$ and $t_{\JJ_0}=1$.
Then $p:\widetilde{X}\to X$ is defined by
\[
(x^{(0)},x^{(1)}\ldots,x^{(\ell)};\,t_1,\ldots,t_\ell) \mapsto (t_{\JJ_0}x^{(0)},t_{\JJ_1}x^{(1)},\ldots,t_{\JJ_\ell}x^{(\ell)}).
\]

\begin{df} Let $Z$ be a subset of $X$.
The multi-normal cone to $Z$ along $\chi$ is the set
$
C_\chi(Z)=\overline{\imin{\widetilde{p}}(Z)} \cap S_\chi.
$
\end{df}

\begin{lem} \label{lem: multi-normal cone} Let $p=(p^{(0)},p^{(1)},\dots,p^{(\ell)};\,0,\dots,0) \in S_\chi$ and let $Z \subset X$. The following conditions are equivalent:
\begin{enumerate}
\item $p \in C_\chi(Z)$.
\item There exist sequences $\{(c_{1,m},\dots,c_{\ell, m})\} \subset (\RP)^\ell$ and $\{(q^{(0)}_m,q^{(1)}_m,\dots,q^{(\ell)}_m)\} \subset Z$ such that $q_m^{(j)}c_{\JJ_j} \to p^{(j)}$, $j=0,1,\dots,\ell$ and $c_{m,j} \to +\infty$, $j=1,\dots,\ell$.
\end{enumerate}
\end{lem}
\begin{proof}  Let us prove 1. $\Rightarrow$ 2. Since $p \in \overline{\imin{\widetilde{p}}(Z)} \cap S_\chi$ there exist sequences $\{(p^{(0)}_m,p^{(1)}_m,\dots,p^{(\ell)}_m)\} \subset Z$, $\{(t_{1, m},\dots,t_{\ell, m})\} \subset (\RP)^\ell$ such that
\[\left\{
    \begin{array}{ll}
      t_{j,m} \to 0, \ \ \ \ \hbox{$j=1,\dots,\ell$,} \\
      (p^{(0)}_m,p^{(1)}_m,\dots,p^{(\ell)}_m) \to (p^{(0)},p^{(1)},\dots,p^{(\ell)}), \\
      (p^{(0)},p^{(1)}t_{\JJ_1,m},\dots,p^{(\ell)}t_{\JJ_\ell,m}) \in Z.
    \end{array}
  \right.
\]
Set $\imin {t_{j,m}}=c_{j,m}$, $j=1,\dots,\ell$ and $p^{(j)}_mt_{\JJ_j,m}=q^{(j)}_m$, $j=0,1,\dots,\ell$. Then we have $\{(q^{(0)}_m,q^{(1)}_m,\dots,q^{(\ell)}_m)\} \subset Z$, $q_m^{(j)}c_{\JJ_j} \to p^{(j)}$, $j=0,1,\dots,\ell$ and $c_{m,j} \to +\infty$, $j=1,\dots,\ell$.

Let us prove 2. $\Rightarrow$ 1. Suppose that there exist sequences $\{(c_{1,m},\dots,c_{\ell, m})\} \subset (\RP)^\ell$ and $\{(q^{(0)}_m,q^{(1)}_m,\dots,q^{(\ell)}_m)\} \subset Z$ such that $q_m^{(j)}c_{\JJ_j} \to p^{(j)}$, $j=0,1,\dots,\ell$ and $c_{m,j} \to +\infty$, $j=1,\dots,\ell$. Define $p_m^{(j)}=q_m^{(j)}c_{\JJ_j}$, $j=0,1,\dots,\ell$ and $t_{j,m}=\imin {c_{m,j}}$. Then clearly
\[\left\{
    \begin{array}{ll}
      t_{j,m} \to 0, \ \ \ \ \hbox{$j=1,\dots,\ell$,} \\
      (p^{(0)}_m,p^{(1)}_m,\dots,p^{(\ell)}_m) \to (p^{(0)},p^{(1)},\dots,p^{(\ell)}), \\
      (p^{(0)},p^{(1)}t_{\JJ_1,m},\dots,p^{(\ell)}t_{\JJ_\ell,m}) \in Z.
    \end{array}
  \right.
\]
So $p \in C_\chi(Z)$.
\end{proof}

Let us consider the canonical map $T_{M_j}\iota(M_j) \overset{}{\to} M_j \hookrightarrow X$, $j=1,\dots,\ell$,
we write for short \[\underset{X,1\leq j\leq \ell}{\times}T_{M_j}\iota(M_j):=T_{M_1}\iota(M_1) \underset{X}{\times} T_{M_2}\iota(M_2) \underset{X}{\times} \cdots \underset{X}{\times} T_{M_\ell}\iota(M_\ell).\]
When $\chi$ satisfies the conditions H1, H2 and H3 we have
\begin{equation}
S_\chi \simeq \underset{X,1\leq j\leq \ell}{\times}T_{M_j}\iota(M_j).
\end{equation}

\begin{oss}
When $\chi$ satisfies conditions H1, H2 and H3,
the zero-section $S_\chi$ becomes a vector bundle over $M$.
However, in general, the simultaneously linearizable condition is not
enough to assure the existence of a vector bundle structure on
$S_\chi$. The important exceptional case where
$\chi$ does not satisfy H2 but $S_\chi$ has a vector bundle structure will be
studied in \S\, \ref{sec:geometry}.
\end{oss}

\begin{es} Let us see two typical examples of multi-normal deformations in the complex case. Let $X={\mathbb C}^2$ ($\simeq\R^4$ as a real manifold)
with coordinates $(z_1, z_2)$.
\begin{enumerate}
\item (Majima) Let $\chi=\{M_1,M_2\}$ with $M_1 = \{z_1 = 0\}$ and $M_2 = \{z_2 = 0\}$. Then $\chi$ satisfies H1, H2 and H3.
We have $I_1=\{1\}$, $I_2=\{2\}$, $J_1=\{1\}$, $J_2=\{2\}$ (in $\R^4$, if $z_1=(x_1,x_2)$ and $z_2=(x_3,x_4)$ we have $I_1=\{1,2\}$, $I_2=\{3,4\}$, $J_1=J_2=\{1\}$, $J_3=J_4=\{2\}$).
The map $p:\widetilde{X} \to X$ is defined by
\[
(z_1,z_2;\,t_1,t_2) \mapsto (t_1z_1,t_2z_2).
\]
Remark that the deformation is real though $X$ is complex. In particular $t_1, t_2 \in \R$.
We have $\iota(M_1)=\iota(M_2)=X$ and then the zero section $S$ of $\widetilde{X}$ is isomorphic to $T_{M_1}X \underset{X}{\times} T_{M_2}X$.

\item (Takeuchi) Let $\chi=\{M_1,M_2\}$ with $M_1 = \{0\}$ and $M_2 = \{z_2 = 0\}$. Then $\chi$ satisfies H1, H2 and H3.
We have $I_1=\{1,2\}$, $I_2=\{2\}$, $J_1=\{1\}$, $J_2=\{1,2\}$ (in $\R^4$, if $z_1=(x_1,x_2)$ and $z_2=(x_3,x_4)$ we have $I_1=\{1,2,3,4\}$, $I_2=\{3,4\}$, $J_1=J_2=\{1\}$, $J_3=J_4=\{1,2\}$).
The map $p:\widetilde{X} \to X$ is defined by
\[
(z_1,z_2;\,t_1,t_2) \mapsto (t_1z_1,t_1t_2z_2).
\]
We have $\iota(M_1)=M_2$, $\iota(M_2)=X$ and then the zero section $S$ of $\widetilde{X}$ is isomorphic to $T_{M_1}M_2 \underset{X}{\times} T_{M_2}X$.

\end{enumerate}
\end{es}

\begin{es} Let us see three typical examples of multi-normal deformations in the real case. Let $X={\mathbb R}^3$ with coordinates $(x_1, x_2, x_3)$.
\begin{enumerate}
\item (Majima)  Let $\chi=\{M_1,M_2,M_3\}$ with
$M_1 = \{x_1=0\}$, $M_2 = \{x_2 = 0\}$ and $M_3 = \{x_3 = 0\}$. Then $\chi$ satisfies H1, H2 and H3. We have $I_1=\{1\}$, $I_2=\{2\}$, $I_3=\{3\}$, $J_1=\{1\}$, $J_2=\{2\}$, $J_3=\{3\}$.
The map $p:\widetilde{X} \to X$ is defined by
\[
(x_1,x_2,x_3;\,t_1,t_2,t_3) \mapsto (t_1x_1,t_2x_2,t_3x_3).
\]
We have $\iota(M_1)=\iota(M_2)=\iota(M_3)=X$ and then the zero section $S$ of $\widetilde{X}$ is isomorphic to $T_{M_1}X \underset{X}{\times} T_{M_2}X \underset{X}{\times} T_{M_3}X$.

\item (Takeuchi)
Let $\chi=\{M_1,M_2,M_3\}$ with
$M_1 = \{0\}$, $M_2 = \{x_2 = x_3 = 0\}$ and $M_3 = \{x_3 = 0\}$. Then $\chi$ satisfies H1, H2 and H3. We have $I_1=\{1,2,3\}$, $I_2=\{2,3\}$, $I_3=\{3\}$, $J_1=\{1\}$, $J_2=\{1,2\}$, $J_3=\{1,2,3\}$.
The map $p:\widetilde{X} \to X$ is defined by
\[
(x_1,x_2,x_3;\,t_1,t_2,t_3) \mapsto (t_1x_1,t_1t_2x_2,t_1t_2t_3x_3).
\]
We have $\iota(M_1)=M_2$, $\iota(M_2)=M_3$, $\iota(M_3)=X$ and then the zero section $S$ of $\widetilde{X}$ is isomorphic to $T_{M_1}M_2 \underset{X}{\times} T_{M_2}M_3 \underset{X}{\times} T_{M_3}X$.

\item (Mixed)
Let $\chi=\{M_1,M_2,M_3\}$ with
$M_1 = \{0\}$, $M_2 = \{x_2 = 0\}$ and $M_3 = \{x_3 = 0\}$. Then $\chi$ satisfies H1, H2 and H3. We have $I_1=\{1,2,3\}$, $I_2=\{2\}$, $I_3=\{3\}$, $J_1=\{1\}$, $J_2=\{1,2\}$, $J_3=\{1,3\}$.
The map $p:\widetilde{X} \to X$ is defined by
\[
(x_1,x_2,x_3;\,t_1,t_2,t_3) \mapsto (t_1x_1,t_1t_2x_2,t_1t_3x_3).
\]
We have $\iota(M_1)=M_2 \cap M_3$, $\iota(M_2)=\iota(M_3)=X$ and then the zero section $S$ is isomorphic to $T_{M_1}(M_2 \cap M_3) \underset{X}{\times} T_{M_2}X \underset{X}{\times} T_{M_3}X$.
\end{enumerate}
\end{es}

\noindent Let $q \in \underset{1 \le j \le \ell}{\bigcap} M_j$ and
$p_j = (q;\, \xi_j)$ be a point in $T_{M_j}\iota(M_j)$ ($j=1,2,\dots,\ell$).
We set
$
p = p_1 \underset{X}{\times} \dots \underset{X}{\times} p_\ell
\in \underset{X,1\leq j \leq \ell}{\times}T_{M_j}\iota(M_j),
$
and $\tilde{p}_j = (q;\, \tilde{\xi}_j) \in T_{M_j} X$ denotes the
image of the point $p_j$ by the canonical embedding $T_{M_j} \iota (M_j) \hookrightarrow T_{M_j} X$.
We denote by $\operatorname{Cone}_{\chi, j}(p)$ ($j=1,2,\dots,\ell$)
the set of open conic cones in
$(T_{M_j}X)_q \simeq {\R}^{n - {\rm dim} M_j}$ that contain the point
$\tilde{\xi}_j \in (T_{M_j} X)_q \simeq {\R}^{n - {\rm dim} M_j}$.\\

\begin{df}
We say that an open set $G\subset (TX)_q$ is a multi-cone along $\chi$
with direction to
$p \in \left(\underset{X,1\leq j\leq \ell}{\times}T_{M_j}\iota(M_j)\right)_q$ if
$G$ is written in the form
\[
G = \underset{1\leq j \leq\ell}{\bigcap} \pi_{j,\,q}^{-1}(G_j) \qquad
G_j \in \operatorname{Cone}_{\chi,j}(p)
\]
where $\pi_{j,\,q}: (TX)_q \to (T_{M_j}X)_q$ is the canonical projection.
We denote by $\operatorname{Cone}_\chi(p)$ the set of multi-cones along $\chi$
with direction to $p$.
\end{df}

For any $q \in X$,  there exists an isomorphism
$
\psi: X \simeq (TX)_q
$
near $q$ and $\psi(q) = (q;\,0)$ that satisfies $\psi(M_j) = (TM_j)_q$
for any $j=1,\dots,\ell$.

Let $Z$ be a subset of $X$.
When $\chi$ satisfies H1, H2 and H3 we also have the following equivalence:
$p \notin C_\chi(Z)$
if and only if
there exist an open subset $\psi(q) \in U \subset (TX)_q$
and a multi-cone $G \in \operatorname{Cone}_{\chi}(\psi_*(p))$ such that
$
\psi(Z) \cap G \cap U = \emptyset
$
holds.

\begin{es} We now give two examples of multi-cones in the complex case. Let $X={\mathbb C}^2$
with coordinates $(z_1, z_2)$.
\begin{enumerate}
\item (Majima)  Let $M_1 = \{z_1 = 0\}$ and $M_2 = \{z_2 = 0\}$.
Then $\operatorname{Cone}_{\chi}(p)$ for $p = (0,0;\,1,1)$ is nothing but
the set of multi sectors along $Z_1 \cup Z_2$ with their direction to $(1,1)$.
\item (Takeuchi) Let $M_1 = \{0\}$ and $M_2 = \{z_2 = 0\}$.
For $p = (0,0;\, 1,1) \in T_{M_1}M_2 \underset{X}{\times} T_{M_2}X$,
it is easy to see that
a cofinal set of $\operatorname{Cone}_{\chi}(p)$ is, for example, given by
the family of the sets
\[
\{(\eta_1, \eta_2);\,
\vert \eta_1 \vert < \epsilon |\eta_2|,\,
\eta_2 \in S
\}_{S \ni 1,\epsilon > 0},
\]
where $S$ is a sector in $\CC$ containing the direction $1$.
\end{enumerate}
\end{es}

\begin{es} We now give three examples of multi-cones in the real case. Let $X={\mathbb R}^3$ with coordinates $(x_1, x_2, x_3)$.
\begin{enumerate}
\item (Majima) Let
$M_1 = \{x_1 = 0\}$, $M_2 = \{x_2 = 0\}$ and $M_3 = \{x_3 = 0\}$.
For $p = (0,0,0;\, 1,1,1) \in T_{M_1}X \underset{X}{\times} T_{M_2}X \underset{X}{\times} T_{M_3}X$,
it is easy to see that
$\operatorname{Cone}_{\chi}(p)=\{(\RP)^3\}$.
\item (Takeuchi)
Let
$M_1 = \{0\}$, $M_2 = \{x_2 = x_3 = 0\}$ and $M_3 = \{x_3 = 0\}$.
For $p = (0,0,0;\, 1,1,1) \in T_{M_1}M_2 \underset{X}{\times} T_{M_2}M_3 \underset{X}{\times} T_{M_3}X$,
it is easy to see that
a cofinal set of $\operatorname{Cone}_{\chi}(p)$ is, for example, given by
the family of the sets
\[
\{(\xi_1, \xi_2, \xi_3);\,
\vert \xi_2 \vert + \vert \xi_3 \vert < \epsilon \xi_1,\,
\vert \xi_3 \vert < \epsilon \xi_2,\,
\xi_3 > 0
\}_{\epsilon > 0}.
\]
\item (Mixed)
Let
$M_1 = \{0\}$, $M_2 = \{x_2 = 0\}$ and $M_3 = \{x_3 = 0\}$.
For $p = (0,0,0;\, 1,1,1) \in T_{M_1}(M_2 \cap M_3) \underset{X}{\times} T_{M_2}X \underset{X}{\times} T_{M_3}X$,
a cofinal set of $\operatorname{Cone}_{\chi}(p)$ is, for example, given by
the family of the sets
\[
\{(\xi_1, \xi_2, \xi_3);\,
\vert \xi_2 \vert + \vert \xi_3 \vert < \epsilon \xi_1,\,
\xi_2 > 0,\,
\xi_3 > 0
\}_{\epsilon > 0}.
\]
\end{enumerate}
\end{es}

This definition is also compatible with the restriction to a subfamily of $\chi$.
Namely, let $k \le \ell$ and $K=\{j_1,\dots,j_k\}$ be a subset of $\{1,2,\dots,\ell\}$.
Set $\chi_K=\{M_{j_1},\ldots,M_{j_k}\}$ and $S_{K}:=T_{M_{j_1}}\iota_\chi(M_{j_1}) \underset{X}{\times} \cdots \underset{X}{\times} T_{M_{j_k}}\iota_\chi(M_{j_k}) \underset{X}{\times} M$.
Let $Z$ be a subset of $X$. Then we have
\[
C_\chi (Z) \cap S_{K} = C_{\chi_K} (Z) \cap S_{K}.
\]
In the following we will denote with the same symbol $C_{\chi_K}(Z)$ the normal cone with respect to $\chi_K$ and its inverse image via the map $\widetilde{X} \to \widetilde{X}_{M_{j_1},\dots,M_{j_k}}$.\\

\subsection{Multi-specialization}

Let $k$ be a field and denote by $\mod(k_{X_{sa}})$ (resp. $D^b(k_{X_{sa}})$) the category (resp. bounded derived category) of sheaves on the subanalytic site $X_{sa}$.
For the theory of sheaves on subanalytic sites we refer to \cite{KS01,Pr08}. For the theory of multi-specialization we refer to \cite{HP}.\\
Let $\chi$ be a family of submanifolds satisfying H1, H2 and H3.\\

\begin{df}\label{def1.11} The multi-specialization along $\chi$ is the functor
\begin{align*}
\nu^{sa}_\chi \colon D^b(k_{X_{sa}}) & \to  D^b(k_{S_{\chi sa}}) , \\
F & \mapsto  \imin sR\varGamma_{\Omega_\chi}\imin p F.
\end{align*}
\end{df}
\begin{oss}\label{stalk-nu}
We can give a description
of the sections of the multi-specialization of
$F \in D^b(k_{X_{sa}})$: let $V$ be a conic subanalytic open subset of
$S_\chi$. Then:
\[\operatorname{H}^j(V;\nu^{sa}_MF) \simeq \lind U \operatorname{H}^j(U;F),\]
where $U$ ranges through the family of open subanalytic subsets of $X$ such that
$C_\chi(X \setminus U) \cap V=\emptyset$. Let $p=(q;\,\xi) \in \underset{X,1\leq j\leq \ell}{\times}T_{M_j}\iota(M_j)$,
let $B_\epsilon \subset (TX)_q$ be an open ball of radius $\epsilon>0$ with its center at
the origin and set
\[\operatorname{Cone}_\chi(p,\, \epsilon):=\{G \cap B_\epsilon;\, G \in \operatorname{Cone}_\chi(p)\}.\] Applying the functor $\imin \rho \colon D^b(k_{S_{\chi sa}}) \to D^b(k_{S_\chi})$ (see \cite{Pr08} for details) we can calculate the fibers at $p \in \underset{X,1\leq j\leq \ell}{\times}T_{M_j}\iota(M_j)$ which are given by
\[
(\imin \rho H^j\nu^{sa}_\chi F)_p \simeq \lind {W} \operatorname{H}^j(W;F),
\]
where $W$ ranges through the family $\operatorname{Cone}_\chi(p,\,\epsilon)$
 for $\epsilon > 0$. If there is no risk of confusion, in the rest of the paper we will use the notation
\begin{eqnarray*}
\nu_\chi:=\imin \rho \nu^{sa}_\chi:D^b(k_{X_{sa}})  \to  D^b(k_{S_\chi}) .
\end{eqnarray*}
Remark that, if $F \in \mod(k_X)$ we have $\nu_\chi R\rho_*(F) \simeq \nu_\chi (F)$, the multi-specialization of $F$ defined with the classical Grothendieck operations.
\end{oss}

\section{Multi-microlocalization}

In this section we introduce the functor of multi-microlocalization as the Fourier-Sato transform of multi-specialization. We then compute its stalks as inductive limits of section supported on convex subanalytic cones.

\subsection{Definition}\label{sec:MM}

Now we are going to apply the Fourier-Sato transform to the multi-specialization. We refer to \cite{KS90} for the classical Fourier-Sato transform and to \cite{Pr11} for its generalization to subanalytic sheaves. First, we need a general result:
Let $\tau^{}_i \colon E^{}_i \to Z$ $(1\leq i \leq \ell)$ be vector bundles  over  $Z$,  and let $E^*_i$
be the dual   bundle of $E^{}_i$.
We denote by $ {\wedge^{}_i}$ and $ {\vee^{}_i}$ the Fourier-Sato  and  the inverse  Fourier-Sato
transformations on $E^{}_i$ respectively. Moreover we denote   by $ {\wedge^{*}_i}$ and $ {\vee^{*}_i}$ the Fourier-Sato  and  the inverse  Fourier-Sato transformations on $E^*_i$ respectively.
Recall that
\[
G^{\wedge^{}_i} =(G^{\vee^{}_i})^a \otimes  \omega^{\otimes -1}_{E^*_i/Z}\,.
\]
Here $\omega^{}_{E^*_i/Z}$ is the dualizing complex and $\omega^{\otimes -1}_{E^*_i/Z}$ its dual.
Set $E:=E_1 \underset{Z}{\times} \cdots \underset{Z}{\times} E^{}_\ell$ and $E^*:=E^*_1 \underset{Z}{\times} \cdots \underset{Z}{\times} E^{*}_\ell$
for short.
Let $\tau \colon E \to  Z$ be the the canonical projection.
Set
$P'_i:= \{(\eta,\xi)\in E^{}_i \underset{Z}{\times} E^*_i;\, \langle\eta,\xi \rangle \leqslant 0\}$.
Further set
\[
P':=P'_1 \underset{Z}{\times} \cdots \underset{Z}{\times} P'_\ell, \qquad
P^+:= E \underset{Z}{\times} E^*   \setminus P',
\] and denote by
$p'_1\colon P' \to E $,
$p_2' \colon P'\to E^*$, and
$p^+_1 \colon P^+\to E^{}_i $,  $p^+_2 \colon P^+\to E^{*} $
the canonical projections
respectively.
Let $F$ and  $G$ be  a multi-conic object on  $E$ and
$E^* $ respectively. Then we set for  short $ {\wedge_E}$ (resp. $ {\vee^*_E}$) the composition of the Fourier-Sato transforms $ {\wedge_i}$ (resp. the composition of the inverse Fourier-Sato transforms $ {\vee^*_i}$) on $E_i$ for each $i \in \{1,\dots,\ell\}$.

\begin{oss}\label{!and!!}
Let $X$, $Y$ be two real analytic manifolds, and $f\colon Y \to X$  a
real analytic mapping. We have a commutative diagram
\[
\xymatrix @R=3ex{Y \ar[d]^\rho \ar[r]^f & X \ar[d]^\rho \\
Y^{}_{\rm sa} \ar[r]^f & X^{}_{\rm sa}}
\]
 For subanalytic
sheaves we can also define the functor of proper direct image $f^{}_{!!}$, and its derived functor $Rf_{!!}$\,.
Note that $Rf^{}_{!!} \circ R\rho^{}_* \not\simeq R\rho^{}_* \circ
Rf^{}_!$ in general.
As in the notation above,
let  $i \colon Z \to E$ be the zero-section embedding.  Then  we
 have
\[R \tau^{}_{!!} \circ R\rho^{}_* =  i^!\circ R\rho^{}_* = R\rho^{}_* \circ  i^! = R\rho^{}_* \circ
R\tau^{}_!\,.
\]
Hence in what follows, we identify $R \tau^{}_{!!}$ with  $R \tau^{}_{!}$.
\end{oss}

\begin{prop}\label{propF}
Let $F$ and  $G$ be  multi-conic objects on  $E$ and
$E^* $ respectively.

$(1)$ $F^{\wedge_E}$ and $G^{\vee^*_E}$
are independent of the order of the  the Fourier-Sato transformations $ {\wedge^{}_i}$
and inverse the Fourier-Sato transformations $ {\vee^*_{i}}$ respectively.

$(2)$ It  follows that
\[
G^{\vee^*_E}
 = Rp'_{1*}p^{\prime \,!}_{2} G.
\]
\end{prop}
\begin{proof}
(1)
By induction on $\ell$,   we may assume that   $\ell =2$.
We have the following commutative diagram:
\[
\xymatrix @C=3em{
 &P' \ar @/_3ex/[ld]_-{p'_1} \ar@/^4ex/[rr]^{p'_2} \ar[r]_{ p^{\prime}_{2,2} }
\ar[d]^{p^{\prime}_{1,1}}
\ar@{}[dr] | {\displaystyle\square}
& P'_1 \ftimes_Z E^*_2
\ar[r]_{p^{\prime}_{1,2}}\ar[d]^{p^{\prime}_{1,1} }& E^*
\\
E& E^{}_1 \ftimes_Z P'_2 \ar[l]^{p ^{\prime}_{2,1}} \ar[r]_{ p ^{\prime}_{2,2}} &
E^{}_1 \ftimes_Z E^*_2 &
}\]
 Then we have
\begin{equation}
\label{eq.s.4}
\begin{split}
G^{\vee ^{*}_{1} \vee ^{*}_{2}} &=
R p^{\prime}_{2,1*}\, p ^{\prime\, !}_{2,2}\,
Rp ^{\prime}_{1,1*} \,p ^{\prime\,!}_{1,2}G
= R p^{\prime}_{2,1*}\, Rp ^{\prime}_{1,1*} \, p ^{\prime\, !}_{2,2}\,
p ^{\prime\,!}_{1,2} G
\\
&
 = Rp'_{1*}p^{\prime \,!}_{2} G.
\end{split}
\end{equation}
For the same reason, we obtain
\[
G^{\vee ^{*}_{1} \vee ^{*}_{2}} =
  Rp'_{1*}p^{\prime \,!}_{2} G
= G^{\vee ^{*}_{2} \vee ^{*}_{1}}.
\]
Therefore  we have
\[
\wedge^{*}_{1} \wedge^{*}_{2}= ( \vee ^{*}_{2}\vee ^{*}_{1})^{-1}= (\vee ^{*}_{1} \vee ^{*}_{2})^{-1}
= \wedge^{*}_{2} \wedge^{*}_{1}.
\]
Replacing $(E^*_1,E^*_2)$ with $(E^{}_1,E^{}_2)$,  we obtain
\[
\wedge^{}_{1} \wedge^{}_{2}= \wedge^{}_{2} \wedge^{}_{1}.
\]

(2) For  $\ell =2$,  the  result follows by \eqref{eq.s.4}.
Next, assume that $\ell >2$. Set $E':= \underset{Z, 2\leq i\leq\ell}{\times}E^{}_i$,
$E'{}^*:= \underset{Z, 2\leq i\leq\ell}{\times}E'{}^*_i$, $P'_{E'}:= \underset{Z, 2\leq i\leq\ell}{\times}P'_i$,
and $ {\vee^*_{E'}}$ the composition of $ {\vee^*_i}$) for each $i \in \{2,\dots,\ell\}$.
We have the following commutative diagram:
\[
\xymatrix @C=3em{
 &P' \ar @/_3ex/[ld]_-{p'_1} \ar@/^4ex/[rr]^{p'_2} \ar[r]_{ p^{\prime}_{E',2} } \ar[d]^{p^{\prime}_{1,1}}
\ar@{}[dr] | {\displaystyle\square}
& P'_1 \ftimes_Z E'{}^*
\ar[r]_{p^{\prime}_{1,2}}\ar[d]^{p^{\prime}_{1,1} }& E^*
\\
E& E^{}_1 \ftimes_Z P'_{E'} \ar[l]^{p ^{\prime}_{E',1}} \ar[r]_{ p ^{\prime}_{E',2}} &
E^{}_1 \ftimes_Z E'{}^*&
}\]
 Then by induction hypothesis,   we have
\[
(G^{\vee ^{*}_{1}}) ^{\vee ^{*}_{E'}} =
R p^{\prime}_{E',1*}\, p ^{\prime\, !}_{E',2}\,
Rp ^{\prime}_{1,1*} \,p ^{\prime\,!}_{1,2}G
= R p^{\prime}_{E',1*}\, Rp ^{\prime}_{1,1*} \, p ^{\prime\, !}_{E',2}\,
p ^{\prime\,!}_{1,2} G
 = Rp'_{1*}p^{\prime \,!}_{2} G.
\]
Therefore, the induction proceeds.
 \end{proof}
We shall need some notation. For a subset $K=\{i_1,\dots,i_k\} \subseteq \{1,\dots,\ell\}$,
set  $\chi_K:=\{M_i;\; i \in K\}$,
$
S_{i}:=T_{M_{i}}\iota(M_{i}) \underset{X}{\times} M$ ($j=1,\dots,\ell$) and
\[
S_{K}: =T_{M_{i_1}}\iota(M_{i_1}) \underset{X}{\times} \cdots \underset{X}{\times} T_{M_{i_k}}\iota(M_{i_k})
=S_{i_1}\ftimes_M \cdots \ftimes_M S_{i_k}\,.
\]
Let $S^*_{K}$ be the dual of $S_{K}$:
\[
S^*_{K}: =T^*_{M_{i_1}}\iota(M_{i_1}) \underset{X}{\times} \cdots \underset{X}{\times} T^*_{M_{i_k}}\iota(M_{i_k})
=S^*_{i_1}\ftimes_M \cdots \ftimes_M S^*_{i_k}\,.
\]
 Given $C_{i_j} \subseteq S_{i_j}$, $j=1,\dots,k$, we set for short $C_K:=C_{i_1} \underset{X}{\times} \cdots \underset{X}{\times} C_{i_k} \subset S_K$. Define $ {\wedge_K}$ as the composition of the Fourier-Sato transformation $ {\wedge_{i_k}}$ on $S_{i_k}$ for each $i_k \in K$.

Let $I,J \subseteq \{1,\dots,\ell\}$ be such that $I \sqcup J = \{1,\dots,\ell\}$. We still denote by $\pi \colon
S_I \underset{M}{\times} S_J^* \to M$ the projection. We define the functor $\nu^{\rm sa}_{\chi_I}\mu^{\rm sa}_{\chi_J}$ by
\[
\nu^{\rm sa}_{\chi_I}\mu^{\rm sa}_{\chi_J} \colon \bDb(k_{X_{\rm sa}})
\ni F  \mapsto  \nu^{\rm sa}_\chi (F)^{\wedge_J} \in  \bDb(k_{(S_I \ftimes_M S_J^*)_{\rm sa}}).
\]
By Proposition \ref{propF}, this is well defined; that is, this definition does not depend on
 the order of the Fourier-Sato transformations.
Composing with the functor $\rho^{-1}$, we set for short
\[
\nu_{\chi_I}\mu_{\chi_J}:= \rho^{-1} \nu^{\rm sa}_{\chi_I}\mu^{\rm sa}_{\chi_J}
\colon  \bDb(k_{X_{\rm sa}})  \to  \bDb(k_{S_I \ftimes_M S_J^*}).
\]
When $I=\emptyset$ we obtain the functor of multi-microlocalization: Set $ \wedge :=
 {\wedge_{\{1,\dots,\ell\}}}$ for short.
\begin{df} The multi-microlocalization along $\chi$ is the functor
\[
\mu^{\rm sa}_\chi \colon \bDb(k_{X_{\rm sa}}) \ni F  \mapsto  \nu^{\rm sa}_\chi(F)^\wedge \in  \bDb(k_{S^*_{\chi \rm sa}}).
\]
\end{df}
As above, we set for short
\[
\mu_\chi :=\imin \rho \mu^{\rm sa}_\chi \colon  \bDb(k_{X_{\rm sa}}) \to  \bDb(k_{S^*_{\chi}}).
\]

\subsection{Stalks}{\label{sec:fiber-formula}}

Let $X$ be a real analytic manifold and consider a family of submanifolds $\chi=\{M_1,\dots,M_\ell\}$ satisfying H1, H2 and H3.
Let $S=T_{M_1}\iota(M_1) \underset{X}{\times} \cdots \underset{X}{\times} T_{M_\ell}\iota(M_\ell)$. Locally $p \in S$ is given by $p=p_1 \times \cdots \times p_\ell = (q;\,\xi^{(1)},\dots,\xi^{(\ell)})$, with $\xi^{(k)} \in T_{M_k}\iota(M_k)$. Set $M=\bigcap_{j=1}^\ell M_j$.
Let $\tau_j:T_{M_j}\iota(M_j) \hookrightarrow T_{M_j}X$
denote the canonical injection and let
$\pi_j: S \to T_{M_j}\iota(M_j)$ be the canonical projection.

\begin{lem} \label{lem:sections A} Let $F \in D^b(k_{X_{sa}})$.
\begin{itemize}
\item[(i)] Let $A$ be a multi-conic closed subanalytic subset of $S$.
Then $\operatorname{H}^k_A(S;\nu^{sa}_\chi F) \simeq \lind {Z,U} \operatorname{H}^k_Z(U;F)$, where $U$ ranges through
the family of open subanalytic neighborhoods of $M$ and $Z$ is a closed subanalytic subset such that $C_\chi(Z) \subset A$.
\item[(ii)] Suppose that $A=\bigcap_{j=1}^\ell \imin {\pi_j}A_j$ with
$A_j$ being a closed conic subanalytic subset
in $T_{M_j}\iota(M_j) \underset{M_j}{\times} M$.
Then $\operatorname{H}^k_A(S;\nu^{sa}_\chi F) \simeq \lind {Z,U} \operatorname{H}^k_Z(U;F)$,
where $U$ ranges through the family of open subanalytic neighborhoods of $M$ and
$Z=Z_1 \cap \cdots \cap Z_\ell$ with each $Z_j$ being a closed subanalytic subset in $X$ and
$C_{M_j}(Z_j) \subset (T_{M_j}X \setminus \tau_j(T_{M_j}\iota(M_j) \setminus A_j))$.
\end{itemize}
\end{lem}

\begin{proof}  (i) We have the exact sequence
\[
\dots \to \operatorname{H}^k_A(S;\nu^{sa}_\chi(F)) \to \operatorname{H}^k(S;\nu^{sa}_\chi F) \to \operatorname{H}^k(S \setminus A;\nu^{sa}_\chi F) \to \dots
\]
We have $\operatorname{H}^k(S;\nu^{sa}_\chi F) \simeq \lind U \operatorname{H}^k(U;F)$, where $U$ ranges through the family of subanalytic neighborhoods of $M$. Moreover
\[
\operatorname{H}^k(S \setminus A;\nu^{sa}_\chi F) \simeq \lind W \operatorname{H}^k(W;F),
\]
where $W \in \op(X_{sa})$ is such that $C_\chi(X \setminus W) \cap (S \setminus A) = \emptyset$. Setting $Z=X \setminus W$ we obtain
\[
\operatorname{H}^k(S \setminus A;\nu^{sa}_\chi F) \simeq \lind {U,Z} \operatorname{H}^k(U \setminus Z;F),
\]
where $U$ ranges through the family of subanalytic neighborhoods of $M$ and $Z$ is closed subanalytic such that $C_\chi(Z) \subset A$. Then the result follows thanks to the five lemma applied to the exact sequence
\[
\dots \to \lind {U,Z} \operatorname{H}^k_Z(U;F) \to \lind U \operatorname{H}^k(U;F) \to \lind {U,Z} \operatorname{H}^k(U \setminus Z;F) \to \dots
\]
where $U$ ranges through the family of subanalytic neighborhoods of $M$ and $Z$ is closed subanalytic such that $C_\chi(Z) \subset A$.\\

(ii) Let $A=\bigcap_{j=1}^\ell \imin {\pi_j}A_j$ with $A_j \subset T_{M_j}\iota(M_j) \times_{M_j} M$.
Set for short $S_j:=T_{M_j}\iota(M_j) \times_{M_j} M$.
Then we have $S \setminus A = \bigcup_{j=1}^\ell \imin {\pi_j}(S_j \setminus A_j)$.
Let $W \in \op(X_{sa})$ be such that $C_\chi(X \setminus W) \cap (S \setminus A) = \emptyset$.
Then $W=\bigcup_{j=1}^\ell W_j$ with $C_\chi(X \setminus W_j) \cap \imin {\pi_j}(S_j \setminus A_j)=\emptyset$.
Let us find $W_1,\dots,W_\ell$. Let $\widetilde{W}_j$ be an open neighborhood of $\imin {\pi_j}(S_j \setminus A_j)$
in the multi-normal deformation $\widetilde{X}$ of $X$.
Then by Proposition  4.6 of \cite{HP} we have $C_\chi(X \setminus \widetilde{p}(\widetilde{W}_j \cap \Omega)) \cap \imin {\pi_j}(S_j \setminus A_j) = \emptyset$. Set $W_j=\widetilde{p}(\widetilde{W}_j \cap \Omega) \cap W$. Up to shrink $W$ we may suppose $W=\bigcup_{j=1}^\ell W_j$.  We have
\begin{eqnarray*}
C_\chi(X \setminus W_j) \cap \imin {\pi_j}(S_j \setminus A_j) = \emptyset  & \Leftrightarrow & C_{\chi}(X \setminus W_j) \cap (S_j \setminus A_j) = \emptyset \\
& \Leftrightarrow & C_{M_j}(X \setminus W_j) \cap \tau_j(S_j \setminus A_j) = \emptyset,
\end{eqnarray*}
where, for the second condition, $S_j \setminus A_j$ is  regarded as a subset of
\[
\{ (q;\,\xi^{(1)},\dots,\xi^{(\ell)}) \in S;\, q \in M,\, \xi^{(k)} = 0 \,(k \ne j) \} \subset S.
\]
And
the first equivalence follows from Lemma 4.2 of \cite{HP} and
the second one from Corollary 4.3 of \cite{HP}.
Setting $Z=X \setminus W$ and $Z_j = X \setminus W_j$,
we obtain $Z=\bigcap_{j=1}^\ell Z_j$ with $C_{M_j}(Z_j) \cap \tau_j(S_j \setminus A_j) = \emptyset$ and the result follows.
\end{proof}
\begin{oss}
We correct typo:
By Lemma 6.1 in \cite{HP},  for any  $F \in \bD^{\rm b}(\Bbbk_{X_{\rm sa}})$ there is a natural isomorphism
\[
\nu^{sa}_\chi (F)= s^{-1}R\varGamma_{\varOmega_\chi} (p^{-1} F) \simeq s^!(p^{!} F)_{\varOmega_\chi}\,.
\]
\end{oss}
\begin{prop}
Assume that  $\chi$ satisfies the conditions \textup{H1}, \textup{H2} and \textup{H3}.
Let $\tau \colon S_\chi \simeq \underset{X,1\leq j\leq \ell}{\times}T_{M_j}\iota(M_j) \to  M$ be the the canonical projection,
 where $M:= \bigcap_{j=1}^\ell M_j$\,. Then
\begin{align}
 \nu^{sa}_\chi (F) |_M &\simeq R\tau_{*}\nu^{sa}_\chi (F) \simeq F|_M\,,
\\
 R\varGamma_M(\nu^{sa}_\chi (F)) &\simeq
R\tau_{!} \nu^{sa}_\chi (F) \simeq R\varGamma_M(F).
\end{align}
\end{prop}
\begin{proof}
Let $k \colon M \to S_\chi$ be the zero-section embedding, and $i \colon M \to X$ the canonical embedding. Then
\allowdisplaybreaks
\begin{align*}
 F|_M &= k^{-1}s^{-1}p^{-1}F \to k^{-1}s^{-1}Ri_{\Omega_\chi *}i_{\Omega_\chi}^{\,-1}p^{-1}(F)
\\
&  \simeq
k^{-1} s^{-1}R\varGamma_{\Omega_\chi}(p^{-1}F) \simeq \nu^{sa}_\chi (F) |_M \,,
\\
R\tau_{!} \nu^{sa}_\chi (F) &= k^{!} \nu^{sa}_\chi (F)\simeq
k^{!}   s^!(p^{!} F)_{\varOmega_\chi} \to k^{!}  s^! p^{!}F = R\varGamma_M(F).
\end{align*}
These morphisms are isomorphisms by the stalk formulae.
\end{proof}

Set $S^* := \proddTM$.
Let $V=V_1 \underset{X}{\times} \dots \underset{X}{\times} V_\ell$
be a multi-conic open subanalytic subset in $S^*$, and let $\pi: S^* \to M$ denote the canonical projection.
We set, for short, $V^\circ:=V_1^\circ \underset{X}{\times} \dots \underset{X}{\times} V_\ell^\circ$ the multi-polar cone in $S$.

\begin{lem} \label{lem:sections Vcirc}
	Let $V=V_1 \underset{X}{\times} \dots \underset{X}{\times} V_\ell$ be
a multi-conic open subanalytic subset in $S^*$ such that
$V \cap \pi^{-1}(q)$ is convex in $S^*_q$ for $q \in \pi(V)$.
Then
$\operatorname{H}^k(V; \mu^{sa}_\chi F) \simeq \lind {Z,\,U} \operatorname{H}^k_Z(U;F)$,
where $U$ ranges through the family of open subanalytic subsets in $X$ with $U \cap M = \pi(V)$
and $Z=Z_1 \cap \cdots \cap Z_\ell$ with $C_{M_j}(Z_j) \subset (T_{M_j}X \setminus \tau_j({\rm Int}(V_j^{\circ a})))$.
Here $(\cdot)^a$ denotes the inverse image of the antipodal map.
\end{lem}
\begin{proof}  Since the fiber of $V_j$ is convex for each $j$ the Fourier-Sato transform gives $\operatorname{H}^k(V;\mu^{sa}_\chi F) \simeq \operatorname{H}^k_{V^\circ}(S;\nu^{sa}_\chi F)$.
Then the result follows from Lemma \ref{lem:sections A} (ii).
\end{proof}

Let
$
p=p_1 \times \dots \times p_\ell = (q;\, \xi^{(1)},\dots,\xi^{(\ell)}) \in S^*
$.
For any $p_k \in T^*_{M_k}\iota(M_k)$, we define the subset in $(T_{M_k}X)_q$
\begin{equation}{\label{eq:def-point_polar_cone}}
p_k^{\#} := (T_{M_k}X)_q \setminus \tau_k(p_k^{\circ a}).
\end{equation}
Here the subset $p_k^{\circ a}$ in $(T_{M_k}\iota(M_k))_q$ denotes the antipodal
polar set of the point $p_k$, i.e.,
$
p_k^{\circ a} = \{\eta \in (T_{M_k}\iota(M_k))_q;\, \langle\eta, \xi^{(k)} \rangle \le 0\}
$.
Note that $p_k^{\#}$ is an open subset. Set for short $\mu_\chi:=\imin \rho \mu^{sa}_\chi$. As a consequence of Lemma \ref{lem:sections Vcirc} we have

\begin{cor}\label{cor:fibers} Let
$p=p_1 \times \dots \times p_\ell = (q;\, \xi^{(1)},\dots,\xi^{(\ell)}) \in S^*$,
and
let $F \in D^b(k_{X_{sa}})$. Then $H^k(\mu_\chi F)_p \simeq \lind {Z,U} \operatorname{H}^k_Z(U;F)$,
where $U \in \op(X_{sa})$ ranges through the family of open subanalytic neighborhoods of $q$
and $Z$ runs through a family of closed sets in the form
$Z_1 \cap Z_2 \cap \dots \cap Z_{\ell}$
with each $Z_k$ being closed subanalytic in $X$ and
$ C_{M_k}(Z_k)_q \subset p_k^{\#} \cup \{0\} $
 $(k=1,2,\dots,\ell)$.
\end{cor}

Now we are going to find a stalk formula for multi-microlocalization given by a limit of sections with support (locally) contained in closed convex cones. As the problem is local,
we may assume that $X=\mathbb{R}^n$ and $q=0$ with coordinates $(x_1,\dots,x_n)$,
and that there exists a subset $I_k$
$(k=1,2,\dots,\ell)$
in $\{1,2,\dots,n\}$
with the conditions (\ref{eq:cond-H}) such that each submanifold
$M_k$ is given by
$
\left\{x = (x_1,\dots,x_n) \in \mathbb{R}^n;\, x_i=0\, (i \in I_k)\right\}.
$
Recall that $\hat{I}_k$ was defined by (\ref{eq:def-hat-I}) and that we set
$M = \cap_k M_k$ and $n_k=\sharp\hat{I}_k$. Then locally we have
\[
X = M \times \left(N_1 \times N_2 \times \dots \times N_\ell\right) = M \times N,
\]
where $N_k$ is ${\mathbb R}^{n_k}$ with coordinates $x^{(k)} = (x_i)_{i \in \hat{I}_k}$.
Set, for $k \in \{1,\dots,\ell\}$,
\begin{equation}{\label{eq:def-prec}}
\begin{aligned}
J_{\prec k}&:=\{j \in \{1,\ldots,\ell\},\;I_j \subsetneqq I_k\},\\
J_{\succ k}&:=\{j \in \{1,\ldots,\ell\},\;I_j \supsetneqq I_k\}, \\
J_{\nparallel\, k}&:=\{j \in \{1,\ldots,\ell\},\;I_j \cap I_k = \emptyset\}.
\end{aligned}
\end{equation}
Clearly we have
\begin{equation}
k \in J_{\prec j} \Leftrightarrow I_k \subsetneqq I_j \Leftrightarrow j \in J_{\succ k},
\end{equation}
and, by the conditions H1, H2 and H3, we also have
\begin{equation}
J_{\prec k} \sqcup \{k\} \sqcup J_{\succ k} \sqcup
J_{\nparallel\, k}
= \{1,2,\dots,\ell\}.
\end{equation}
Let $p=p_1 \times \dots \times p_\ell =
(q;\, \xi^{(1)},\dots,\xi^{(\ell)}) \in \proddTM$
and consider the following conic subset in $N$
\begin{equation} \label{gamma k}
\gamma_k :=
\left\{ (x^{(j)})_{j=1,\dots,\ell} \in N;\,
\begin{array}{ll}
	x^{(j)}=0, & j \in J_{\prec k} \sqcup J_{\nparallel\, k} \\
	x^{(j)} \in  \mathbb{R}^{n_j},& j \in J_{\succ k} \\
	\langle x^{(j)}, \xi^{(k)} \rangle > 0, & j=k
\end{array}
\right\}.
\end{equation}
Note that, if $\xi^{(k)} = 0$, then $\gamma_k$ is empty.

\begin{es} We now compute $\gamma_k$ of \eqref{gamma k} on the complex case in the following two typical situations. Let $X={\mathbb C}^2$
with coordinates $(z_1, z_2)$.
\begin{enumerate}
\item (Majima)  Let $M_1 = \{z_1 = 0\}$ and $M_2 = \{z_2 = 0\}$.
Then
\begin{eqnarray*}
\gamma_1 & = & \{(z_1,0);\,\operatorname{Re}\langle z_1,\eta_1 \rangle >0\}, \\
\gamma_2 & = & \{(0,z_2);\,\operatorname{Re}\langle z_2,\eta_2 \rangle >0\}.
\end{eqnarray*}
\item (Takeuchi) Let $M_1 = \{0\}$ and $M_2 = \{z_2 = 0\}$.
Then
\begin{eqnarray*}
\gamma_1 & = & \{(z_1,0);\,\operatorname{Re}\langle z_1,\eta_1 \rangle >0\}, \\
\gamma_2 & = & \{(z_1,z_2);\,\operatorname{Re}\langle z_2,\eta_2 \rangle >0\}.
\end{eqnarray*}
\end{enumerate}
\end{es}

\begin{es} We now compute $\gamma_k$ of \eqref{gamma k} on the real case in three typical situations. Let $X={\mathbb R}^3$ with coordinates $(x_1, x_2, x_3)$.
\begin{enumerate}
\item (Majima) Let $M_1 = \{x_1 = 0\}$, $M_2 = \{x_2 = 0\}$ and $M_3 = \{x_3 = 0\}$.
Then
\begin{eqnarray*}
\gamma_1 & = & \{(x_1,0,0);\,\langle x_1,\xi_1 \rangle >0\}, \\
\gamma_2 & = & \{(0,x_2,0);\,\langle x_2,\xi_2 \rangle >0\}, \\
\gamma_3 & = & \{(0,0,x_3);\,\langle x_3,\xi_3 \rangle >0\}.
\end{eqnarray*}
\item (Takeuchi)
Let $M_1 = \{0\}$, $M_2 = \{x_2 = x_3 = 0\}$ and $M_3 = \{x_3 = 0\}$.
Then
\begin{eqnarray*}
\gamma_1 & = & \{(x_1,0,0);\,\langle x_1,\xi_1 \rangle >0\}, \\
\gamma_2 & = & \{(x_1,x_2,0);\,\langle x_2,\xi_2 \rangle >0\}, \\
\gamma_3 & = & \{(x_1,x_2,x_3);\,\langle x_3,\xi_3 \rangle >0\}.
\end{eqnarray*}
\item (Mixed)
Let $M_1 = \{0\}$, $M_2 = \{x_2 = 0\}$ and $M_3 = \{x_3 = 0\}$.
Then
\begin{eqnarray*}
\gamma_1 & = & \{(x_1,0,0);\,\langle x_1,\xi_1 \rangle >0\}, \\
\gamma_2 & = & \{(x_1,x_2,0);\,\langle x_2,\xi_2 \rangle >0\}, \\
\gamma_3 & = & \{(x_1,0,x_3);\,\langle x_3,\xi_3 \rangle >0\}.
\end{eqnarray*}
\end{enumerate}
\end{es}

\begin{teo} \label{th:stalk formula}
Let $p=p_1 \times \dots \times p_\ell = (q;\, \xi^{(1)},\dots,\xi^{(\ell)}) \in S^*$,
and let $F \in D^b(k_{X_{sa}})$. Then we have
\begin{equation}\label{eq:th-stalk-formula}
H^k(\mu_\chi F)_p \simeq \lind {G,U} \operatorname{H}^k_G(U;F).
\end{equation}
Here $U$ is an open subanalytic neighborhood of $q$ in $X$ and
$G$ is a closed subanalytic subset in the form
$M \times \left(\displaystyle\sum_{k=1}^\ell G_k \right)$
with $G_k$ being a closed subanalytic convex cone in $N$
satisfying $G_k \setminus\{0\} \subset \gamma_k$, where $\gamma_k$ is defined in \eqref{gamma k}.
\end{teo}

\begin{proof}  As a point of the base manifold $M$
is irrelevant in the subsequent arguments,
we may assume $M = \{0\}$ for simplicity.
We also assume $q=0$ and $|\xi^{(k)}| \le 1$ for $k=1,2,\dots,\ell$.

We first prove that, for any $Z=Z_1 \cap \cdots \cap Z_\ell$ with $Z_k$ being closed in $X$ and
$C_{M_k}(Z_k)_q \subset p_k^{\#} \cup \{0\}$ ($k=1,2,\dots,\ell$),
there exists $G$ described in the theorem with $G \supset Z$.
As $G_k$ is convex and it contains the origin, we have
\[
G_1^\circ \cap \dots \cap G_{\ell}^\circ
= \left(G_1 + \dots + G_{\ell}\right)^{\circ}
\]
and
\[
\left(G_1^\circ \cap \dots \cap G_{\ell}^\circ\right)^\circ
= \overline{G_1 + \dots + G_{\ell}}
\]
where $G_k^\circ$ designates the usual polar cone of $G_k$ in $X$ as a vector space.
We shall find a closed convex cone $V$ such that $V=V_1 \cap \cdots \cap V_\ell$
with $V_k$ convex for each $k=1,\dots,\ell$ and
$V_k^\circ \setminus \{0\} \subset \gamma_k$ satisfying $V^\circ \supseteq Z$.
Furthermore
we choose $V_k$ so that every $V_k^\circ$ is proper with respect to
the same direction $\widetilde{\xi} \ne 0$,
i.e., $V^\circ \setminus \{0\} \subset \{x \in X;\,
\langle x, \widetilde{\xi} \rangle > 0\}$.
In this way
\[
V^\circ=(V_1^{\circ\circ} \cap \cdots \cap V_\ell^{\circ\circ})^\circ=
\overline{V_1^\circ+\cdots+V_\ell^\circ}
= V_1^\circ+\cdots+V_\ell^\circ
\]
and setting $G_k:=V_k^\circ$ we obtain the claim.
Here the last equality follows from the fact that every $V^\circ$ is closed
and properly contained in the same half space in $X$.

It follows from the definition of $Z_k$ that there exists
$\epsilon > 0$ and a closed convex cone $\Gamma_k \subset N_k$
with $\Gamma_k\setminus \{0\} \subset \{\langle x^{(k)},\, \xi^{(k)} \rangle > 0\}$
which satisfies
\[
Z_k \subset \{x \in X;\, x^{(k)} \in \Gamma_k\} \cup
\{x \in X;\, \epsilon|x^{(k)}| \le \sum_{j \in J_{\prec k}}|x^{(j)}|\}.
\]
Note that, for $k$ with $\xi^{(k)}=0$, we always take $\Gamma_k = \{0\}$.
The existence of such an $\epsilon$ and a $\Gamma_k$ is shown in the following way.
We set
\[
N' := \underset{j \in J_{\prec k}}{\times} N_j,\qquad
N'' := \underset{j \in J_{\succ k} \cup J_{\nparallel\, k}}{\times} N_j,
\]
for which we have $X = N' \times N_k \times N''$
with coordinates $(x',\, x^{(k)}, x'')$. Note that
$M_k = \{0\}_{N' \times N_k} \times N''$ holds.
We also define a closed subset
$D$ in $N_k$ by $\{x^{(k)} \in N_k;\, \langle x^{(k)},\, \xi^{(k)} \rangle \le 0\}$.
Then $C_{M_k}(Z_k)_q \subset p_k^{\#} \cup \{0\}$ implies that,
for any $\theta \in D \setminus \{0\}$, there exist an open cone
$Q_\theta$ in  $N' \times N_k$ with direction
$(0_{N'},\, \theta)$ and
an open neighborhood $U_\theta$ of $q$ in $X$ satisfying
\[
(Q_\theta \times N'') \cap U_\theta \subset X \setminus Z_k.
\]
Then, as $\{0\}_{N'} \times D$ is a closed conic subset in $N' \times N_k$,
we can find a finite subset $\Theta$ in $D \setminus \{0\}$ such that
\[
\begin{aligned}
	&\{0\}_{N'} \times (D\setminus \{0\}) \subset
	\underset{\theta \in \Theta}{\bigcup} Q_\theta, \\
	&\left(\left(\underset{\theta \in \Theta}{\bigcup} Q_\theta\right) \times N''\right)
	 \cap \left(\underset{\theta \in \Theta}{\bigcap} U_\theta\right)
\subset X \setminus Z_k.
\end{aligned}
\]
Then, by taking $U$ in
(\ref{eq:th-stalk-formula})
sufficiently small so that $U \subset \displaystyle\underset{\theta \in \Theta}{\bigcap} U_\theta$,
we may assume, from the beginning,
\[
\left(\underset{\theta \in \Theta}{\bigcup} Q_\theta\right) \times N''
\subset X \setminus Z_k.
\]
As $\displaystyle\underset{\theta \in \Theta}{\bigcup} Q_\theta$ is an open conic neighborhood
of $\{0\}_{N'} \times (D\setminus \{0\})$ in $N' \times N_k$, there exist
an open cone $T$ in $N_k$ with $D \setminus \{0\} \subset T$ and $\epsilon > 0$
satisfying
\[
\left\{
(x',\, x^{(k)}) \in N' \times N_k;\,
x^{(k)} \in T,\, \sum_{j \in J_{\prec k}} |x^{(j)}| < \epsilon |x^{(k)}|\right\}
\subset \underset{\theta \in \Theta}{\cup} Q_\theta.
\]
Hence we have
\[
\left\{
(x',\, x^{(k)}) \in N' \times N_k;\,
x^{(k)} \in T,\, \sum_{j \in J_{\prec k}} |x^{(j)}| < \epsilon |x^{(k)}|\right\}
\times N'' \subset X \setminus Z_k,
\]
which is equivalent to saying that
\[
\left\{x \in X;\, x^{(k)} \in (N_k \setminus T)\right\} \cup
\left\{x \in X;\, \sum_{j \in J_{\prec k}} |x^{(j)}| \ge \epsilon |x^{(k)}|\right\}
\supset Z_k.
\]
This shows the existence of $\epsilon > 0$ and $\Gamma_k := N_k \setminus T$.

Now we set
\[
\begin{aligned}
Z_{k,\,\Gamma} &:= \left\{x \in X;\, x^{(k)} \in \Gamma_k\right\} \\
Z_{k,\,\epsilon} &:= \left\{x \in X;\, \epsilon|x^{(k)}| \le \sum_{j \in J_{\prec k}}|x^{(j)}|\right\}.
\end{aligned}
\]
Note that, for $k \in \{1,2,\dots,\ell\}$ with $\hat{I}_k=I_k$, we have
$Z_k \subset Z_{k,\,\Gamma}$ and no $Z_{k,\,\epsilon}$ appears.

Define $V=V_1 \cap \cdots \cap V_\ell$. Here each $V_k$ is given by, if $\xi^{(k)} \ne 0$,
\[
\left\{x\in X;\, x^{(k)} \in T_k,\,
\delta |\langle x^{(k)}, \xi^{(k)} \rangle| \ge \sum_{j \in J_{\succ k}}|x^{(j)}|\right\}
\]
where $\delta > 0$ and $T_k$ is a proper closed convex cone in $N_k$ with
$T_k \subset \{x^{(k)} \in N_k;\, \langle x^{(k)},\, \xi^{(k)} \rangle \ge 0\}$ and
$\xi^{(k)} \in \operatorname{Int}_{N_k} T_k$.
And if $\xi^{(k)} = 0$, then $V_k$ is the whole $X$.
Note that $V_k$ is a convex set in any case, i.e., $V_k^{\circ\circ} = V_k$.
Then such a $V$ satisfies the desired properties. Indeed it is easy to see
$V_k^\circ \setminus \{0\} \subset \gamma_k$. We will show that
\[
V^\circ \supseteq \bigcap_k (Z_{k,\,\Gamma} \cup Z_{k,\,\epsilon}) \supseteq Z.
\]
Here we emphasize that the inequality appearing in $Z_{k,\epsilon}$
\begin{equation}{\label{eq:xxx-0}}
	\epsilon|x^{(k)}| \le \sum_{j \in J_{\prec k}}|x^{(j)}|
\end{equation}
and that in $V_k$ for $k$ with $\xi^{(k)} \ne 0$
\begin{equation}{\label{eq:xxx-1}}
	\delta |\langle x^{(k)},\, \xi^{(k)} \rangle| \ge \sum_{j \in J_{\succ k}}|x^{(j)}|
\end{equation}
play an important role below.

Let $\sigma = \sigma_1 \sigma_2 \dots \sigma_{\ell}$ be an $\ell$-length sequence where
$\sigma_k$ is either the symbols $\Gamma$ or $\epsilon$ $(k=1,2,\dots,\ell)$, and let us define
\[
K_\sigma := Z_{1,\,\sigma_1} \cap Z_{2,\, \sigma_2}
\cap \dots \cap Z_{\ell,\, \sigma_{\ell}}.
\]
Then we have
\[
\underset{k}{\bigcap}(Z_{k,\,\Gamma} \cup Z_{k,\,\epsilon})
=
\underset{\sigma}\cup K_\sigma.
\]
We now show that, for each sequence $\sigma$, we obtain $V^\circ \supset K_\sigma$
if we take $T_k$ ($k=1,2,\dots,\ell$) and $\delta > 0$ sufficiently small.

Set
\[
J_\Gamma(\sigma) := \{j \in \{1,2,\dots,\ell\};\, \sigma_j = \Gamma\}
\]
and
\[
J_\epsilon(\sigma) := \{j \in \{1,2,\dots,\ell\};\, \sigma_j = \epsilon\}.
\]
Note that, for $k \in \{1,2,\dots,\ell\}$ with $\hat{I}_k = I_k$, we have
$k \in J_\Gamma(\sigma)$ and $k \notin J_\epsilon(\sigma)$, which implies,
in particular, $J_\Gamma(\sigma)$ is non-empty.
As both $V_k$ and $\Gamma_k$ are proper cones with its direction $\xi^{(k)}$ in $N_k$
if $\xi^{(k)} \ne 0$, and as $\Gamma_k = \{0\}$ if $\xi^{(k)} = 0$,
there exists a constant $M > 0$ such that
\[
M|x^{(j)}||y^{(j)}|\, \le\quad \langle x^{(j)},\, y^{(j)} \rangle \quad \le\, |x^{(j)}||y^{(j)}|
\]
holds for $j \in J_\Gamma(\sigma)$ and
$x = (x^{(1)},\dots,x^{(\ell)}) \in K_\sigma$ and
$y = (y^{(1)},\dots,y^{(\ell)}) \in V$. Furthermore, by (\ref{eq:xxx-0}),
there exists a constant $N>0$ such that, for any $j \in J_\epsilon(\sigma)$,
we have
\[
|x^{(j)}| \le \frac{N}{\epsilon^N} \sum_{\alpha \in J_\Gamma(\sigma) \cap J_{\prec j}} |x^{(\alpha)}|
\]
for $x = (x^{(1)},\dots,x^{(\ell)}) \in K_\sigma$.
By noticing these facts, we obtain, for $x = (x^{(1)},\dots,x^{(\ell)}) \in K_\sigma$ and $y = (y^{(1)},\dots,y^{(\ell)}) \in V$.
\[
\begin{aligned}
	\langle x,y\rangle  &= \sum_{j} \langle x^{(j)},y^{(j)}\rangle  \quad =
	\sum_{j \in J_\Gamma(\sigma)} \langle x^{(j)},y^{(j)}\rangle  +
	\sum_{j \in J_\epsilon(\sigma)} \langle x^{(j)},y^{(j)}\rangle  \\
	&\ge M \sum_{j \in J_\Gamma(\sigma)} |x^{(j)}||y^{(j)}| -
	      \sum_{j \in J_\epsilon(\sigma)} |x^{(j)}||y^{(j)}| \\
	&\ge M \sum_{j \in J_\Gamma(\sigma)} |x^{(j)}||y^{(j)}| -
	\frac{N}{\epsilon^N}
	      \sum_{j \in J_\epsilon(\sigma)}
	      \left(\sum_{\alpha \in J_\Gamma(\sigma) \cap J_{\prec j}} |x^{(\alpha)}|\right)|y^{(j)}|.
\end{aligned}
\]
Here, as $\Gamma_\alpha = \{0\}$ for $\alpha$ with $\xi^{(\alpha)} = 0$,
we have $|x^{(\alpha)}| = 0$ for such an $\alpha \in J_\Gamma(\sigma)$
and the last term in the above inequalities is equal to
\begin{equation}{\label{eq:last-estimate-xxx}}
M \sum_{j \in J_\Gamma(\sigma)} |x^{(j)}||y^{(j)}| -
\frac{N}{\epsilon^N}
      \sum_{j \in J_\epsilon(\sigma)}
\left(\sum_{\alpha \in J_\Gamma(\sigma) \cap J_{\prec j},\, \xi^{(\alpha)} \ne 0}
|x^{(\alpha)}|\right)|y^{(j)}|.
\end{equation}
It follows from  (\ref{eq:xxx-1}) that we get
$\delta|y^{(\alpha)}| \ge \delta |\langle y^{(\alpha)},\, \xi^{(\alpha)} \rangle |
\ge |y^{(j)}|$ for $\alpha$ with $\xi^{(\alpha)} \ne 0$ and for $j \in J_{\succ \alpha}$
($\Leftrightarrow \alpha \in J_{\prec j}$). Hence
the (\ref{eq:last-estimate-xxx}) is lower bounded by
\[
\begin{aligned}
	& M \sum_{j \in J_\Gamma(\sigma)} |x^{(j)}||y^{(j)}| -
	\frac{\delta N}{\epsilon^N}
	      \sum_{j \in J_\epsilon(\sigma)}
	      \sum_{\alpha \in J_\Gamma(\sigma) \cap J_{\prec j},\, \xi^{(\alpha)} \ne 0}
	      |x^{(\alpha)}||y^{(\alpha)}| \\
	 & \ge M \sum_{j \in J_\Gamma(\sigma)} |x^{(j)}||y^{(j)}| -
	\frac{\delta N \#J_\epsilon(\sigma)}{\epsilon^N}
              \sum_{\alpha \in J_\Gamma(\sigma)} |x^{(\alpha)}||y^{(\alpha)}| \\
         & = \left(M - \frac{\delta N \#J_\epsilon(\sigma)}{\epsilon^N}\right)
              \sum_{\alpha \in J_\Gamma(\sigma)} |x^{(\alpha)}||y^{(\alpha)}|.
\end{aligned}
\]
Note that the set $J_\Gamma(\sigma)$ is non-empty as noted above.
Hence, if we take $\delta$ sufficiently small, $\langle x,y \rangle$ takes non-negative values
for $x \in K_\sigma$ and $y \in V$,
which implies $K_\sigma \subset V^\circ$.
Hence we have shown the existence of $G$ described in the theorem with $Z \subset G$.

Now we show that $G$ described in the theorem satisfies
$C_{M_k}(G)_q \subset p^{\#}_k \cup \{0\}$ for any $k$.
We may assume $G = V^\circ$ where $V$ was defined in the first part of the proof.
Suppose that there exists a non-zero vector $\eta \in (T_{M_k}X)_q = N' \times N_k$
such that
\[
0 \ne \eta \in C_{M_k}(V^\circ)_q \cap \tau_k(p^{\circ a}_k) .
\]
Note that $\eta \in \tau_k(p^{\circ a}_k)$ implies the existence of $0 \ne \eta^{(k)} \in N_k$ such that
$\eta = (0_{N'},\, \eta^{(k)})$ and $\langle \eta^{(k)},\, \xi^{(k)} \rangle \le 0$. Set, for any $\epsilon > 0$,
\[
Q_\epsilon := \left\{x = (x',x^{(k)},x'') \in X;\,
\begin{aligned}
&|x'| < \epsilon \langle \eta^{(k)},\, x^{(k)}\rangle,\, |x''| < \epsilon \\
&|x^{(k)}| < \epsilon,\, x^{(k)} \in Q^{(k)}_\epsilon
\end{aligned}
\right\},
\]
where $\{Q^{(k)}_\epsilon\}_{\epsilon>0}$ is a family of open cone neighborhoods of
the direction $\eta^{(k)}$ in $N_k$.
Then $\eta \in C_{M_k}(V^\circ)_q$ means $Q_\epsilon \cap V^\circ \ne \emptyset$ for
any $\epsilon > 0$.  By noticing $\langle \eta^{(k)},\, \xi^{(k)} \rangle \le 0$, it follows from
the definition of $V$ that there exists a vector $v = (v', v^{(k)}, 0_{N''}) \in V$ such that
we can find a positive constant $C > 0$ with
\[
\langle x^{(k)},\, v^{(k)} \rangle < - C |x^{(k)}| \quad (x^{(k)} \in Q^{(k)}_\epsilon)
\]
for any sufficiently small $\epsilon > 0$. Hence we have, for
$x = (x', x^{(k)}, x'') \in Q_\epsilon$,
\[
\langle x, v \rangle
=
\langle x', v' \rangle +
\langle x^{(k)}, v^{(k)} \rangle
\le (\epsilon |\eta^{(k)}| |v'| - C) |x^{(k)}|.
\]
As a result,  if we take a sufficiently small $\epsilon > 0$, we have $\langle x, v \rangle < 0$ for any
$x \in Q_\epsilon$, and thus, we get $Q_\epsilon \cap V^\circ = \emptyset$ which contradicts
$Q_\epsilon \cap V^\circ \ne \emptyset$. Therefore we have obtained the conclusion.
This completes the proof.
\end{proof}

\begin{oss} In the case $\ell=2$ with $M_1 \subset M_2 \subset X$ we obtain the stalk formula computed in \cite{Ta96}.
\end{oss}

Now let us consider the mixed cases between specialization and microlocalization. We shall need some notations.
Given a subset $K=\{i_1,\dots,i_k\} \subseteq \{1,\dots,\ell\}$,
let $\chi_K=\{M_i,\; i \in K\}$, set $S_{i_j}=T_{M_{i_j}}\iota(M_{i_j}) \underset{M_{i_j}}{\times} M$ ($j=1,\dots,k$)
and $S_K=T_{M_{i_1}}\iota(M_{i_1}) \underset{X}{\times} \cdots \underset{X}{\times} T_{M_{i_k}}\iota(M_{i_k})$ and let $S^*_K$ be its dual. Given $C_{i_j} \subseteq S_{i_j}$, $j=1,\dots,k$, we set for short $C_K:=C_{i_1} \underset{X}{\times} \cdots \underset{X}{\times} C_{i_k} \subset S_K$. Define $ {\wedge_K}$ as the composition of the Fourier-Sato transforms $ {\wedge_{i_k}}$ on $S_{i_k}$ for each $i_k \in K$.

Let $I,J \subseteq \{1,\dots,\ell\}$ be such that $I \sqcup J = \{1,\dots,\ell\}$.

\begin{lem} \label{lem:mixed sections Vcirc}
Let $V=V_I \underset{X}{\times} V_J$ be
a multi-conic open subanalytic subset in $S_I \underset{X}{\times} S_J^*$ such that
$V \cap \pi^{-1}(q)$ is convex for $q \in \pi(V)$.
Then
$\operatorname{H}^k(V; \nu^{sa}_{\chi_I}\mu^{sa}_{\chi_J} F) \simeq \lind {Z,\,U} \operatorname{H}^k_Z(U;F)$,
where $U$ ranges through the family of open subanalytic subsets in $X$ with $C_{\chi_I}(X \setminus U) \cap \bigcup_{i \in I} \imin {\pi_i}(V_i) = \emptyset$
and $Z=\bigcap_{j \in J}Z_j$ with $C_{M_j}(Z_j) \subset (T_{M_j}X \setminus \tau_j({\rm Int}(V_j^{\circ a})))$.
Here $(\cdot)^a$ denotes the inverse image of the antipodal map.
\end{lem}
\begin{proof}  We write $\times$ instead of $\underset{X}{\times}$ for short. Since $V$ is convex the Fourier-Sato transform gives $\operatorname{H}^k(V;\nu^{sa}_{\chi_I}\mu^{sa}_{\chi_J} F) \simeq \operatorname{H}^k_{V_I \times V_J^\circ}(S;\nu^{sa}_\chi F)$.
Consider the distinguished triangle
\begin{equation}\label{eq:mixed dt}
\dt{R\varGamma_{(S_I \setminus V_I) \times V_J^\circ}\nu^{sa}_\chi F}{R\varGamma_{S_I \times V_J^\circ}\nu^{sa}_\chi F}{R\varGamma_{V_I \times V_J^\circ}\nu^{sa}_\chi F}
\end{equation}
By Lemma \ref{lem:sections Vcirc} we have $\operatorname{H}^k_{S_I \times V_J^\circ}(S;\nu^{sa}_\chi F) \simeq \lind {Z,\,U} \operatorname{H}^k_Z(U;F)$,
where $U$ ranges through the family of open subanalytic subsets of $X$ such that $U \cap M=\pi(V)$
and $Z=\bigcap_{j \in J}Z_j$ with $C_{M_j}(Z_j) \subset (T_{M_j}X \setminus \tau_j({\rm Int}(V_j^{\circ a})))$. By Lemma \ref{lem:sections A} we also have $\operatorname{H}^k_{(S_I \setminus V_I) \times V_J^\circ}(S;\nu^{sa}_\chi F) \simeq \lind {Z,\,U} \operatorname{H}^k_Z(U;F)$,
where $U$ ranges through the family of open subanalytic subsets of $X$ such that $U \cap M=\pi(V)$
and $Z=\bigcap_{j = 1}^\ell Z_j$ with $C_{M_j}(Z_j) \subset (T_{M_j}X \setminus \tau_j({\rm Int}(V_j^{\circ a})))$ if $j \in J$ and $C_{M_j}(Z_j) \subset (T_{M_j}X \setminus \tau_j(V_j))$ if $j \in I$. Thanks to the long exact sequence associated to \eqref{eq:mixed dt} we obtain $\operatorname{H}^k_{V_I \times V_J^\circ}(S;\nu^{sa}_\chi F) \simeq \lind {Z,\,U} \operatorname{H}^k_Z(U \cap W;F)$,
where $U$ ranges through the family of open subanalytic subsets of $X$ such that $U \cap M=\pi(V)$, $Z=\bigcap_{j \in J} Z_j$ with $C_{M_j}(Z_j) \subset (T_{M_j}X \setminus \tau_j({\rm Int}(V_j^{\circ a})))$ and $W= \bigcup_{i \in I}(X \setminus Z_i)$ such that $C_{M_i}(Z_i) \subset (T_{M_i}X \setminus \tau_i(V_i))$. Then the result follows since, as in Lemma \ref{lem:sections A}
\[
C_{M_i}(Z_i) \subset (T_{M_i}X \setminus \tau_i(V_i)) \Leftrightarrow C_\chi(Z_i) \cap \imin {\pi_i}(V_i) = \emptyset.
\]
\end{proof}

Let
$
p=p_1 \times \dots \times p_\ell = (q;\, \xi^{(1)},\dots,\xi^{(\ell)})=(q;\,\xi_I,\xi_J) \in S_I \underset{X}{\times} S_J^*
$. Locally we may identify $S_J$ with its dual.
 Set for short $\nu_{\chi_I}\mu_{\chi_J}:=\imin \rho \nu^{sa}_{\chi_I}\mu^{sa}_{\chi_J}$. As a consequence of Lemma \ref{lem:mixed sections Vcirc} we have

\begin{cor}\label{cor:mixed fibers} Let
$p=p_1 \times \dots \times p_\ell = (q;\, \xi^{(1)},\dots,\xi^{(\ell)}) \in S_I \underset{X}{\times}S_J^*$, and
let $F \in D^b(k_{X_{sa}})$. Then $H^k(\nu_{\chi_I}\mu_{\chi_J} F)_p \simeq \lind {Z,W_\epsilon} \operatorname{H}^k_Z( W_\epsilon;F)$,
where $W_\epsilon=W \cap B_\epsilon$, with $W \in \operatorname{Cone}_\chi(q;\,\xi_I,0_J)$, $B_\epsilon$ is an open ball containing $q$ of radius $\epsilon >0$
and $Z$ runs through a family of closed sets in the form
$Z_1 \cap Z_2 \cap \dots \cap Z_{\ell}$
with each $Z_k$ being closed subanalytic in $X$ and
$ C_{M_k}(Z_k)_q \subset p_k^{\#} \cup \{0\} $
 $(k=1,2,\dots,\ell)$.
\end{cor}
\begin{proof}  The result follows since for any subanalytic conic neighborhood $V$ of $(q;\,\xi_I,0_J)$, any $U \in \op(X_{sa})$ such that $C_\chi(X \setminus U) \cap V = \emptyset$ contains $W \cap B_\epsilon$, $q \in B_\epsilon$, $\epsilon >0$, $W \in \operatorname{Cone}(q;\,\xi_I,0_J)$. Moreover, by definition of multi-cone we may assume $W=\bigcap_{j=1}^\ell W_j$ such that $ C_{M_j}(\overline{W_j})_q \subset p_j^{\#} \cup \{0\} $ if $j \in I$ and $W_j=X$ if $j \in J$.
\end{proof}

As in Theorem \ref{th:stalk formula} we can find a cofinal family to the family of closed subsets defining the stalk formula in Corollary \ref{cor:mixed fibers} which (locally) consists of convex cones and we can formulate the stalk formula in the mixed case.

\begin{teo} \label{th:mixed stalk formula}
Let $p=p_1 \times \dots \times p_\ell = (q;\, \xi^{(1)},\dots,\xi^{(\ell)}) \in S_I \underset{X}{\times} S_J^*$,
and let $F \in D^b(k_{X_{sa}})$. Then we have
\begin{equation}\label{eq:th-mixed-stalk-formula}
H^k(\nu_{\chi_I}\mu_{\chi_J} F)_p \simeq \lind {G,W_\epsilon} \operatorname{H}^k_G(W_\epsilon;F).
\end{equation}
Here $W_\epsilon=W \cap B_\epsilon$, with $W \in \operatorname{Cone}_\chi(q;\,\xi_I,0_J)$, $B_\epsilon$ is an open ball of radius $\epsilon >0$ containing $q$
 and a closed subanalytic subset
 $G=M \times \left(\displaystyle\sum_{k=1}^\ell G_k \right)$
with $G_k$ being a closed subanalytic convex cone in $N$
satisfying $G_k \setminus\{0\} \subset \gamma_k$, where $\gamma_k$ is defined in \eqref{gamma k}.
\end{teo}

\section{Multi-microlocalization and microsupport}

In this section we give an estimate of the microsupport of multi-microlocalization. The main point is to find a suitable ambient space: this is done (via Hamiltonian isomorphism) by identifying $T^*S_\chi$ with the normal deformation of $T^*X$ with respect to a suitable family of submanifolds $\chi^*$.

\subsection{Geometry}{\label{sec:geometry}}

Let $X$ be a real analytic manifold and consider a family of submanifolds $\chi=\{M^{}_1,\dots,M^{}_\ell\}$ satisfying H1, H2 and H3.
We consider the conormal bundle $T^*X$ with local coordinates $(x^{(0)},x^{(1)},\dots,x^{(\ell)};\,\xi^{(0)},\xi^{(1)},\dots,\xi^{(\ell)})$,
where  $x^{(j)}=(x^{}_{j^{}_1},\dots,x^{}_{j^{}_p})$ with $\hat{I}_j=\{j^{}_1,\dots,j^{}_p\}$ etc.
We use the notations  in \S\, \ref{sec:MM}; for example, we set
$S_{i}:=T_{M_{i}}\iota(M_{i}) \ftimes_X M$.
Let $I,J \subseteq \{1,\dots,\ell\}$ be such that $I \sqcup J = \{1,\dots,\ell\}$.
Recall that
\begin{align*}
S^{}_{\chi}&= S^{}_1 \ftimes_X \cdots \ftimes_X  S^{}_{\ell},
\\
 S^*_{\chi}&=S^{*}_1 \ftimes_X \cdots \ftimes_X  S^{*}_{\ell},
\\
 S^{}_{I} \ftimes_M  S^{*}_{J}
&= (\ftimes_{M, i\in I} S^{}_i )\ftimes_M (\ftimes_{M, j\in J} S^{*}_j)\,.
\end{align*}
Then we consider a mapping
\begin{multline*}
H^{}_{IJ} \colon T^*S^{}_{\chi} \ni (x^{(0)},x^{(1)},\dots,x^{(\ell)};\,
\xi^{(0)},\xi^{(1)},\dots,\xi^{(\ell)})
\\\mapsto (x^{(0)},(x^{(i)})_{i \in I},(\xi^{(j)})_{j\in J};\,
\eta^{(0)},(\xi^{(i)})_{i\in I}, (-x^{(j)})_{j \in J}) \in  T^*(S^{}_{I} \ftimes_M  S^{*}_{J}).
\end{multline*}
Note that $H^{}_{IJ} $ is induced by the Hamiltonian isomorphisms $T^*S^{}_J \iso T^*S^*_J$.
\begin{prop} \label{prop: Hij}
$H^{}_{IJ}$ gives a
bundle isomorphism over $M$\textup{;} that is, $H^{}_{IJ}$ does not depend on the choice of local coordinates.
\end{prop}
\begin{proof}
Let $\varphi \colon X \to X$ be
a local coordinate transformation near any  $x \in X$.
We may assume that $X = {\mathbb R}^n$
with coordinates  $x=(x^{(0)},x^{(1)}\dots,x^{(\ell)})$, where  $M$ is given by   $(x^{(0)},0,\dots,0)$,
 and $\varphi$ is given by
\[
y^{(j)} = \varphi^{(j)}(x^{(0)},x^{(1)},\dots,x^{(\ell)}) \quad (j=0,1,\dots,\ell).
\]
Here $\varphi^{(j)}(x)=(\varphi^{}_{j^{}_1}(x),\dots,\varphi^{}_{j^{}_{p(j)}}(x))$
with $\hat{I}_j=\{j^{}_1,\dots,j^{}_{p(j)}\}$. This induces a coordinate transformation
\[
T^*X \ni (x;\,\xi) \mapsto  (y;\,\eta) \in T^*X
\]
defined by
\[\left\{
    \begin{aligned}
      y^{(j)}& =\varphi^{(j)}(x), \\
      \xi^{(j)} &=\sum_{i=0}^\ell {}^{\rm t}\!\bigl[\dfrac{\partial \varphi^{(i)}}{\partial x^{(j)}}(x)\bigr]\eta^{(i)}, \\
    \end{aligned}
  \right.
\]
where
\[
\frac{\partial \varphi^{(i)}}{\partial x^{(j)}}(x) = \begin{bmatrix}
\dfrac{\partial \varphi^{}_{i^{}_1}}{\partial x^{}_{j^{}_1}}(x) &\cdots & \dfrac{\partial \varphi^{}_{i^{}_1}}{\partial x^{}_{j^{}_{p(j)}}}(x)
\\
\vdots & \ddots &  \vdots \\
\dfrac{\partial \varphi^{}_{i^{}_{p(i)}}}{\partial x^{}_{j^{}_1}}(x) &\cdots &
 \dfrac{\partial \varphi^{}_{i^{}_{p(i)}}}{\partial x^{}_{j^{}_{p(j)}}}(x)
\end{bmatrix}
\]
is a $p(i) \times p(j)$-matrix, and ${}^{\rm t}$ means the transpose of a matrix.
 Set $J^{}_{j}(x^{(0)}):= \dfrac{\partial \varphi^{(j)}}{\partial x^{(j)}}(x^{(0)},0)$ for short.
 Then  the  coordinate transformation
\[
(x^{(0)},x^{(1)},\dots, x^{(\ell)}) \mapsto (y^{(0)},y^{(1)},\dots, y^{(\ell)})
\]
on  $S^{}_{\chi}$ is given by
\begin{equation}
\label{coordS}
\left\{
    \begin{aligned}
      y^{(0)}&=\varphi^{(0)}(x^{(0)},0), \\
    y^{(j)}&=J^{}_{j}(x^{(0)})\,x^{(j)} \quad (j=1,\dots, \ell).
    \end{aligned}
  \right.
\end{equation}
The  Jacobian matrix of \eqref{coordS}  is
\[
\begin{bmatrix} J^{}_{0}(x^{(0)}) & 0& \cdots & 0\\[2ex]
\dfrac{\partial J^{}_{1}}{\partial x^{(0)}}(x^{(0)})\,x^{(1)}& J^{}_{1}(x^{(0)})
& \cdots & 0\\
\vdots&\vdots & \ddots & \vdots\\
\dfrac{\partial J^{}_{\ell}(x^{(0)})}{\partial x^{(0)}}(x^{(0)})\,x^{(\ell)}& 0 & \cdots
& J^{}_{\ell}(x^{(0)})
\end{bmatrix}.
\]
Thus the the coordinate transformation
\begin{multline*}
(x^{(0)},x^{(1)},\dots,x^{(\ell)};\,
\xi^{(0)},\xi^{(1)},\dots,\xi^{(\ell)})
\\\mapsto (y^{(0)},y^{(1)},\dots,y^{(\ell)};\,
\eta^{(0)},\eta^{(1)},\dots,\eta^{(\ell)})
\end{multline*}
on $T^*S^{}_\chi$ is given by  \eqref{coordS} and
\begin{equation}
\label{coordS1}\left\{
    \begin{aligned}
    \xi^{(0)}&={}^{\rm t}\!J^{}_{0}(x^{(0)})\, \eta^{(0)}+\sum_{i=1}^{\ell}{}^{\rm t}x^{(i)}
 \dfrac{\partial \,{}^{\rm t}\!J^{}_{i}}{\partial x^{(0)}}(x^{(0)})\,\eta^{(i)}, \\
   \xi^{(j)}&= {}^{\rm t}\! J^{}_{j}(x^{(0)}) \,\eta^{(j)} \quad (j=1,\dots,  \ell).
    \end{aligned}
  \right.
\end{equation}
Next, consider the  the coordinate transformation on $S^{}_{I} \ftimes_M  S^{*}_{J}$.
After a permutation, we may assume that $I=\{1,\dots,p\}$, $J=\{p+1,\dots,\ell\}$.
Then the coordinate transformation
\[
(x^{(0)},(x^{(i)})_{i=1}^p,(\xi^{(j)})_{j=p+1}^{\ell}) \mapsto (y^{(0)},(y^{(i)})_{i=1}^p,(\eta^{(j)})_{j=p+1}^{\ell})
\]
on  $S^{}_{I} \ftimes_M  S^{*}_{J}$ is given by
\begin{equation}
\label{coordS*}
\left\{
    \begin{aligned}
      y^{(0)}&=\varphi^{(0)}(x^{(0)},0), \\
    y^{(i)}&=J^{}_{i}(x^{(0)})\,x^{(i)} \quad (i=1,\dots, p),
\\
     \eta^{(j)}&={}^{\rm  t}\! J^{\,-1}_{j}(x^{(0)})\,\xi^{(j)} \quad (j=p+1,\dots,  \ell).
    \end{aligned}
  \right.
\end{equation}
The Jacobian matrix of  \eqref{coordS*}  is
\[
\setlength\arraycolsep{.3ex}
\left[\begin{array}{ccccccc} J^{}_{0}(x^{(0)}) & 0& \cdots &0 &0 &\cdots &0
\\[2ex]
\dfrac{\partial J^{}_{1}}{\partial x^{(0)}}(x^{(0)})\,x^{(1)}& J^{}_{1}(x^{(0)})
& \cdots &0 & 0&\cdots & 0 \\
\vdots&\vdots & \ddots &\vdots & \hspace{-3ex}\vdots && \vdots
\\
\dfrac{\partial J^{}_{p}}{\partial x^{(0)}}(x^{(0)})\,x^{(p)}& 0 & \cdots
& J^{}_{p}(x^{(0)}) &0&\cdots &0
\\[2ex]
\dfrac{\partial\, {}^{\rm t}\!
J^{\,-1}_{p+1}}{\partial x^{(0)}}(x^{(0)})\,\xi^{(p+1)}&0& \cdots&0 & {}^{\rm t}\! J^{\,-1}_{p+1}(x^{(0)})
& \cdots & 0\\
\vdots&\vdots &  & \vdots && \ddots& \vdots\\
\dfrac{\partial\, {}^{\rm t}\!
J^{\,-1}_{\ell}}{\partial x^{(0)}}(x^{(0)})\,\xi^{(\ell)}& 0 & \cdots
& 0 & 0& \cdots&{}^{\rm t}\! J^{\,-1}_{\ell}(x^{(0)})
\end{array}
\right].
\]
Thus the the coordinate transformation
\begin{multline*}
(x^{(0)},(x^{(i)})_{i=1}^p,(\xi^{(j)})_{j=p+1}^{\ell};\, \xi^{(0)},(\xi^{(i)})_{i=1}^p,(-x^{(j)})_{j=p+1}^{\ell})
\\*
\mapsto
(y^{(0)},(y^{(i)})_{i=1}^p,(\eta^{(j)})_{j=p+1}^{\ell};\, \eta^{(0)},(\eta^{(i)})_{i=1}^p,(-y^{(j)})_{j=p+1}^{\ell})
\end{multline*}
on $T^*S^{*}_\chi$ is given by \eqref{coordS*} and
\[
\left\{
    \begin{aligned}
       \xi^{(0)}={}&{}^{\rm t}\!J^{}_{0}(x^{(0)})\, \eta^{(0)}
+\sum_{i=1}^{p}{}^{\rm t}x^{(i)}
 \dfrac{\partial \,{}^{\rm t}\!J^{}_{i}}{\partial x^{(0)}}(x^{(0)})\,\eta^{(i)}
-\smashoperator{\sum_{j=p+1}^{\ell}}
 {}^{\rm t} \xi^{(j)}\dfrac{\partial
J^{\,-1}_{j}}{\partial x^{(0)}}(x^{(0)})\,
y^{(j)},
\\
 \xi^{(i)} ={}&{}^{\rm  t}\! J^{}_{i}(x^{(0)})\,\eta^{(i)} \quad (i=1,\dots,  p),
\\
    x^{(j)}={}& J^{\,-1}_{j}(x^{(0)})\,y^{(j)} \quad (j=p+1,\dots, \ell).
    \end{aligned}
  \right.
\]
Since
\[
\dfrac{\partial
J^{\,-1}_{j}}{\partial x^{(0)}}(x^{(0)}) = -J^{\,-1}_{j}(x^{(0)}) \dfrac{\partial
J^{}_{j}}{\partial x^{(0)}}(x^{(0)})  \,J^{\,-1}_{j}(x^{(0)}),
\]
we have
\begin{align*}
- ({}^{\rm t} \xi^{(j)}&\dfrac{\partial
J^{\,-1}_{j}}{\partial x^{(0)}}(x^{(0)})\,y^{(j)})^{}_k
 =  ({}^{\rm t} \xi^{(j)}J^{\,-1}_{j}(x^{(0)}) \dfrac{\partial J^{}_{j}}{\partial x^{(0)}}(x^{(0)})  \,J^{\,-1}_{j}(x^{(0)})\,y^{(j)})^{}_k
\\
& = ({}^{\rm t} \eta^{(j)} \dfrac{\partial J^{}_{j}}{\partial x^{(0)}}(x^{(0)})\, x^{(j)})^{}_{k}=
\smashoperator{\sum_{\mu,\nu\in \hat{I}^{}_j}} \eta^{(j)}_\mu  \Bigl(\dfrac{\partial  J^{}_{j}}{\partial x^{(0)}_k}(x^{(0)})\Bigr)_{\mu,\nu}\, x^{(j)}_\nu
\\
&= ({}^{\rm t}x^{(j)}
 \dfrac{\partial \,{}^{\rm t}\!J^{}_{j}}{\partial x^{(0)}}(x^{(0)}) \,\eta^{(j)})^{}_k\,.
\end{align*}
Therefore
\[
-\smashoperator{\sum_{j=p+1}^{\ell}}
 {}^{\rm t} \xi^{(j)}\dfrac{\partial
J^{\,-1}_{p}}{\partial x^{(0)}}(x^{(0)})\,
y^{(j)}=\smashoperator{\sum_{j=p+1}^{\ell}}
{}^{\rm t}x^{(j)}
 \dfrac{\partial \,{}^{\rm t}\!J^{}_{j}}{\partial x^{(0)}}(x^{(0)}) \,\eta^{(j)}.
\]
Thus  we can  prove that
\begin{equation}\label{Hij}
H^{}_{IJ} \colon T^*S^{}_{\chi} \iso   T^*(S^{}_{I} \ftimes_M  S^{*}_{J}).
\qedhere
\end{equation}
 \end{proof}
Hence, using  Proposition 5.5.5 of \cite{KS90} repeatedly, we obtain:
\begin{prop}
Let $I,J \subseteq \{1,\dots,\ell\}$ be such that $I \sqcup J = \{1,\dots,\ell\}$.
Then, under the identification by \eqref{Hij}, for any $F \in \bDb(k^{}_X)$  it follows that
\[\xymatrix @C=1em @R=3ex{
 T^*S_{\chi} \ar@{}[d]_{\bigcup} \ar@{=}[r]& T^*(S^{}_{I} \ftimes_M  S^{*}_{J}) \ar@{}[d]_{\bigcup}
\\
\MS (\nu^{}_\chi (F)) \ar@{=}[r] &\MS (\nu^{}_{\chi_I}\mu^{}_{\chi_J} (F)).
}\]
In particular,  it follows that
\[\xymatrix @C=1em @R=2ex{
 T^*S_{\chi} \ar@{}[d]_{\bigcup} \ar@{=}[r]& T^*S^{*}_{\chi} \ar@{}[d]_{\bigcup}
\\
\MS (\nu^{}_\chi (F)) \ar@{=}[r] &\MS (\mu^{}_{\chi} (F)).
}\]
\end{prop}
Next,  we study the relation between the normal deformations
of $T^*X$ with respect to $\chi^*:=\{T^*_{M_1}X,\dots,T^*_{M_\ell}X\}$  and  of $X$ with respect to $\chi$.
We denote by $\widetilde{T^*X}^{}_{\chi^*}:=\widetilde{T^*X}_{T^*_{M_1}X,\dots,T^*_{M_\ell}X}$ the normal deformation of $T^*X$
with respect to $\chi^*$ and by $S_{\chi^*}$ its zero-section.
 Set $x:=(x^{(0)},x^{(1)},\dots,x^{(\ell)})$, $\xi:=(\xi^{(0)},\xi^{(1)},\dots,\xi^{(\ell)})$ and $t:=(t^{}_1,\dots,t^{}_\ell)$.
We have a mapping
\[
\widetilde{T^*X}^{}_{\chi^*} \ni(x;\,\xi;\,t)  \mapsto  (\mu_x(x;\,t);\,\mu_\xi(\xi;\,t)) \in  T^*X
\]
defined by
\begin{align*}
\mu_x(x;\,t) & :=  (t_{\JJ_0}x^{(0)},t_{\JJ_1}x^{(1)},\dots,t_{\JJ_\ell}x^{(\ell)}) , \\
\mu_\xi(\xi;\,t) & :=  (t_{\JJ^c_0}\xi^{(0)},t_{\JJ^c_1}\xi^{(1)},\dots,t_{\JJ^c_\ell}\xi^{(\ell)}) ,
\end{align*}
where $\JJ^c_j :=\{1,\dots,\ell\} \setminus \JJ_j$ ($j=0,1,\dots,\ell$).
In particular $\JJ^c_0=\{1,\dots,\ell\}$ since $\JJ^{}_0=\emptyset$. In particular $t_{\JJ_0}=1$ and
$t_{\JJ^c_0}= t^{}_1 \cdots t^{}_\ell$.

\begin{teo} \label{thm: geometry}
As vector  bundles,
there exist the following canonical isomorphism\textup{:}
\[
S_{\chi^*} \simeq T^*S^{}_\chi \simeq T^*S^*_\chi .
\]
\end{teo}
\begin{proof}
Let $\varphi \colon X \to X$ be
a local coordinate transformation near any  $x \in X$, and
retain the notation of the proof of Proposition \ref{prop: Hij}.
The coordinate transformation $(x;\,\xi;\,t) \mapsto (y;\,\eta;\,t)$ on $\widetilde{T^*X}^{}_{\chi^*}
\setminus S^{}_{\chi^*}$ is given by
\[\left\{
    \begin{aligned}
      y^{(j)}&=\dfrac{1}{t^{}_{\JJ_j}}\, \varphi^{(j)}(t^{}_{\JJ} x), \\
      \xi^{(j)}&=\displaystyle\sum_{i=0}^\ell {}^{\rm t}[\dfrac{\partial \varphi^{(i)}}{\partial x^{(j)}}(t^{}_{\JJ} x)]{t^{}_{J^c_i} \over t^{}_{J^c_j}}\eta^{(i)},
    \end{aligned}
  \right.
\]
where $t^{}_{\JJ} x:=\mu_x(x;\,t) =  (t_{\JJ_0}x^{(0)},t_{\JJ_1}x^{(1)},\dots,t_{\JJ_\ell}x^{(\ell)})$.
Let us consider the coordinate transformation on $S^{}_{\chi^*}$.
We write for short $t \to 0$ instead of $(t^{}_1,\dots,t^{}_\ell) \to (0,\dots,0)$.
 Set $J^{}_{j}(x^{(0)}):= \dfrac{\partial \varphi^{(j)}}{\partial x^{(j)}}(x^{(0)},0)$ for short.
Then, by Proposition 1.5 of \cite{HP} on $S_{\chi^*}$
\[\left\{
    \begin{aligned}
y^{(0)} & =  \varphi^{(0)}(x^{(0)},0), \\
y^{(j)} & = J^{}_{j}(x^{(0)})\,x^{(j)}, \quad (j=1,\dots, \ell),\\
    \end{aligned}
  \right.
\]
that means $y^{}_k=\sum_{p \in \hat{I}^{}_k}\dfrac{\partial \varphi^{}_k}{\partial x_p}(x^{(0)},0)\,x^{}_p$ for all $k \in \hat{I}_j$.
Concerning the variable $\xi^{(0)}$, as in the proof of Proposition \ref{prop: Hij}  we get
\[
\xi^{(0)}={}^{\rm t}[\dfrac{\partial \varphi^{(0)}}{\partial x^{(0)}}(t_{\JJ} x)]
\eta^{(0)} + \sum_{i=1}^\ell{}^{\rm t}[\dfrac{\partial \varphi^{(i)}}{\partial x^{(0)}}(t_{\JJ} x)]
\dfrac{t_{\JJ^c_i}}{t_{\JJ^c_0}}\, \eta^{(i)}.
\]
Let $M^{}_i=\{x_k=0;\;k\in I^{}_i\}$ and $I^{}_i=\hat{I}_{j^{}_1} \sqcup \cdots \sqcup \hat{I}^{}_{j^{}_p}$.
By expanding ${}^{\rm t}[\dfrac{\partial \varphi^{(i)}}{\partial x^{(0)}}(t_{\JJ} x)]$ along the submanifold $M^{}_i$, we obtain
\begin{align*}
{}^{\rm t}[\dfrac{\partial \varphi^{(i)}}{\partial x^{(0)}}(t_{\JJ} x)] \dfrac{t_{\JJ^c_i}}{t_{\JJ^c_0}}\, \eta^{(i)} ={} &
{}^{\rm t}[\dfrac{\partial \varphi^{(i)}}{\partial x^{(0)}}(t_{\JJ} x)]\Big|_{M^{}_i} \,\dfrac{t_{\JJ^c_i}}{ t_{\JJ^c_0}} \eta^{(i)}
\\*
&+ \sum_{k \in I^{}_i} x^{}_k{}^{\rm t}[\dfrac{\partial^2 \varphi^{(i)}}{\partial x^{}_k \partial x^{(0)} }(t_{\JJ} x)]\Big|_{M^{}_i}
\dfrac{t_{\JJ^c_i}t_{\JJ_k}}{ t_{\JJ^c_0}} \eta^{(i)}+\cdots
\end{align*}
Since $\varphi^{(i)}(t_{\JJ} x)\big|_{M_i}=0$ and $\hat{I}^{}_0 \cap I^{}_i = \emptyset$
we have $\dfrac{\partial \varphi^{(i)}}{\partial x^{(0)}}(t_{\JJ} x)\Big|_{M^{}_i}=0$. Moreover
\[\left\{
    \begin{aligned}
      \dfrac{t_{\JJ^c_i}t_{\JJ_i} }{ t_{\JJ^c_0}}&=1, \\
      \dfrac{t_{\JJ^c_i}t_{\JJ_k}}{t_{\JJ^c_0}} &\to 0, \quad (k \neq i), \\
    \end{aligned}
  \right.
\]
when $t \to 0$. This is because $\JJ_k \subsetneqq \JJ_i$ when $k \in \{j_1,\dots,j_p\}$, $k \neq i$ by Lemma \ref{lem:inclusions of J} and \eqref{eq:equality of J}. In a similar way the higher order terms vanish when $t \to 0$. Hence on $S_{\chi^*}$
\allowdisplaybreaks
\begin{align*}
\xi^{(0)} & ={}^{\rm t}[\dfrac{\partial \varphi^{(0)}}{ \partial x^{(0)}}(x^{(0)},0)]\,\eta^{(0)} + \sum_{i=1}^\ell
{}^{\rm t}x^{(i)}\,{}^{\rm t}[
\dfrac{\partial^2 \varphi^{(i)} }{\partial x^{(i)} \partial x^{(0)}}(x^{(0)},0)]\eta^{(i)}
\\
& ={}^{\rm t}\!J^{}_{0}(x^{(0)})\, \eta^{(0)}+\sum_{i=1}^{\ell}{}^{\rm t}x^{(i)}
 \dfrac{\partial \,{}^{\rm t}\!J^{}_{i}}{\partial x^{(0)}}(x^{(0)})\,\eta^{(i)}.
\end{align*}
Concerning the variable $\xi^{(j)}$ ($j \neq 0$), we get
\[
\xi^{(j)}=\sum_{i=0}^\ell {}^{\rm t}[\dfrac{\partial \varphi^{(i)}}{\partial x^{(j)}}(t_{\JJ} x)]\,\dfrac{t_{\JJ^c_i}}{ t_{\JJ^c_j}}\, \eta^{(i)}.
\]
\begin{enumerate}[(i)]
\item  If $\JJ^c_j \subsetneqq \JJ^c_i  \Leftrightarrow  \JJ^{}_i \subsetneqq \JJ^{}_j$
we have $\dfrac{t_{\JJ^c_i}}{ t_{\JJ^c_j}}\to 0$ ($t \to 0$).
\item
If  $\JJ^{}_i \supsetneqq \JJ^{}_j$ or $\JJ^{}_i \cap \JJ^{}_j = \emptyset$, we have $\hat{I}^{}_j \cap I^{}_i = \emptyset$.
\end{enumerate}
By expanding $\dfrac{\partial \varphi^{(i)}}{\partial x^{(j)}}(t^{}_{\JJ} x)$ along the submanifold $M^{}_i$, we obtain
\begin{align*}
{}^{\rm t}[\dfrac{\partial \varphi^{(i)}}{\partial x^{(j)}}(t^{}_{\JJ} x)]\,\dfrac{t_{\JJ^c_i}}{t_{\JJ^c_j}} \eta^{(i)}
 = {}& {}^{\rm t}[\dfrac{\partial \varphi^{(i)}}{\partial x^{(j)}}(t^{}_{\JJ} x)]\Big|_{M^{}_i}\,\dfrac{t_{\JJ^c_i}}{t_{\JJ^c_j}} \,\eta^{(i)}
\\*
&
 + \sum_{k \in I_i}x^{}_k{}^{\rm t}[\dfrac{\partial^2 \varphi^{(i)}}{ \partial x^{}_{k}\partial x^{(j)} }(t^{}_{\JJ} x)]\Bigr|_{M^{}_i}
\dfrac{t_{\JJ^c_i}t_{\JJ_k}}{t_{\JJ^c_j}} \eta^{(i)}+\cdots
\end{align*}
Since $\varphi^{(i)}(t_{\JJ} x)\big|^{}_{M^{}_i}=0$ and $\hat{I}^{}_j \cap I^{}_i = \emptyset$
 we have $\dfrac{\partial \varphi^{(i)}}{\partial x^{(j)}}(t_{\JJ} x)\Big|^{}_{M^{}_i}=0$. Moreover
\[
\dfrac{t_{\JJ^c_i}t_{\JJ_k}}{ t_{\JJ^c_j}} \to 0
\]
when $t \to 0$ since $\hat{I}^{}_k \subseteq I^{}_i  \Rightarrow  \JJ^{}_k \supseteq \JJ^{}_i$ by Lemma \ref{lem:inclusions of J}
and $\JJ^c_i \cup \JJ^{}_k = \{1,\dots,\ell\} \supsetneqq \JJ^c_j$ when $j \neq 0$. Hence on $S_{\chi^*}$
\[
\xi^{(j)}= {}^{\rm t}[\frac{\partial \varphi^{(j)}}{ \partial x^{(j)}}(x^{(0)},0)]\,\eta^{(j)}=
{}^{\rm t}\! J^{}_{j}(x^{(0)}) \,\eta^{(j)} .
\]
Summarizing, the coordinate transformation on $S_{\chi^*}$ is given by
\begin{align*}&\left\{
    \begin{aligned}
      y^{(0)}&=\varphi^{(0)}(x^{(0)},0), \\
    y^{(i)}&=J^{}_{i}(x^{(0)})\,x^{(i)} \quad (1 \leqslant i \leqslant \ell),
    \end{aligned}
  \right.
\\
&\left\{
    \begin{aligned}
    \xi^{(0)}&={}^{\rm t}\!J^{}_{0}(x^{(0)})\, \eta^{(0)}+\sum_{i=1}^{\ell}{}^{\rm t}x^{(i)}
 \dfrac{\partial \,{}^{\rm t}\!J^{}_{i}}{\partial x^{(0)}}(x^{(0)})\,\eta^{(i)}, \\
   \xi^{(i)}&= {}^{\rm t}\! J^{}_{i}(x^{(0)}) \,\eta^{(i)} \quad (1 \leqslant i \leqslant \ell).
    \end{aligned}
  \right.
\end{align*}
This is nothing but \eqref{coordS}, \eqref{coordS1}.
 \end{proof}

 \begin{es} Let $X={\mathbb C}^2$
with coordinates $(z_1, z_2)$ and consider $T^*X$ with coordinates $(z;\,\eta)=(z_1,z_2;\,\eta_1,\eta_2)$. Set $t=(t_1,t_2) \in (\RP)^2$.
\begin{enumerate}
\item (Majima)  Let $M_1 = \{z_1 = 0\}$ and $M_2 = \{z_2 = 0\}$.
Then $\chi^*=\{T^*_{M_1}X,T^*_{M_2}X\}$ and we have a map
\begin{eqnarray*}
\widetilde{T^*X} & \to & T^*X, \\
(z;\,\eta;\,t) & \mapsto & (\mu_z(z;\,t);\,\mu_\eta(\eta;\,t)),
\end{eqnarray*}
which is defined by
\begin{eqnarray*}
\mu_z(z;\,t) & = & (t_1z_1,t_2z_2) , \\
\mu_\eta(\eta;\,t) & = & (t_2\eta_1,t_1\eta_2).
\end{eqnarray*}
By Theorem \ref{thm: geometry} we have $S_\chi^* \simeq T^*(T_{M_1}X \underset{X}{\times}T_{M_2}X) \simeq T^*(T^*_{M_1}X \underset{X}{\times}T^*_{M_2}X)$.

\item (Takeuchi) Let $M_1 = \{0\}$ and $M_2 = \{z_2 = 0\}$.
Then $\chi^*=\{T^*_{M_1}X,T^*_{M_2}X\}$ and we have a map
\begin{eqnarray*}
\widetilde{T^*X} & \to & T^*X, \\
(z;\,\eta;\,t) & \mapsto & (\mu_z(z;\,t);\,\mu_\eta(\eta;\,t)),
\end{eqnarray*}
which is defined by
\begin{eqnarray*}
\mu_z(z;\,t) & = & (t_1z_1,t_1t_2z_2) , \\
\mu_\eta(\eta;\,t) & = & (t_2\eta_1,\eta_2).
\end{eqnarray*}
By Theorem \ref{thm: geometry} we have $S_\chi^* \simeq T^*(T_{M_1}M_2 \underset{X}{\times}T_{M_2}X) \simeq T^*(T^*_{M_1}M_2 \underset{X}{\times}T^*_{M_2}X)$.

\end{enumerate}
\end{es}

\begin{es} Let $X={\mathbb R}^3$ with coordinates $(x_1, x_2, x_3)$ and consider $T^*X$ with coordinates $(x;\,\xi)=(x_1,x_2,x_3;\,\xi_1,\xi_2,\xi_3)$. Set $t=(t_1,t_2,t_3) \in (\RP)^3$.
\begin{enumerate}
\item (Majima) Let $M_1 = \{x_1 = 0\}$, $M_2 = \{x_2 = 0\}$ and $M_3=\{x_3=0\}$.
Then $\chi^*=\{T^*_{M_1}X,T^*_{M_2}X,T^*_{M_3}X\}$ and we have a map
\begin{eqnarray*}
\widetilde{T^*X} & \to & T^*X, \\
(x;\,\xi;\,t) & \mapsto & (\mu_x(x;\,t);\,\mu_\xi(\xi;\,t)),
\end{eqnarray*}
which is defined by
\begin{eqnarray*}
\mu_x(x;\,t) & = & (t_1x_1,t_2x_2,t_3x_3) , \\
\mu_\xi(\xi;\,t) & = & (t_2t_3\xi_1,t_1t_3\xi_2,t_2t_3\xi_3).
\end{eqnarray*}
By Theorem \ref{thm: geometry} we have $S_\chi^* \simeq T^*(T_{M_1}X \underset{X}{\times}T_{M_2}X \underset{X}{\times}T_{M_2}X) \simeq T^*(T^*_{M_1}X \underset{X}{\times}T^*_{M_2}X \underset{X}{\times}T^*_{M_3}X)$.

\item (Takeuchi) Let $M_1 = \{0\}$, $M_2 = \{x_2 = x_3 = 0\}$ and $M_3 = \{x_3 = 0\}$.
Then $\chi^*=\{T^*_{M_1}X,T^*_{M_2}X,T^*_{M_3}X\}$ and we have a map
\begin{eqnarray*}
\widetilde{T^*X} & \to & T^*X, \\
(x;\,\xi;\,t) & \mapsto & (\mu_x(x;\,t);\,\mu_\xi(\xi;\,t)),
\end{eqnarray*}
which is defined by
\begin{eqnarray*}
\mu_x(x;\,t) & = & (t_1x_1,t_1t_2x_2,t_1t_2t_3x_3) , \\
\mu_\xi(\xi;\,t) & = & (t_2t_3\xi_1,t_3\xi_2,\xi_3).
\end{eqnarray*}
By Theorem \ref{thm: geometry} we have $S_\chi^* \simeq T^*(T_{M_1}M_2 \underset{X}{\times}T_{M_2}M_3 \underset{X}{\times}T_{M_2}X) \simeq T^*(T^*_{M_1}M_2 \underset{X}{\times}T^*_{M_2}M_3 \underset{X}{\times}T^*_{M_3}X)$.

\item (Mixed)
Let $M_1 = \{0\}$, $M_2 = \{x_2 = 0\}$ and $M_3 = \{x_3 = 0\}$.
Then $\chi^*=\{T^*_{M_1}X,T^*_{M_2}X,T^*_{M_3}X\}$ and we have a map
\begin{eqnarray*}
\widetilde{T^*X} & \to & T^*X, \\
(x;\,\xi;\,t) & \mapsto & (\mu_x(x;\,t);\,\mu_\xi(\xi;\,t)),
\end{eqnarray*}
which is defined by
\begin{eqnarray*}
\mu_x(x;\,t) & = & (t_1x_1,t_1t_2x_2,t_1t_3x_3) , \\
\mu_\xi(\xi;\,t) & = & (t_2t_3\xi_1,t_3\xi_2,t_2\xi_3).
\end{eqnarray*}
By Theorem \ref{thm: geometry} we have $S_\chi^* \simeq T^*(T_{M_1}(M_2 \cap M_3) \underset{X}{\times}T_{M_2}X \underset{X}{\times}T_{M_2}X) \simeq T^*(T^*_{M_1}(M_2 \cap M_3) \underset{X}{\times}T^*_{M_2}X \underset{X}{\times}T^*_{M_3}X)$.

\end{enumerate}
\end{es}

\subsection{Estimate of microsupport}

In this section we shall prove  an estimate for the microsupport of the multi-specialization and multi-microlocalization of a sheaf on $X$. We refer to \cite{KS90} for the theory of microsupport of sheaves.

\begin{teo} \label{teo: estimate cone} Let $F \in \bDb(k_X)$. Then
\[
\MS(\nu_\chi (F))=\MS (\mu^{}_\chi (F)) \subseteq C_{\chi^*}(\MS(F)).
\]
\end{teo}
Since the problem is local, we may assume that $X=\R^n$ with coordinates $(x_1,\dots,x_n)$.
Let $I_k$ $(k=1,2,\dots,\ell)$ in $\{1,2,\dots,n\}$ such that each submanifold $M_k$ is given by
\[
\left\{x = (x_1,\dots,x_n) \in \mathbb{R}^n;\, x_i=0\, (i \in I_k)\right\}.
\]
Then  Theorem \ref{teo: estimate cone} follows from
 Lemma \ref{lem: multi-normal cone} and the following theorem:
\begin{teo} \label{thm: estimate sequences} Let $F \in \bDb(k_X)$ and take a point
\[
p_0 =(x^{(0)}_0,x^{(1)}_0,\dots,x^{(\ell)}_0;\,\xi^{(0)}_0,\xi^{(1)}_0,\dots,\xi^{(\ell)}_0) \in T^*S_\chi.
\]
Assume that $p_0 \in \MS(\nu_\chi (F))$. Then there exist sequences
\begin{align*}
\{(c_{1,k},\dots,c_{\ell,k})\}_{k=1}^\infty
& \subset (\RP)^\ell,
\\
\{(x^{(0)}_{k},x^{(1)}_k,\dots,x^{(\ell)}_{ k};\,\xi^{(0)}_{k},\xi^{(1)}_k,\dots,\xi^{(\ell)}_{k})\}_{k=1}^\infty
& \subset \MS(F),
\end{align*}
 such that
\[\left\{
    \begin{aligned}
     \lim_{k\to \infty} & c_{j,k} = \infty, \qquad (j=1,\dots,\ell), \\
     \lim_{k\to \infty} &(x^{(0)}_{k},x^{(1)}_k c_{\JJ_1,k},\dots,x^{(\ell)}_{k}c_{\JJ_\ell,k}
;\,\xi^{(0)}_k c^{}_k,\xi^{(1)}_{k}c_{\JJ_1^c,k},\dots,\xi^{(\ell)}_{ k}c_{\JJ_\ell^c,k})
\\
& = (x^{(0)}_0,x^{(1)}_0,\dots,x^{(\ell)}_0;\,\xi^{(0)}_0,\xi^{(1)}_0,\dots,\xi^{(\ell)}_0) ,
    \end{aligned}
  \right.
\]
where $c^{}_k :=\prod\limits_{j=1}^\ell c_{j,k}$, $\JJ^c_j:=\{1,\dots,\ell\}\setminus\JJ_j$,
and $c_{J,k} :=\prod\limits_{j \in J}c_{j,k}$ for any $J \subseteq \{1,\dots,\ell\}$.
\end{teo}
\begin{proof}
Let $(x;\,t)=(x^{(0)},x^{(1)},\dots,x^{(\ell)};\,t_1,\dots,t_\ell)$ be the coordinates in $\widetilde{X}$. It follows from the estimate of the microsupport of the inverse image of a closed embedding (Lemma 6.2.1 (ii) and Proposition 6.2.4 (iii) of \cite{KS90}) that there exists a sequence
\[
\{(x^{(0)}_k,x^{(1)}_k,\dots,x^{(\ell)}_k;\,t^{}_{1,k},\dots,t_{\ell, k};\,\xi^{(0)}_k,\xi^{(1)}_k,\dots,\xi^{(\ell)}_k;\,
\tau_{1,k},\dots,\tau_{\ell, k})\}_{k=1}^\infty
\]
in $\MS(Rj_{\Omega *}\widetilde{p}^{\,-1}F)$ such that for any $j=1,\dots,\ell$
\[
\left\{
    \begin{aligned}
        \lim_{k\to \infty}& x^{(j)}_k =  x^{(j)}_0, \quad
         \lim_{k\to \infty}\xi^{(j)}_k  = \xi^{(j)}_0, \quad
      \lim_{k\to \infty}   t_{j,k}   =0, \\
     \lim_{k\to \infty}  &  |(t_{1,k},\dots,t_{\ell,k})|\cdot|(\tau_{1,k},\dots,\tau_{\ell,k})|= 0.
    \end{aligned}
  \right.
\]
By Theorem 6.3.1 of \cite{KS90} we have
\[
\MS(Rj_{\Omega *}\widetilde{p}^{\,-1}F) \subseteq \MS(\widetilde{p}^{\,-1}F) \mathop{\widehat{+} }N^*(\Omega).
\]
By Proposition 5.4.5 of \cite{KS90}, we have
\allowdisplaybreaks
\begin{align*}
\MS(\widetilde{p}^{\,-1}&F) = \widetilde{p}^{}_d(\widetilde{p}^{\,-1}_\pi(\MS(F))\\*
 =  \{(&x^{(0)},\dfrac{x^{(1)} }{ t_{\JJ_1}},\dots,\dfrac{x^{(\ell)}}{t_{\JJ_\ell}};\,
t_1,\dots,t_\ell;\,\xi^{(0)},t_{\JJ_1}\xi^{(1)},\dots,t_{\JJ_\ell}\xi^{(\ell)};\,\tau_1,\dots,\tau_\ell);\\*
&  t_j>0\; (j=1,\dots,\ell),\, (x^{(0)},x^{(1)},\dots,x^{(\ell)};\,\xi^{(0)},\xi^{(1)},\dots,\xi^{(\ell)}) \in \MS(F)\},
\end{align*}
where we did not calculate the terms in the variables $\tau_j$ ($j=1,\dots,\ell$) since we are not going to use them.
Thanks to Remark 6.2.8 (ii) of \cite{KS90} and the fact that $N^*(\Omega)\subset
 \{(x;\,t;\,0;\,\tau)\}$, for each $k \in \N$ we get sequences
\begin{align*}
\{(x^{(0)}_{k,m},x^{(1)}_{k,m},\dots,x^{(\ell)}_{k,m};\,\xi^{(0)}_{k,m},\xi^{(1)}_{k,m},\dots,\xi^{(\ell)}_{k,m})\}_{m=1}^\infty
&\subset \MS(F),
\\
\{(t_{1,k,m},\dots,t_{\ell, k,m})\}& \subset (\RP)^\ell,
\end{align*} such that
\begin{align*}
\lim_{m\to \infty}&(x^{(0)}_{k,m},{x^{(1)}_{k,m} \over t_{\JJ_1,k,m}},\dots,{x^{(\ell)}_{k,m} \over t_{\JJ_\ell,k,m}};\,t_{1,k,m},
\dots,t_{\ell,k,m})
\\ & =
(x^{(0)}_k,x^{(1)}_k,\dots,x^{(\ell)}_k;\,t_{1,k},\dots,t_{\ell,k}), \\
\lim_{m\to \infty}&    (\xi^{(0)}_{k,m},t_{\JJ_1,k,m},\xi^{(1)}_{k,m},\dots,t_{\JJ_\ell,k,m}\xi^{(\ell)}_{k,m}) =
 (\xi^{(0)}_k,\xi^{(1)}_k,\dots,\xi^{(\ell)}_k).
  \end{align*}
Then extracting a subsequence, we can find
\[
\{(x^{(0)}_{k},x^{(1)}_{k},\dots,x^{(\ell)}_{k};\,\xi^{(0)}_{k},\xi^{(1)}_{k},\dots,\xi^{(\ell)}_{k})\}_{k=1}^\infty \subset \MS(F)
\]
and
$\{(t_{1,k},\dots,t_{\ell, k})\}_{k=1}^\infty \subset (\RP)^\ell$ such that
\[
\left\{
    \begin{aligned}
\lim_{k\to \infty}(x^{(0)}_{k},{x^{(1)}_{k} \over t_{\JJ_1,k}},\dots,{x^{(\ell)}_{k} \over t_{\JJ_\ell,k}})
&=  (x^{(0)},x^{(1)},\dots,x^{(\ell)}), \\
  \lim_{k\to \infty}  (\xi^{(0)}_{k},t_{\JJ_1,k}\xi^{(1)}_{k},\dots,t_{\JJ_\ell,k}\xi^{(\ell)}_{k}) &= (\xi^{(0)},\xi^{(1)},\dots,\xi^{(\ell)}),
\\
\lim_{k\to \infty}      (t_{1,k},\dots,t_{\ell,k}) &= (0,\dots,0).
    \end{aligned}
  \right.
\]
Since $\MS(F)$ is conic, we have
\[
(x^{(0)}_k,x^{(1)}_k,\dots,x^{(\ell)}_k;\,t_k\xi^{(0)}_k,t_k\xi^{(1)}_k,\dots,t_k\xi^{(\ell)}_k) \in \MS(F),
\]
where $t_k:=\prod\limits_{j=1}^\ell t_{j,k}$. Setting $c_{j,k} :=\dfrac{1}{t_{j,k}}$ ($j=1,\dots,\ell$)
we obtain the desired result. \end{proof}

\begin{oss} Theorem \ref{teo: estimate cone} extends the estimate of microsupport computed in \cite{KT01}.
\end{oss}

\section{Microfunctions along $\chi$}
In this section, we establish a vanishing theorem of cohomology groups for
multi-microlocalization and introduce multi-microlocalized objects along $\chi$
which are natural extensions of sheaves of microfunctions and holomorphic ones.

\subsection{The edge of the wedge theorem with bounds}
We say that an open subset $\Omega \subset \mathbb{C}^n$ is an analytic open polyhedron
if there exist holomorphic functions
$f_1,\dots, f_\ell$ on $\mathbb C^n$
satisfying
\[
\Omega = \{z \in \mathbb C^n; |f_k(z)| < 1\,\, (k=1,2,\dots,\ell)\}.
\]
In the same way,  a closed subset $K \subset \mathbb{C}^n$ is said to be an analytic closed
polyhedron if $K$ has the form
\[
K = \{z \in \mathbb C^n; |f_k(z)| \le 1\,\, (k=1,2,\dots,\ell)\}.
\]
for some holomorphic functions $f_1,\dots, f_\ell$ on $\mathbb{C}^n$.
We first establish the edge of the wedge theorem for $\OW$.
Set $X:= \mathbb{C}^n \times \mathbb{C}^m$.

\begin{teo}\label{th:vanishing-main}
Let $\Omega$ and $\omega$ in $\mathbb C^n$ be relatively compact analytic open polyhedra,
and let $K$ be a closed analytic polyhedron or a closed convex subanalytic subset
in $\mathbb C^m$. Then we have
\[
\operatorname{H}^k(X;\, \mathbb{C}_{(\Omega\setminus \omega) \times K} \wtens \mathcal{O}_X) = 0
\qquad (k \ne n).
\]
\end{teo}

\noindent To show the theorem,  we need several lemmas. In what follows, we always assume that
$K$ is a closed analytic polyhedron or a closed convex subanalytic subset in $\mathbb{C}^m$.
We note that, by the result of A.~Dufresnoy \cite{Du79},
we have
\[
\operatorname{H}^k(\mathbb{C}^m;\, \mathbb C_K \wtens \mathcal{O}_{\mathbb{C}^m}) = 0
\quad (k \ne 0)
\]
for the both $K$.
\begin{lem}
Assume that $\Omega_i$ $(i=1,2,\dots,n)$ is a non-empty open disk in $\mathbb{C}$.
Set $\Omega := \Omega_1 \times \dots \times \Omega_n$. Then
$
\mathrm{R}\Gamma(X;\, \mathbb C_{\Omega \times K} \wtens \mathcal{O}_X)
$
is concentrated in degree $n$.
\end{lem}
\begin{proof}
We have
$\operatorname{H}^k(\mathbb{C};\, \mathbb{C}_{\mathbb{C} \setminus \Omega_i} \wtens \mathcal O_{\mathbb{C}}) =0$ for
$k \ne 0$, from which we get
\[
\operatorname{H}^k(\mathbb{C};\, \mathbb{C}_{\Omega_i} \wtens \mathcal O_{\mathbb{C}}) = 0 \qquad (k \ne 1)
\]
and the exact sequence
\[
0
\to \operatorname{H}^0(\mathbb{C};\, \mathbb{C}_{\mathbb{C}} \wtens \mathcal O_{\mathbb{C}})
\overset{\rho}{\to}
\operatorname{H}^0(\mathbb{C};\, \mathbb{C}_{\mathbb{C} \setminus \Omega_i} \wtens \mathcal O_{\mathbb{C}})
\to \operatorname{H}^1(\mathbb{C};\, \mathbb{C}_{\Omega_i} \wtens \mathcal O_{\mathbb{C}}) \to 0
\]
As $\rho$ has a closed range by the maximum modulus principle for a holomorphic function,
the cohomology group
$
\operatorname{H}^1(\mathbb{C};\, \mathbb{C}_{\Omega_i} \wtens \mathcal O_{\mathbb{C}})
$
becomes a FN space, and hence, we have an isomorphism in $D^b(FN)$
\begin{equation}\label{eq:qi-one-dimension}
	\mathrm{R}\Gamma(\mathbb{C};\, \mathbb{C}_{\Omega_i} \wtens \mathcal O_{\mathbb{C}}) \simeq
\operatorname{H}^1(\mathbb{C};\, \mathbb{C}_{\Omega_i} \wtens \mathcal O_{\mathbb{C}}) [-1].
\end{equation}
Further, we have
\[
\operatorname{H}^k(\mathbb{C}^m;\, \mathbb C_K \wtens
\mathcal{O}_{\mathbb{C}^m}) = 0\qquad (k \ne 0),
\]
and $\operatorname{H}^0(\mathbb{C}^m;\, \mathbb C_K \wtens \mathcal{O}_{\mathbb{C}^m}) $ is a FN space to which
$\mathrm{R}\Gamma(\mathbb{C}^m;\, \mathbb C_K \wtens
\mathcal{O}_{\mathbb{C}^m})$
is isomorphic in $D^b(FN)$.
Hence the claim of the lemma follows from the tensor product formula of Proposition 5.3
\cite{KS96}.
\end{proof}

The following lemma is a key in the proof of the theorem.
\begin{lem}[Martineau]\label{lem:martieaus_lemma}
Let $\Omega_i$ and $\omega_i$ $(i=1,2,\dots,n)$ be non-empty open disks in $\mathbb{C}$
with $\omega_i \subset \Omega_i$.
Set $\Omega=\Omega_1 \times \dots \times \Omega_n$ and
$\omega=\omega_1 \times \dots \times \omega_n$. Then
the canonical morphism associated with inclusion of sets
\[
\iota: \operatorname{H}^n(X;\, \mathbb C_{\omega \times K} \wtens \mathcal{O}_X) \to
\operatorname{H}^n(X;\, \mathbb C_{\Omega \times K} \wtens \mathcal{O}_X)
\]
is injective.
\end{lem}
\begin{oss}
 The above morphism does not have a closed range, and hence,
the cohomology group
$\operatorname{H}^n(X;\, \mathbb C_{(\Omega \setminus \omega) \times K} \wtens \mathcal{O}_X)$ is
not an FN space in general.
\end{oss}
\begin{proof}
We apply the arguments in the proof of Theorem 4.1.6 \cite{Ka88} to our Whitney case.
Let $(z,\,w)$ be the coordinates of $X = \mathbb{C}_z^n \times \mathbb{C}^m_w$.
Set, for $k=1,\dots, n$ and for a subset $\alpha = \{i_1,\dots,i_\ell\}$ in the set $\{1,2,\dots,n\}$,
\[
\begin{aligned}
V_\omega^{(k)} &:= \mathbb{C}_{z_1} \times \dots \times \mathbb{C}_{z_{k-1}} \times
(\mathbb{C}_{z_k} \setminus \omega_k)
\times \mathbb{C}_{z_{k+1}} \times \dots \times \mathbb{C}_n,\\
V_\omega^{(\alpha)} &:= V_\omega^{(i_1)} \cap \dots \cap V_\omega^{(i_\ell)}
\text{ and } V_\omega^{(\emptyset)} := \mathbb{C}^n.
\end{aligned}
\]
We also define $V_\Omega^{(k)}$ and $V^{(\alpha)}_\Omega$ in the same way by replacing $\omega$ with
$\Omega$. Then
$\mathbb{C}_{\omega_k}$ is
isomorphic to the complex, in $D_{\mathbb{R}-c}^b(\mathbb{C})$,
\[
\mathcal{L}_{\omega_k}: 0 \longrightarrow \mathbb{C}_{\mathbb{C}} \longrightarrow
\mathbb{C}_{\mathbb{C} \setminus \omega_k}
\longrightarrow 0.
\]
Similarly we can define the complex $\mathcal{L}_{\Omega_k}$ which is isomorphic to
$\mathbb{C}_{\Omega_k}$  in $D_{\mathbb{R}-c}^b(\mathbb{C})$.
Then the canonical sheaf morphism
$\mathbb{C}_{\mathbb{C} \setminus \omega_k} \to
\mathbb{C}_{\mathbb{C} \setminus \Omega_k}$ induces the morphism of complexes
$\mathcal{L}_{\omega_k} \to \mathcal{L}_{\Omega_k}$, which is  nothing but an extension
of the canonical sheaf morphism $\mathbb{C}_{\omega_k} \to \mathbb{C}_{\Omega_k}$ to the complexes.
Now we have
\[
\mathbb{C}_{\omega \times K} =
\mathbb{C}_{\omega_1}
\underset{\mathbb{C}}{\boxtimes} \cdots \underset{\mathbb{C}}{\boxtimes}
\mathbb{C}_{\omega_n}
\underset{\mathbb{C}}{\boxtimes} \mathbb{C}_K
\simeq
\mathcal{L}_{\omega_1}
\underset{\mathbb{C}}{\boxtimes} \cdots \underset{\mathbb{C}}{\boxtimes}
\mathcal{L}_{\omega_n}
\underset{\mathbb{C}}{\boxtimes} \mathbb{C}_K,
\]
and the last complex is isomorphic to the complex
\[
\mathcal{L}_\omega: 0 \to \underset{|\alpha| = 0} {\oplus}\mathbb{C}_{V_\omega^{(\alpha)} \times K}
\to
\underset{|\alpha| = 1} {\oplus}\mathbb{C}_{V_\omega^{(\alpha)} \times K} \to
\cdots \to
\underset{|\alpha| = n} {\oplus}\mathbb{C}_{V_\omega^{(\alpha)} \times K} \to 0,
\]
where $|\alpha|$ denotes the number of elements of a set $\alpha$.
By the same reasoning, $\mathbb{C}_{\Omega \times K}$ is isomorphic to the complex
\[
\mathcal{L}_\Omega: 0 \to \underset{|\alpha| = 0} {\oplus}\mathbb{C}_{V_\Omega^{(\alpha)} \times K} \to
\underset{|\alpha| = 1} {\oplus}\mathbb{C}_{V_\Omega^{(\alpha)} \times K} \to
\cdots \to
\underset{|\alpha| = n} {\oplus}\mathbb{C}_{V_\Omega^{(\alpha)} \times K} \to 0,
\]
and the canonical sheaf morphism $\mathbb{C}_{V_\omega^{(\alpha)}} \to \mathbb{C}_{V_\Omega^{(\alpha)}}$
induces the one of complexes from $\mathcal{L}_\omega$ to $\mathcal{L}_\Omega$.
This morphism is an extension of the sheaf morphism
$\mathbb{C}_{\omega \times K} \to \mathbb{C}_{\Omega \times K}$ to the complexes.
It follows from the result in \cite{Du79} that we get,
for any $\alpha \subset \{1,2,\dots,n\}$ and for $k \ne 0$,
\[
\operatorname{H}^k(\mathbb{C}^n;\, \mathbb C_{V^{(\alpha)}_\omega} \wtens
\mathcal{O}_{\mathbb{C}^n}) =
\operatorname{H}^k(\mathbb{C}^n;\, \mathbb C_{V^{(\alpha)}_\Omega} \wtens
\mathcal{O}_{\mathbb{C}^n}) =
\operatorname{H}^k(\mathbb{C}^m;\, \mathbb C_{K} \wtens \mathcal{O}_{\mathbb{C}^m}) = 0,
\]
and hence, we obtain
\[
\operatorname{H}^k(X;\, \mathbb C_{V^{(\alpha)}_\omega \times K} \wtens \mathcal{O}_X) =
\operatorname{H}^k(X;\, \mathbb C_{V^{(\alpha)}_\Omega \times K} \wtens \mathcal{O}_X) = 0\quad (k \ne 0).
\]
By these observations, we can conclude that the canonical morphism $\iota$ coincides with
\[
\dfrac{\CHZ(X;\, \mathbb C_{V_\omega \times K} \wtens \mathcal{O}_X)}
{\underset{|\alpha|=n-1}{\oplus}\, \CHZ(X;\, \mathbb C_{V^{(\alpha)}_\omega \times K} \wtens \mathcal{O}_X)}
\to
\dfrac{\CHZ(X;\, \mathbb C_{V_\Omega \times K} \wtens \mathcal{O}_X)}
{\underset{|\alpha| = n-1}{\oplus}\, \CHZ(X;\, \mathbb C_{V^{(\alpha)}_\Omega \times K} \wtens \mathcal{O}_X)},
\]
where the morphism of these cohomology groups is given by the natural restriction and
we set $V_\omega := V_\omega^{(\{1,2,\dots,n\})}$
and $V_\Omega := V_\Omega^{(\{1,2,\dots,n\})}$ for simplicity.

Let us show injectivity of the above $\iota$. We first note that
$\CHZ(X;\, \mathbb C_{V_\omega \times K} \wtens \mathcal{O}_X)$ consists of holomorphic Whitney
jets on $V_\omega \times K$, that is, a holomorphic Whitney jet  is
a family $F(z,w) = \{f_\beta(z,w)\}_{\beta \in \mathbb{Z}_{\ge 0}^{2m}}$
of continuous functions on $V_\omega \times K$ satisfying the conditions:
\begin{enumerate}
\item
Every $f_\beta(z,w)$ is a holomorphic function of $z$ in $V_\omega^\circ$ for each $w \in K$.
Furthermore $\partial_z^\tau f_\beta(z,w)$ ($\tau \in \mathbb{Z}_{\ge 0}^n$) is continuous in
$V_\omega^\circ \times K$ and continuously extends to $V_\omega \times K$.
\item For any $\tau \in \mathbb{Z}_{\ge 0}^n$,
the jet $F_\tau(z,w) := \{\partial_z^\tau f_\beta(z,w)\}_{\beta \in \mathbb{Z}_{\ge 0}^{2m}}$
satisfies Whitney's remainder term estimate with respect to variables $\operatorname{Re}w$  and
$\operatorname{Im}w$ (i.e., with respect to indices $\beta \in \mathbb{Z}_{\ge 0}^{2m}$)
which is locally uniform with respect to $z \in V_\omega$.
It also satisfies $\overline{\partial}_{w_k} F = 0$ for $1 \le k \le m$.
\end{enumerate}
Let $F(z,w) = \{f_\beta(z,w)\}_\beta \in \CHZ(X;\, \mathbb C_{V_\omega \times K} \wtens \mathcal{O}_X)$
be a holomorphic Whitney jet on $V_\omega \times K$.
Let us define
\[
G(z,w) := \frac{1}{(2\pi\sqrt{-1})^n}\int_{\partial \omega_1 \times \dots \times \partial \omega_n}
\frac{F(\zeta,w)}{(\zeta_1 - z_1)\dots(\zeta_n - z_n)} d\zeta,
\]
where $\partial \omega_k$ is the boundary of $\omega_k$ with the clockwise orientation.
Then, by repetition of the integration by parts, we can  easily confirm that $G(z,w)$ satisfies
the conditions 1.~and 2.~above, and hence, we obtain
$G(z,w) \in \CHZ(X;\, \mathbb C_{V_\omega} \wtens \mathcal{O}_X)$.
Let $D$ be a sufficiently large disk in $\mathbb{C}$ satisfying $z \in \operatorname{int} D^n$.
It follows from Cauchy's integral formula that we get
\[
F(z,w) = \frac{1}{(2\pi\sqrt{-1})^n}\int_{(\partial \omega_1 - \partial D) \times \dots \times
(\partial \omega_n - \partial D)}
\frac{F(\zeta,w)}{(\zeta_1 - z_1)\dots(\zeta_n - z_n)} d\zeta.
\]
Hence $F(z,w) - G(z,w)$ is a sum of integrals of the form
\[
\frac{1}{(2\pi\sqrt{-1})^n}\int_{\gamma_1 \times \dots \times \gamma_n}
\frac{F(\zeta,w)}{(\zeta_1 - z_1)\dots(\zeta_n - z_n)} d\zeta,
\]
where $\gamma_k$ is either $\partial \omega_k$ or $-\partial D$.
Furthermore $\gamma_k = -\partial D$ holds at least one of $k$'s in
the above $\gamma_1\times \dots \times \gamma_n$.
Hence
these integrals are zero in $\operatorname{H}^n(X;\, \mathbb C_{\omega \times K} \wtens \mathcal{O}_X)$.
As a conclusion, the holomorphic Whitney jets $F(z,w)$ and $G(z,w)$ determine
the same cohomology class in
$\operatorname{H}^n(X;\, \mathbb C_{\omega \times K} \wtens \mathcal{O}_X)$.

Now assume $\iota(F(z,w)) = 0$. Then, by deforming the path of the integration to
$\partial \Omega_1 \times \dots \times \partial \Omega_n$,
we have $G(z,w) = 0$ on $V_\Omega \times K$
since $F(z,w)$ belongs to
${\underset{|\alpha|=n-1}{\oplus}\, \CHZ(X;\, \mathbb C_{V^{(\alpha)}_\Omega \times K} \wtens \mathcal{O}_X)}$,
i.e., a sum of holomorphic Whitney jets which are holomorphic on the entire $\mathbb{C}$
with respect to some variable $z_k$.
Hence $G(z,w)$ = 0 on $V_\omega \times K$
follows from the unique continuation property of $G(z,w)$ with respect to the variables $z$.
This show the injectivity of $\iota$.
\end{proof}

As an immediate consequence of the above lemma, we obtain the following.
\begin{lem}\label{lem:vanishing-product}
Let $\Omega$ and $\omega$ be the same as those given in the previous lemma. Then we have
\[
\operatorname{H}^k(X;\, \mathbb C_{(\Omega \setminus \omega) \times K} \wtens \mathcal{O}_X) = 0
\qquad (k < n).
\]
\end{lem}
\begin{proof}

From the previous two lemmas and the following long exact sequence, the result easily follows.
\[
\begin{aligned}
&\to \operatorname{H}^k(X;\, \mathbb C_{ \omega \times K} \wtens \mathcal{O}_X)
\to \operatorname{H}^k(X;\, \mathbb C_{\Omega \times K} \wtens \mathcal{O}_X) \\
&\qquad \to \operatorname{H}^k(X;\, \mathbb C_{(\Omega \setminus \omega) \times K}
\wtens \mathcal{O}_X)
\to \operatorname{H}^{k+1}(X;\, \mathbb C_{ \omega \times K} \wtens \mathcal{O}_X) \to.
\end{aligned}
\]
\end{proof}

\noindent \textbf{Proof of Theorem \ref{th:vanishing-main}.}\ \ First we show the claim
\[
\operatorname{H}^k(X;\, \mathbb C_{(\Omega\setminus \omega) \times K} \wtens \mathcal{O}_X) = 0
\quad (k > n).
\]
We have the isomorphism in $D^b(FN)$
\[
\mathrm{R}\Gamma(\mathbb{C}^m;\, \mathbb C_{K} \wtens \mathcal{O}_{\mathbb{C}^m})
\simeq
\CHZ(\mathbb{C}^m;\, \mathbb C_{K} \wtens \mathcal{O}_{\mathbb{C}^m}) .
\]
Furthermore, in $D^b(FN)$, it follows from the definition that the object
$\mathrm{R}\Gamma(\mathbb{C}^n;\, \mathbb{C}_{\Omega\setminus \omega} \wtens \mathcal{O}_{\mathbb{C}^n}) $
is isomorphic to the complex of FN spaces of length $n$
\[
\mathcal{L}:\,\, 0 \to \Gamma(\mathbb{C}^n;\,
\mathbb{C}_{\Omega\setminus \omega} \wtens \mathcal{C}_{\mathbb{R}^{2n}}^{\infty, (0,0)})
\overset{\overline{\partial}}
\to
\dots
\overset{\overline{\partial}}{\to}
\Gamma(\mathbb{C}^n;\, \mathbb{C}_{\Omega\setminus \omega} \wtens
\mathcal{C}_{\mathbb{R}^{2n}}^{\infty,(0,n)})
\to 0.
\]
Then, by Proposition 5.3 of \cite{KS96}, the object
$\mathrm{R}\Gamma(X;\, \mathbb C_{(\Omega\setminus \omega) \times K} \wtens \mathcal{O}_X)$
is isomorphic to the complex
$
\mathcal{L}\, \underset{\mathbb{C}}{\widehat{\boxtimes}}\,
\operatorname{H}^{0}(\mathbb{C}^m;\,
\mathbb{C}_K \wtens \mathcal{O}_{\mathbb{C}^{m}})
$
of FN spaces of length $n$. Hence we have obtained the claim.

Now we show, for $k < n$,  the $k$-th cohomology group vanishes.
We may assume $\omega \subset \Omega$.
Then there are holomorphic functions $f_1$, $\dots$, $f_{\ell'}$, $\dots$, $f_{\ell}$ on $\mathbb{C}^n$
such that
\[
\Omega = \{z \in \mathbb{C}^n;\, |f_1(z)| < 1,\, \dots, |f_{\ell'}(z)| < 1\}
\]
and
\[
\omega = \{z \in \mathbb{C}^n;\, |f_1(z)| < 1,\, \dots, |f_{\ell'}(z)| < 1, \dots, |f_{\ell}(z)| < 1\}.
\]
As $\Omega$ is relative compact, there exists $R > 0$ such that
\[
\Omega \subset \{z \in \mathbb{C}^n;\,|z_1|<R,\dots, |z_n|<R\},\quad
|f_k(\Omega)| < R\quad  (k=1,2,\dots,\ell).
\]

We now apply the well-known method due to Oka to the following situation.
Let us consider the closed embedding
$\varphi: X=\mathbb{C}_{z,w}^{n+m} \to \mathbb{C}_{z,\tau, w}^{n+\ell+m}$ defined by
\[
\varphi(z,w) = (z, f_1(z), \,\dots, \,f_{\ell'}(z),\,\dots, \,f_{\ell}(z), w).
\]
Set $\widetilde{X} := \mathbb{C}_{z,\tau,w}^{n+\ell + m} =
\mathbb{C}^{n + \ell}_{z,\tau} \times \mathbb{C}_w^m$.
We also define $\widetilde{\Omega}$ and $\widetilde{\omega}$ in $\mathbb{C}_{z,\tau}^{n + \ell}$ by
\[
\widetilde{\Omega} := \left\{(z,\tau) \in \mathbb{C}^{n+\ell};\,
\begin{aligned}
&|z_1|<R,\dots,|z_n| <R,\,  |\tau_1| < 1, \dots,|\tau_{\ell'}| <1,\, \\
&|\tau_{\ell'+1}| < R,\,\dots,\, |\tau_{\ell}| < R
\end{aligned}
\right\}
\]
and
\[
\widetilde{\omega} := \left\{(z,\tau) \in \mathbb{C}^{n+\ell};\,
\begin{aligned}
&|z_1|<R,\dots,|z_n| <R,\, |\tau_1| < 1, \dots,|\tau_{\ell'}| <1,\, \\
&|\tau_{\ell'+1}| < 1,\,\dots,\, |\tau_{\ell}| < 1
\end{aligned}
\right\}.
\]
Noticing
$(\Omega \setminus \omega) \times K= \varphi^{-1}
((\widetilde{\Omega} \setminus \widetilde{\omega}) \times K)$,
by Theorem 5.8 (ii) in \cite{KS96}, we have
\[
\mathrm{R}\Gamma(X;\, \mathbb C_{(\Omega \setminus \omega) \times K} \wtens {\mathcal O}_{X}) =
\mathrm{R}\Gamma(X;\, \underline{\varphi}^{-1}
(\mathbb C_{(\widetilde{\Omega} \setminus \widetilde{\omega}) \times K}
\wtens {\mathcal O}_{ \widetilde{X}})).
\]
Let $\mathcal M$ be the right $\mathcal{D}_{\widetilde{X}}$ module
$
{\mathcal D}_{\widetilde{X}}/
(\tau_1 -f_1(z), \dots, \tau_\ell - f_\ell(z)) {\mathcal D}_{\widetilde{X}}
$.
Then, as $\mathcal M$ is defined globally on $\widetilde{X}$ and its support is
contained in the graph of $\varphi$, we have
\[
\mathrm{R}\Gamma(X;\,
\underline{\varphi}^{-1}(\mathbb C_{(\widetilde{\Omega} \setminus \widetilde{\omega})\times K} \wtens
{\mathcal O}_{\widetilde{X}}))
\simeq
\mathrm{R}\Gamma(\widetilde{X};\, \mathcal M \underset{\mathcal D_{\widetilde{X}}}
{\overset{L}{\otimes}}
(\mathbb C_{(\widetilde{\Omega} \setminus \widetilde{\omega})\times K} \wtens {\mathcal O}_{\widetilde{X}})).
\]
Define the $\mathfrak{D}:=\Gamma(\widetilde{X};\, \mathcal{D}_{\widetilde{X}})$-module
$
\mathfrak{M}
$
by
$
\mathfrak{D} /
(\tau_1 -f_1(z), \dots, \tau_\ell - f_\ell(z)) \mathfrak{D}
$.
Then we obtain
\[
\mathrm{R}\Gamma(\widetilde{X};\, \mathcal{M}
\underset{\mathcal D_{\widetilde{X} }}{\overset{L}{\otimes}}
(\mathbb C_{(\widetilde{\Omega} \setminus \widetilde{\omega})\times K} \wtens
{\mathcal O}_{ \widetilde{X}}))
\simeq
\mathfrak{M} \underset{\mathfrak{D}}{\overset{L}{\otimes}}
\mathrm{R}\Gamma(\widetilde{X};\,
\mathbb C_{(\widetilde{\Omega} \setminus \widetilde{\omega})\times K}
\wtens {\mathcal O}_{ \widetilde{X}}).
\]
Since we have
$\operatorname{H}^k(\widetilde{X};\,
\mathbb C_{(\widetilde{\Omega} \setminus \widetilde{\omega})\times K}
\wtens {\mathcal O}_{ \widetilde{X}}) = 0$
for $k < n + \ell$ by Lemma \ref{lem:vanishing-product}, and since $\mathfrak{M}$ has a free
$\mathfrak{D}$
resolution of the length $\ell$, we obtain, for $k < (n+\ell) - \ell = n$,
\[
\operatorname{H}^k(
\mathfrak{M} \underset{\mathfrak{D}}{\overset{L}{\otimes}}
\mathrm{R}\Gamma(\widetilde{X};\,
\mathbb C_{(\widetilde{\Omega} \setminus \widetilde{\omega}) \times K}
\wtens {\mathcal O}_{ \widetilde{X}})) = 0.
\]
This completes the proof.
\nopagebreak \phantom{} \hfill $\square$ \\

\begin{cor}\label{cor:vanishing-main}
Let $\Omega$, $\omega$ and $K$ be the same as those in Theorem \ref{th:vanishing-main}, and
let $g: \mathbb{C}^n \to \mathbb{C}^d$ be a holomorphic map and $L$ a closed analytic
polyhedron in $\mathbb{C}^d$. Then we have,  for $k < n - d$,
\[
\operatorname{H}^k(X;\, \mathbb{C}_{((\Omega\setminus \omega) \cap g^{-1}(L) )\times K} \wtens \mathcal{O}_X) = 0.
\]
\end{cor}
\begin{proof}
The proof goes in the same way as that for Theorem \ref{th:vanishing-main}.
In this case,  we apply Oka's method to the closed embedding
$\varphi: \mathbb{C}^{n+m} \to \mathbb{C}^{n+\ell+d+m} =: \widetilde{X}$ defined by
\[
\varphi(z,w) := (z,\,f_1(z), \,\dots,\, f_\ell(z),\, g(z),\, w).
\]
Then $((\Omega\setminus \omega) \cap g^{-1}(L)) \times K =
\varphi^{-1}((\widetilde{\Omega}\setminus\widetilde{\omega}) \times L \times K)$ holds.
It follows from Lemma \ref{lem:vanishing-product} that we obtain,  for $k < n + \ell$,
\[
\operatorname{H}^k(\widetilde{X};\,
\mathbb{C}_{(\tilde{\Omega}\setminus\tilde{\omega}) \times L \times K} \wtens
\mathcal{O}_{\widetilde{X}})
= 0.
\]
Hence,  by noticing that the module $\mathfrak{M}$ has a free resolution of
length $\ell + d$, we have the conclusion.
\end{proof}

The following theorem is well-known, which can be also proved by the same method as that in the proof of
Theorem \ref{th:vanishing-main}.
\begin{teo}
	Let $K_1$ and $K_2$ be analytic compact polyhedra in $\mathbb{C}^n$
	and let $U$ be a Stein open subanalytic subset in $\mathbb{C}^m$. Then we have
	\[
	\operatorname{H}^k_{(K_1 \setminus K_2) \times U}
        (X;\, \mathcal{F}) = 0 \quad (k \ne n),
	\]
	where $\mathcal{F}$ is $\mathcal{O}^t_{X_{sa}}$ or $\mathcal{O}_X$.
\end{teo}

\subsection{Some geometrical preparations}
Let $X$ be $\mathbb{R}^n$ with coordinates $(x_1,\dots,x_n)$ and $G$
a closed conic subset in $X$.
Recall that $G$ is said to be a proper cone with respect to direction $\xi \ne 0$
if $G \setminus \{0\} \subset \{x \in X;\, \langle x, \xi\rangle > 0\}$ holds.
Note that the cone $G = \{0\}$ is proper with respect to all the directions $\xi \ne 0$.
We also say that $G$ is a linear cone if
there exist vectors
$\xi_1$, $\dots$, $\xi_k$ such that
\[
G = \underset{1 \le i \le k}{\bigcap} \{x \in X;\, \langle x,\, \xi_i \rangle \ge 0\}.
\]
If $G$ is a proper cone, then $G$ is linear if and only if
there exist vectors
$\xi_1$, $\dots$, $\xi_{k'}$ such that
\[
G = \mathbb{R}_{\ge 0} \xi_1 + \dots + \mathbb{R}_{\ge 0} \xi_{k'}
\,(=
\overline{\mathbb{R}_{\ge 0} \xi_1 + \dots + \mathbb{R}_{\ge 0} \xi_{k'}}).
\]
From the above equivalence the following lemma immediately follows.
\begin{lem}{\label{lem:cone-sum-proper}}
Let $G_j$ be closed subsets in $X$ ($j=1,\dots,\ell$).
Assume that $G_j$'s are linear proper cones with respect to
the same direction $\xi \ne 0$.
Then $G_1 + \dots + G_\ell$ is a closed
linear proper cone with respect to the direction $\xi$.
\end{lem}

\

Now let us back to the situation of the fiber formula in \S\, \ref{sec:fiber-formula}.
Let $I_j$ be subsets in $\{1,\dots,n\}$ satisfying the conditions
(\ref{eq:cond-H}).  Recall that $x^{(j)}$ denotes the coordinates $x_i$'s with $i \in \widehat{I}_j$
and $\xi^{(j)}$ designates the dual coordinates of $x^{(j)}$ ($j=1,\dots,\ell$).
We also denote by $x^{(0)}$ the coordinates $(x_i)$ with
$i \in \{1,\dots,n\} \setminus I$ where we set $I := \cup_j I_j$.
Then the coordinates of $S^*_{\chi}$ is given by
$
(x^{(0)};\, \xi^{(1)},\, \dots,\, \xi^{(\ell)}).
$

Let $p = (q;\,\xi) = (x^{(0)};\, \xi^{(1)},\dots, \xi^{(\ell)})$ be a point in $S^*_{\chi}$.
Recall also the definitions of $J_{\prec j}$, $J_{\succ j}$ and $\gamma_j$
given in the fiber formula
which are often used in subsequent arguments.
Now we define the subset $J^*(\xi)$ of $\{1,\dots,\ell\}$ by
\begin{equation}
\{j \in \{1,\dots,\ell\};\, \text{ $\xi^{(\alpha)} = 0$ for any $\alpha \in J_{\preceq j}$}\}.
\end{equation}
Here $J_{\preceq j}$ denotes $J_{\prec j} \cup \{j\}$.
In what follows, we suppose $I = \{1,\dots,n\}$ for simplicity. Hence
the coordinates of $S^*_{\chi}$ is given by $(\xi^{(1)},\, \dots,\, \xi^{(\ell)})$.
Let $\xi= (\xi^{(1)},\, \dots,\, \xi^{(\ell)})$ be a point in $S^*_{\chi}$.
Set
\begin{equation}
L(\xi) := \underset{j \in J^*(\xi)}{\bigcap}\{x \in \mathbb{R}^n;\, x^{(j)} = 0\}.
\end{equation}

\begin{lem}\label{lem:linearity_of_G}
Let $G_j$ $(j=1,\dots,\ell)$ be a closed linear cone with
$G_j \setminus \{0\} \subset \gamma_j$ which has
an interior point in $\gamma_j$
if $\gamma_j$ is non-empty $(j=1,\dots,\ell)$.
Set
$
G := G_1 + \dots + G_\ell \subset \mathbb{R}^n
$.
Then we have the followings.
\begin{enumerate}
\item $G$ is a closed linear cone which is proper with respect to
	some direction $\xi_G$.
\item We have $G \subset L(\xi)$ and
	$G = \overline{\operatorname{Int}_{L(\xi)}(G)}$. In particular,
\[
\rh_{\mathbb{C}_X}\left(\mathbb{C}_G,\, \mathbb{C}_{X}\right)
= \mathbb{C}_{\operatorname{Int}_{L(\xi)}(G)}[-\operatorname{codim}_{\mathbb{R}}(L(\xi))]
\]
holds where $\operatorname{Int}_{L(\xi)}(G)$ denotes the set of interior points
of $G$ in $L(\xi)$.
\end{enumerate}
\end{lem}
\begin{proof}
Let us show the claim 1.~of the lemma.
Recall the definition of the cone $V_j$ which appeared
in the proof of Theorem \ref{th:stalk formula}.
If $\xi^{(j)} \ne 0$, then $V_j \subset \mathbb{R}^n$ has the form
\[
\{x = (x^{(k)}) \in X;\, x^{(j)} \in T_j,\,
\delta |\langle x^{(j)}, \xi^{(j)} \rangle| \ge \sum_{k \in J_{\succ j}}|x^{(k)}|\}
\]
where $\delta > 0$ and $T_j \subset {\mathbb{R}}^{n_j}$ is a proper closed convex cone with
\[
T_j \subset \{x^{(j)} \in \mathbb{R}^{n_j};\,\langle x^{(j)},\, \xi^{(j)} \rangle \ge 0\} \text{ and }
\xi^{(j)} \in \operatorname{Int}_{\mathbb{R}^{n_j}} T_j.
\]
If $\xi^{(j)} = 0$, then we set $V_j := \mathbb{R}^n$.
It follows from the definitions of $G_j$ and $\gamma_j$ that
$G_j$ is contained in the polar cone $V_j^\circ$ for an appropriate $V_j$.

Let $\#J_{\succ j}$ denote the number of elements in $J_{\succ j}$.
For $\sigma > 0$, we determine the positive real number $\sigma_j$ by $\sigma^{\#J_{\succ j}}$.
Now we define the vector by
\begin{equation}
\xi_G := (\sigma_1 \xi^{(1)},\, \sigma_2 \xi^{(2)},\, \dots,\, \sigma_\ell \xi^{(\ell)}).
\end{equation}
Then,  as $\#J_{\succ \alpha} < \#J_{\succ j}$ holds for $\alpha \in J_{\succ j}$,
it follows the definition of $V_j$ that the vector $\xi_G$ belongs to the interior of $V_j$ for any $j$
by taking $\sigma$ sufficiently large. Hence, since $G_j$ is contained $V_j^\circ$, we see that
$G_j \setminus \{0\}$ is contained in $\{x \in \mathbb{R}^n;\, \langle x, \xi_G \rangle > 0\}$.
Then the first claim is a consequence of Lemma \ref{lem:cone-sum-proper}.

Next we show the claim 2.~of the lemma. By the definition of $\gamma_j$, we
can easily see that $G \subset L(\xi)$ and $G$ has an interior point in $L(\xi)$.
Then, as $G$ is convex, the rest of the claim follows from these.
\end{proof}

\subsection{The result for some family of real analytic submanifolds}
Let $X=\mathbb{C}^n$ with coordinates
$
(z_1 = x_1 + \sqrt{-1}y_1,\, \dots,\, z_n = x_n + \sqrt{-1}y_n)
$
and $\zeta_i = \xi_i + \sqrt{-1}\eta_i$ ($i=1,\dots,n$) the
dual variable of $z_i = x_i + \sqrt{-1}y_i$.
Let $I_j$ ($j=1,2,\dots,\ell$) be a subset of $\{1,\,2,\,\dots,\, n\}$ which
satisfies the conditions (\ref{eq:cond-H}),  and set $I_0$ by (\ref{eq:def-I_0}).
Let $I_{\mathbb{R}}$ be a subset of $I := \underset{1 \le j \le \ell}{\cup}I_j$ and
$I_{\mathbb{C}} := I \setminus I_{\mathbb{R}}$.
Define, for $i \in I$, the function $q_i(z)$ in $X$ by
\[
q_i(z) :=
\left\{
\begin{array}{ll}
	z_i \qquad& (i \in I_{\mathbb{C}}),\\
	\sqrt{-1}\operatorname{Im} z_i \qquad & (i \in I_{\mathbb{R}}).
\end{array}
\right.
\]
Then we define the closed real analytic submanifolds
\[
N_j := \{z \in X;\, q_i(z) = 0,\, i \in I_j\} \qquad (j=1,\dots,\ell),
\]
and set
\[
\chi := \{N_1, \dots, N_\ell\}, \qquad
N = N_1 \cap \dots \cap N_\ell.
\]

In what follows, we regard the function $q_i$ as the complex coordinate variable
$z_i$ if $i \in I_{\mathbb{C}}$ and as the imaginary coordinate variable $\sqrt{-1}y_i$
if $i \in I_{\mathbb{R}}$. In the same way, $p_i$ is regraded as
the dual variable of $q_i$, that is, $p_i$ denotes $\zeta_i$ if $i \in I_{\mathbb{C}}$
and $\sqrt{-1}\eta_i$ if $i \in I_{\mathbb{R}}$.
As usual convention, we write by $q^{(j)}$ (resp. $p^{(j)}$) the coordinates $q_i$'s (resp. $p_i$'s) with
$i \in \widehat{I}_j$.
Under these conventions,  the coordinates of $S^*_{\chi}$ are given by
\[
(q^{(0)};\, p^{(1)},\,\dots,\, p^{(\ell)}),
\]
where $q^{(0)}$ denotes the set of the coordinate variables $z_i$'s
($i \in I_0$) and $x_i$'s ($i \in I_{\mathbb{R}}$).
Let $\theta_* = (q_*;p_*) = (q_*^{(0)};\, p_*^{(1)},\,\dots,\, p_*^{(\ell)}) \in S^*_{\chi}$.
Recall the definition of $J^*(\theta_*)$, that is,
\[
\begin{aligned}
J^*(\theta_*) &:= \{j \in \{1,\dots,\ell\};\, p_*^{(\alpha)} = 0 \text{ for all $\alpha \in J_{\preceq j}$}\} \\
   & =
   \{j \in \{1,\dots,\ell\};\, p_*^{(\alpha)} = 0
   \text{ for all $\alpha$ with $N_j \subset N_\alpha$}\}. \\
\end{aligned}
\]
We set
\begin{equation}
I^*(\theta_*) := \underset{j \in J^*(\theta_*)}{\bigcup} \widehat{I}_j \,\, \subset \{1,\dots,n\}.
\end{equation}
Then we define the integer $N(\theta_*)$ by
\begin{equation}
N(\theta_*) = \#I + \# (I^*(\theta_*) \cap I_{\mathbb{C}}),
\end{equation}
where $\#$ denotes the number of elements in a set.
Note that $\#I$ is equal to $\operatorname{Codim}_{\mathbb{C}}N$, i.e.,
the complex codimension of the maximal complex linear subspace contained in $N$.
\begin{teo}\label{th:vanishing-for-mm}
We have
\[
\operatorname{H}^k(\mu_\chi(\OW_{X_{sa}}))_{\theta_*} = 0\qquad (k \ne N(\theta_*)).
\]
\end{teo}
\begin{proof}
We may assume $q_* = 0$. Furthermore, by a complex rotation of each variable $z_i$
($i \in I_{\mathbb{C}}$),
we also assume that $p_*^{(j)}$
is purely imaginary or zero for $j=1,\dots,\ell$. Set
\[
J := \{1,\dots,\ell\},\quad
J' := J \setminus J^*(\theta_*),\quad J'' :=J^*(\theta_*).
\]
Recall $n_j = \#\widehat{I}_j$ and set
\[
m_1 := \sum_{j \in J'} n_j, \quad
m_2 := \sum_{j \in J''} n_j, \quad m_3 = n - m_1 - m_2.
\]
Note that, because of $\sqcup_{j \in J} \widehat{I}_j = I$,  we get $m_1 + m_2 = \# I$.
Then we have $X = \mathbb{C}^{m_1}_{z'} \times
\mathbb{C}^{m_2}_{z''} \times \mathbb{C}^{m_3}_{z'''}$ where
the coordinates $(z',z'',z''')$ are given by
\[
(z',\, z'',\, z''') = ( \{z^{(j)}\}_{j \in J'},\, \{z^{(j)}\}_{j \in J''},\, z^{(0)}).
\]
Here $z^{(j)}$ (resp. $z^{(0)}$) denotes, as usual convention,
the coordinates $z_i$'s with $i \in \widehat{I}_j$
(resp. $i \in I_0$).
We also set
\[
(q',\, q'') = (\{q^{(j)}\}_{j \in J'},\, \{q^{(j)}\}_{j \in J''})
\]
and
\[
\begin{aligned}
(p',\, p'') &= (\{p^{(j)}\}_{j \in J'},\, \{p^{(j)}\}_{j \in J''}),\,   \\
(p'_*,\, p''_*) &= (\{p_*^{(j)}\}_{j \in J'},\, \{p_*^{(j)}\}_{j \in J''}).
\end{aligned}
\]
Note that, by the definition of $J^*(\theta_*)$, it follows that $p''_* = 0$ and $p'_* \ne 0$. Furthermore
$p'_*$ is purely imaginary by the assumption.
Then define the partially complex linear subspaces $L \subset \mathbb{C}^{m_1}_{z'}$ and $Z \subset \mathbb{C}^{m_2}_{z''}$
by
\[
\begin{aligned}
L &:= \text{The linear subspace of $\mathbb{C}^{m_1}_{z'}$ spanned by the vectors $dq'$,} \\
Z &:= \{z'' \in \mathbb{C}^{m_2}_{z''};\, q'' = 0\}.
\end{aligned}
\]
Clearly the coordinates of $L$ are given by $q'$.
We denote by $L^\perp$ the orthogonal linear subspace of $L$ in $\mathbb{R}^{2m_1} = \mathbb{C}^{m_1}$.
Then we have $\mathbb{C}^{m_1} = L^\perp \times L$ and
\[
L^\perp \subset \mathbb{R}^{m_1} = \mathbb{R}^{m_1} \times \{0\}
\subset \mathbb{R}^{m_1} \times \sqrt{-1} \mathbb{R}^{m_1} = \mathbb{C}^{m_1}
\]
as $q_i$ is either $z_i$ or $\sqrt{-1}y_i$.
Now let us define the closed set $\gamma_j$ ($j \in J'$) in $L$ which corresponds
to the one in the fiber formula.
\begin{equation}
\gamma_j := \left\{q' = (q^{(k)})_{k \in J'} \in L;\,
\begin{array}{ll}
	q^{(k)}=0 &(k \in J' \text{ with } k \in J_{\prec j} \sqcup J_{\nparallel\, j}) \\
	\operatorname{Re}\langle q^{(k)}, p^{(k)}_* \rangle > 0 & (k=j)
\end{array}
\right\}.
\end{equation}
It follows from the fiber formula that we have
\[
\operatorname{H}^k(\mu_\chi(\OW_{X_{sa}}))_{\theta_*} =
\lind{G,\,U_1,\,U_2,\,U_3}
\operatorname{H}^k_{(L^{\perp} \times G) \times
Z \times \mathbb{C}^{m_3}}
\left(U_1 \times U_2 \times U_3;\, \OW_{X_{sa}}\right).
\]
Here $U_i$ runs through a family of open convex subanalytic neighborhoods
of the origin in $\mathbb{C}^{m_i}$
($i=1,2,3$) and $G \subset L$
runs through a family of closed subanalytic cones of the form
$
G := \sum_{j \in J'} G_j
$
with a subanalytic closed cone $G_j \subset L$ satisfying $G_j \setminus \{0\} \subset \gamma_j$.
Here we may assume that each $G_j$ is a proper linear cone in $L$ and it
has an interior point in $\gamma_j$.
Hence, by Lemma \ref{lem:linearity_of_G},
the $G$ is a closed proper linear cone in $L$ with some direction.
Furthermore, the $L^{\perp} \times G$ and
$L^{\perp} \times \operatorname{Int}_L(G)$ are analytic polyhedra
in $\mathbb{C}^{m_1}_{z'}$ and
the latter one is a non-empty open cone.
By the claim 2.~of the same lemma, we have
\[
\operatorname{H}^k_{(L^\perp \times G) \times Z
\times \mathbb{C}^{m_3}}\left(U_1 \times U_2 \times U_3;\, \OW_{X_{sa}}\right)
= \operatorname{H}^{k- m'_2}
\left(X;\, \mathbb{C}_{W(G,U_1,U_2,U_3)} \wtens \mathcal{O}_X\right),
\]
where
\begin{equation}{\label{eq:def_m_2_dash}}
m'_2 := m_2 + \# (I^*(\theta_*) \cap I_{\mathbb{C}})
\end{equation}
and
\[
W(G,U_1,U_2,U_3) :=
( (L^\perp \times \operatorname{Int}_L(G)) \cap \overline{U_1})
\times (Z \cap \overline{U_2})
\times \overline{U_3}.
\]
From now on, we also use as coordinates variables of $\mathbb{C}^{m_1}_{z'}$
\[
z' = (z_{i_1}, \dots, z_{i_{m_1}}) = (x_{i_1} + \sqrt{-1}y_{i_1}, \dots, x_{i_{m_1}} + \sqrt{-1}y_{i_{m_1}}).
\]
Furthermore, we identify
$\mathbb{R}^{m_1}$ and $\sqrt{-1}\mathbb{R}^{m_1}$
with the linear
subspaces
$\mathbb{R}^{m_1} \times \sqrt{-1}\{0\}$ and
$\{0\} \times \sqrt{-1}\mathbb{R}^{m_1}$
in $\mathbb{C}^{m_1} = \mathbb{R}^{m_1} \times \sqrt{-1}\mathbb{R}^{m_1}$ respectively.
Since $p_*' \ne 0$ and it is purely imaginary as we have already noted,
the cone
\[
\widetilde{G}:= \dfrac{1}{\sqrt{-1}}(G \cap \sqrt{-1}\mathbb{R}^{m_1})
\]
is proper and linear in $\mathbb{R}^{m_1}$.
Furthermore, as $p_*'$ is purely imaginary, we may assume
\[
G \subset (L \cap \mathbb{R}^{m_1}) \times \sqrt{-1}\widetilde{G} \subset L.
\]
By Lemma \ref{lem:linearity_of_G}, we choose
a vector $\xi_{\widetilde{G}} \in \mathbb{R}^{m_1}$
with $|\xi_{\widetilde{G}}| = 1$ satisfying
\[
\widetilde{G} \subset \{x' \in \mathbb{R}^{m_1};\, \langle x',  \xi_{\widetilde{G}} \rangle > 0\} \cup \{0\}.
\]
Let us define the holomorphic function on $\mathbb{C}^{m_1}_{z'}$ by
\[
\varphi(z') :=  \langle z', \xi_{\widetilde{G}}\rangle
+ \sqrt{-1} \displaystyle\sum_{1 \le k \le m_1} z_{i_k}^2.
\]
Let $a>0$ and $\epsilon$ be sufficiently small positive real numbers,  and we set
\[
\begin{aligned}
	K &:=  (Z \cap \overline{U_2}) \times \overline{U_3} \subset
\mathbb{C}^{m_2} \times \mathbb{C}^{m_3},\\
\Omega &:= (L^\perp \times \operatorname{Int}_L(G)) \cap
\left\{z' \in \mathbb{C}^{m_1};\,
|z_{i_k}| < a\, (1 \le k \le m_1) \right\}, \\
\omega &:= (L^\perp \times \operatorname{Int}_L(G)) \cap
\left\{z' \in \mathbb{C}^{m_1};\,
\operatorname{Im} \varphi(z') > \epsilon,\,
|z_{i_k}| < a\, (1 \le k \le m_1) \right\}. \\
\end{aligned}
\]
We take an open poly-disks with center at the origin as $U_i$ ($i=1,2,3$).
Then $K$ is a convex closed analytic polyhedron in $\mathbb{C}^{m_2} \times \mathbb{C}^{m_3}$, and
$\Omega$ and $\omega$ are relatively compact
and analytic open polyhedra in $\mathbb{C}^{m_1}$.
Hence it follows from Theorem \ref{th:vanishing-main} that we have
\begin{equation}{\label{eq:vanishing_omega_Omega}}
\operatorname{H}^k(X;\, \mathbb{C}_{(\Omega \setminus \omega) \times K}\wtens \mathcal{O}_X)
= 0\qquad (k \ne m_1).
\end{equation}
Now, by applying the same reasoning as that in the proof of Theorem 2.2.2 \cite{KKK},  we
obtain, for each $G$ and a sufficiently small $a > 0$,
\begin{equation}{\label{eq:limit_G_omega_equive}}
\begin{aligned}
&\lind{U_1,\,U_2,\,U_3}
\operatorname{H}^{k- m'_2}
\left(X;\, \mathbb{C}_{W(G,U_1,U_2,U_3)} \wtens \mathcal{O}_X\right) \\
&\qquad=
\lind{\epsilon > 0,\, U_2,\,U_3}
\operatorname{H}^{k - m'_2}(X;\, \mathbb{C}_{(\Omega \setminus \omega) \times K}\wtens
\mathcal{O}_X).
\end{aligned}
\end{equation}
As a matter of fact,
as $\widetilde{G}$ is proper with respect to the direction $\xi_{\widetilde{G}}$,  we get
\[
\widetilde{G} \subset
\{y' \in \mathbb{R}^{m_1};
| y' - \langle y', \xi_{\widetilde{G}} \rangle \xi_{\widetilde{G}}| \le
 \sigma \langle y', \xi_{\widetilde{G}} \rangle \}
\]
for some $\sigma > 0$.
Hence there exists $a_\sigma > 0$ depending only on $\sigma$ such that
\[
\sum_{1 \le k \le m_1} y_{i_k}^2 \le \dfrac{1}{2} \langle y', \xi_{\widetilde{G}} \rangle
\quad (|y_{i_1}|,\dots,|y_{i_m}| \le a_\sigma,\,\, y' \in \widetilde{G}).
\]
Then, by noticing
\[
\operatorname{Im} \varphi(z') =
 \langle y', \xi_{\widetilde{G}}\rangle
+ \displaystyle\sum_{1 \le k \le m_1} ({x_{i_k}}^2 - {y_{i_k}}^2),
\]
we obtain
\[
\dfrac{1}{2} \langle y', \xi_{\widetilde{G}}\rangle
+ \displaystyle\sum_{1 \le k \le m_1} x_{i_k}^2 \le \epsilon \quad
(0 < a \le a_\sigma,\, z' \in \Omega \setminus \omega).
\]
Therefore, if we take $\epsilon > 0$ sufficiently small,  by noticing that
$\widetilde{G}$ is proper with respect to the direction $\xi_{\widetilde{G}}$,
the subset $W(G,U_1,U_2,U_3)$ contains
$(\Omega \setminus \omega) \times K$
as a closed subset.  On the other hand, clearly the subset $\{z' \in \mathbb{C}^{m_1};\,
\operatorname{Im} \varphi(z') \le \epsilon\}$
($\epsilon > 0$) is a neighborhood of $z' = 0$,
the subset $(\Omega \setminus \omega) \times K$
contains $W(G,U',U_2,U_3)$
for a sufficiently small open neighborhood $U'$ of the origin in $\mathbb{C}^{m_1}_{z'}$
as a closed subset. Hence we have a conclusion (\ref{eq:limit_G_omega_equive}).

Then the result follows from (\ref{eq:vanishing_omega_Omega})
and
\[
m_1 + m'_2 = m_1 + m_2 + \# (I^*(\theta_*) \cap I_{\mathbb{C}})
= \#I + \# (I^*(\theta_*) \cap I_{\mathbb{C}}) = N(\theta_*).
\]
This completes the proof.
\end{proof}

We also have the similar results for $\mathcal{O}^t_X$ and $\mathcal{O}_X$
by employing the same argument as that in the above proof, which
is much easier because we can take $G \times Z$ as a cone $G$ in the proof,
\begin{teo}
We have
\[
H^k(\mu_\chi(\mathcal{F}))_{\theta_*} = 0\qquad (k \ne \operatorname{codim}_{\mathbb{C}}N
= \#I),
\]
where $\mathcal{F}$ is either $\mathcal{O}^t_{X_{sa}}$ or $\mathcal{O}_X$.
\end{teo}

\subsection{The typical examples}
As the results given in the previous subsection has been considered in a fairly general situation, we here describe
the corresponding results for typical cases.

\

We first consider the corresponding result for families of complex submanifolds,
i.e., $I = I_\mathbb{C}$.
Let $X$ be a complex manifold and $\chi = \{Z_1,\dots,Z_\ell\}$ a family of closed
complex submanifolds of $X$ which satisfies the conditions H1, H2 and H3.
Set $Z = Z_1 \cap \dots \cap Z_\ell$.
Let $p = (q;\, \zeta) = (q;\, \zeta^{(1)},\dots,\zeta^{(\ell)}) \in S^*_{\chi}$.
Remember that the subset $J^*(p)$ of $\{1,\dots,\ell\}$ was defined by
\[
J^*(p) := \{j \in \{1,\dots,\ell\};\,
\zeta^{(\alpha)} = 0 \text{ for all $\alpha$ with $Z_j \subset Z_\alpha$}\}.
\]
We also define $\widehat{J}^*(p)$ by the subset of $J^*(p)$
that consists of the minimal elements with respect to the order relation
$k \prec j \iff Z_{k} \subsetneq Z_{j}$ for $k, j \in J^*(p)$.
Now we define the integer $N(p)$ by
\begin{equation}
N(p) = \operatorname{codim}_{\mathbb{C}} Z +
\sum_{j \in \widehat{J}^*(p)}
\operatorname{codim}_{\mathbb{C}} Z_j.
\end{equation}
Then the following theorem immediately comes from Theorem \ref{th:vanishing-for-mm}.
\begin{cor}
We have
\[
H^k(\mu_\chi(\OW_{X_{sa}}))_p = 0\quad (k \ne N(p)).
\]
\end{cor}
\begin{oss}
	In the complex case,  the result depends on a point $p \in S^*_{\chi}$.
For example,
$N(p) = 2 \operatorname{codim}_{\mathbb{C}} Z$
if $p = (q;\,0,\dots,0)$ and
$N(p) = \operatorname{codim}_{\mathbb{C}} Z$
if all the $\zeta^{(j)}$'s are non-zero.
\end{oss}

\begin{cor}
We have
\[
H^k(\mu_\chi(\mathcal{F}))_p = 0\quad
(k \ne \operatorname{codim}_{\mathbb{C}}Z),
\]
where $\mathcal{F}$ is $\mathcal{O}_X^t$ or $\OO_X$.
\end{cor}
\begin{df}
The sheaf of holomorphic microfunctions along $\chi$ in $S^*_{\chi}$ is defined by
\begin{equation}
\C^{\mathbb{R}}_{\chi} :=
\mu_\chi(\OO_{X}) \underset{\mathbb{Z}_{S^*_\chi}}{\otimes} or_{S^*_\chi}[\operatorname{codim}_{\mathbb{C}} Z],
\end{equation}
where $or_{S^*_\chi}$ denotes the orientation sheaf of $S^*_\chi$.
We also define $\C^{\mathbb{R},f}_{\chi}$ and
$\C^{\mathbb{R}, \mathrm{w}}_{\chi}$ by replacing $\OO_X$
in the above definition with
$\mathcal{O}_X^t$ and $\OW_X$ respectively.
\end{df}
Note that
$\C^{\mathbb{R}}_{\chi}$ and
$\C^{\mathbb{R},f}_{\chi}$ are really sheaves on $S^*_\chi$.
We also note that $\C^{\mathbb{R}, \mathrm{w}}_{\chi}$ is a complex.
It is, however, concentrated in degree $0$
outside the zero section, i.e., $\{(z;\zeta^{(1)},\dots, \zeta^{(\ell)})
\in S^*_\chi;\, \zeta^{(j)} \ne 0\}$.

\

Next we consider the corresponding result for the case $I = I_\mathbb{R}$.
Let $M$ be a connected real analytic manifold and $X$ its complexification.
Let $\Theta_1$, $\dots$, $\Theta_\ell$ be real analytic vector subbundles of
$TM$ which are involutive, that is, $[\theta_1,\, \theta_2] \in \Theta_k$
for any vector fields $\theta_1, \theta_2 \in \Theta_k$.
We denote by $\Theta^{\mathbb{C}}_k \subset TX$
the complex vector subbundle over $X$ that is
a complexification of $\Theta_k$ near $M$.
Now we introduce the conditions for $\Theta_k$'s which are counterparts of the ones H1, H2 and H3.
Set, for $1 \le k \le \ell$,
\[
\operatorname{NR}(k) := \{j \in \{1,\dots,\ell\};\; \Theta_j \nsubseteq \Theta_k, \Theta_k \nsubseteq \Theta_j\}
\]
Then we assume that, for any $q \in M$ and any $k$
with $\operatorname{NR}(k) \ne \emptyset$,
\[
(TM)_q = (\Theta_k)_q + \left(\underset{j \in \operatorname{NR}(k)}{\bigcap} (\Theta_j)_q\right).
\]
We also assume that,
for simplicity,
$\Theta_k$'s are mutually distinct,
i.e., $\Theta_{k_1} \ne \Theta_{k_2}$ if $k_1 \ne k_2$.
For a local description of $\Theta_k$'s, we have the lemma below.
\begin{lem}
Under the above situation, there exist subsets $I_j$ ($j=1,\dots,\ell$)
of $\{1,\dots,n\}$ satisfying the condition (\ref{eq:cond-H}) for which
the following holds.
For every $q \in M$, there exist an open neighborhood $U \subset M$ of $q$ and
a real analytic coordinates system $(x_1,\dots,x_n)$ of $U$
such that, in $U$, each $\Theta_k$ is given by
\[
\{(x;\,\nu) \in T\mathbb{R}^n;\, \nu_i = 0\, (i \in I_k)\}.
\]
\end{lem}{\label{lem:summersive-local-system}}
\begin{proof}
As $\Theta_k$ is involutive, there exist locally real analytic functions
$\pi^{(k)}_i$ ($i=1$, $\dots$, $\operatorname{codim}_{\mathbb{R}}\Theta_k$)
such that
\[
\Theta_k = \{(q;\nu) \in TM;\, d\pi^{(k)}_i(q)(\nu) = 0\,
(i=1,\dots,
\operatorname{codim}_{\mathbb{R}}\Theta_k) \}.
\]
Then, employing the same argument as that in Proposition 1.2 in \cite{HP},
we can choose desired coordinate functions
from these $\pi^{(k)}_i$'s.
\end{proof}

Let $N_{M,\,j} \subset X$ $(j=1,\dots,\ell)$ be the union of the
complex integral submanifolds of the involutive complex vector bundle
$\Theta^{\mathbb{C}}_j \subset TX$ passing through each point
$q \in M$, that is,
\[
N_{M,\,j} := \underset{q \in M}{\bigcup} \mathcal{L}(\Theta^\mathbb{C}_j,\,q)
\]
where $\mathcal{L}(\Theta^{\mathbb{C}}_j,\,q)$ denotes the complex integral
submanifold
of $\Theta_j$
passing through the point $q$.
Set
\[
\chi := \{N_{M,\,1}, \dots, N_{M,\,\ell}\},\quad
N_M := N_{M,\,1} \cap \dots \cap N_{M,\,\ell} \subset X.
\]
These $N_{M,\,j}$'s and $S^*_{\chi}$ are locally described as follows:
Let
\[
(z_1 = x_1 + \sqrt{-1}y_1,\,\dots,\, z_n = x_n + \sqrt{-1}y_n).
\]
be a system of local coordinates of $X$ which are complexification
of the coordinates $(x_1,\dots,x_n)$ given in Lemma
{\ref{lem:summersive-local-system}}.
Then it follows from the lemma that we have
\[
N_{M,\,j} = \{z \in X;\, \operatorname{Im}z_i = 0\, (i \in I_j)\}
\qquad (j=1,\dots,\ell),
\]
where $I_j$'s were given in the lemma. We also set $I_0 = \widehat{I}_0$ by (\ref{eq:def-I_0}).
Denote by $z^{(j)} = x^{(j)} + \sqrt{-1}y^{(j)}$
the coordinates variables $z_i = x_i + \sqrt{-1}y_i$'s with $i \in \widehat{I}_j$
and by $\zeta^{(j)} = \xi^{(j)} + \sqrt{-1}\eta^{(j)}$
the dual variables of $z^{(j)} = x^{(j)} + \sqrt{-1}y^{(j)}$
for $j=0,\dots,\ell$.
Then the coordinates of $S^*_{\chi}$ are locally given by
\[
(z^{(0)}, x^{(1)},\dots, x^{(\ell)};\, \sqrt{-1}\eta^{(1)},\dots,\sqrt{-1}\eta^{(\ell)}).
\]

\begin{cor}
Let $p \in S^*_{\chi}$.  Then we have
\[
H^k(\mu_\chi(\mathcal{F}))_p = 0\quad
(k \ne \operatorname{codim}_{\mathbb{R}} N_M),
\]
where $\mathcal{F}$ is either $\OW_{X_{sa}}$, $\mathcal{O}^t_{X_{sa}}$ or $\OO_X$.
\end{cor}
\begin{oss}
In this case, the result is independent of a point $p \in S^*_{\chi}$ contrary to
the complex case.
\end{oss}

\newcommand{\CO}{{\mathcal{C}_{N_M,\,\chi}}}
\newcommand{\CWO}{{\mathcal{C}^{\mathrm{w}}_{N_M,\,\chi}}}
\newcommand{\CTO}{{\mathcal{C}^f_{N_M,\,\chi}}}
\begin{df}
The sheaf of microfunctions along $\chi$ with holomorphic parameters is defined by
\begin{equation}
\CO :=
\mu_\chi(\OO_{X}) \underset{\mathbb{Z}_{S^*_\chi}}{\otimes} or_{S^*_\chi}
[\operatorname{codim}_{\mathbb{R}} N_M],
\end{equation}
where $or_{S^*_\chi}$ denotes the orientation sheaf of $S^*_\chi$.
We also define $\CTO$ and
$\CWO$ by replacing $\OO_X$
in the above definition with
$\mathcal{O}_X^t$ and $\OW_X$ respectively.
\end{df}
Note that these are really sheaves in $S^*_\chi$, that is, they are concentrated
in degree $0$ everywhere.

\section{Applications to $\mathcal{D}$-modules}\label{sec:AD}
In this section, we consider applications of multi-microlocalizations to $\mathcal{D}$-module theory.
First, recall  the notation  of  \S\,  \ref{sec:MM}; for example,
let $\tau^{}_i \colon E^{}_i \to Z$ $(1\leq i \leq \ell)$ be vector bundles  over  $Z$,  and let $E^*_i$
be the dual   bundle of $E^{}_i$.
\begin{teo}[cf.\ \cite{Uc95}]\label{FundTri}
Let $F$ be  a multi-conic object on  $E$. Then there exists a natural isomorphism
\[
\tau^!R \tau^{}_{!}  F \earrow
Rp^{}_{1*} \,p^{\,!}_{2}(F^{\wedge^{}_E}),
\]
and the natural morphism
$
F \to \tau^{!}R\tau^{}_{!}\,F$
is embedded  to the following distinguished triangle\textup{:}
\[
F \to  \tau^{!}  R\tau^{}_{!} F \to
Rp^{+}_{1*}p^{+!}_{2}(F^{\wedge^{}_E}) \xrightarrow{+1}.
\]
\end{teo}
\begin{proof}
If $\ell=1$, the results follows by Lemma A.2 of \cite{Uc95}. Assume $\ell >1$,  and set
$E':= \smashoperator[r]{\mathop{\times}_{Z, i=2}^\ell}E^{}_i$,
$E'{}^*:= \smashoperator[r]{\mathop{\times}_{Z, i=2}^\ell}E^{*}_i$, and $P'_{E'}:= \smashoperator{\mathop{\times}_{Z, i=2}^\ell}P'_i$.
Moreover, let $\wedge_{E'}$ (resp $\vee^*_{E'}$) be the composition of $\wedge_i$ (resp. $\vee^*_i$) for $i=2,\dots,\ell$.
Consider:
\[
\xymatrix @C=3em @R=2em{
 & E \ftimes_Z E^*  \ar @/_4ex/[ld]_-{p^{}_1} \ar@/^4ex/[rr]^-{p^{}_2}
 \ar[r]_-{p^{}_{E',2} }
\ar[d]^-{p^{}_{1,1}}
\ar@{}[dr] | {\displaystyle\square}
& E^{}_1 \ftimes_Z E^*
\ar[r]_-{p^{}_{1,2}}\ar[d]^-{p^{}_{1,1}}& E^*
\\
E& E^{}_1 \ftimes_Z E' \ftimes_Z E'{}^*  \ar[l]^-{p^{}_{E',1}} \ar[r]_-{p^{}_{E',2}}
& E^{}_1 \ftimes_Z E'{}^* &
}\]
By the commutative diagram
\[
\xymatrix @C=3em{
 E\ar[r]^-{ \tau^{}_1}
\ar[d]^-{\tau^{}_{E'}}
\ar@{}[dr] | {\displaystyle\square}
& E'\ar[d]^-{\tau^{}_{E'}}
\\
E^{}_1\ar[r]_-{ \tau^{}_1 } & Z
}\]
 we have
\begin{align*}
\tau^{!}_{E'}\,&R\tau^{}_{E'!} \,\tau^{!}_1 \,
R\tau^{}_{1!} \, F
 = \tau^{!}_{E'} \,R\tau^{}_{E'!} \,\tau^{-1}_1
R\tau^{}_{1!} \, F \tens \omega^{}_{E/Z}
\\
&
= \tau^{!}_{E'} \, \tau^{-1}_1 R\tau^{}_{E'!} \,
R\tau^{}_{1!} \,F \tens \omega^{}_{E/Z}
= \tau^{!}_{E'} \,\tau^{!}_1 R\tau^{}_{E'!} \,
R\tau^{}_{1!} \, F =\tau^{!}  R\tau^{}_{!} \,F.
\end{align*}
Hence,  by Lemma A.2 of \cite{Uc95} and induction hypothesis, we obtain:
\begin{equation}\label{eq.s.2}
\vcenter{
\xymatrix @!0@C=9em @R=8ex{
F \ar[d]\ar[r]^-{\dsim} &
F^{\wedge_{E'}\vee^*_{E'}} \ar[r] ^-{\dsim}
&F^{\wedge_{E'}\vee^*_{E'}\wedge^{}_1\vee^{*}_{1} } \ar[d]
\\
 \tau^!_{E'}R \tau^{}_{E'!} F \ar[d] \ar[rr] ^-{\dsim}& &
Rp^{}_{E',1*} \,p^{\,!}_{E',2}(F^{\wedge_{E'}}) \ar[d]
\\
\tau^!_{1} R \tau^{}_{1!}\tau^!_{E'} R \tau^{}_{E'!}  F \ar[rr] ^-{\dsim}
 \ar@{=}[d]& & \smash{\tau^!_{1}
R \tau^{}_{1!}Rp^{}_{E',1*} \,p^{\,!}_{E',2}
(F^{\wedge_{E'}}})
\ar[d]^-{\displaystyle\! \wr}
\\
\tau^!R \tau^{}_{!}  F \ar[rr] ^-{\dsim}& &
\smash{Rp^{}_{1,1*} \,p^{\,!}_{1,2}((Rp^{}_{E',1*} \,p^{\,!}_{E',2}
(F^{\wedge_{E'}}) )^{\wedge^{}_1}})
}} \end{equation}
For the same reasoning as in \eqref{eq.s.4},
 for any multi-conic object $H$ on $E^{}_1\ftimes_Z E'{}^*$
we have
\[
H^{\vee^{*}_i\vee ^{}_{1}}=H^{\vee ^{}_{1} \vee^{*}_i} .
\]
Therefore, we have
\begin{equation}
\label{eq.s.5}
H^{\vee ^{*}_{i} \wedge^{}_1}
 =(H^{\vee ^{*}_{i}\vee ^{}_{1}})^{\mathrm{id} \times a} \tens  \omega^{\otimes -1}_{E^{}_1/Z}
= (H^{\vee^{}_1\vee ^{*}_{i} })^{\mathrm{id} \times a} \tens  \omega^{\otimes -1}_{E^{}_1/Z}
= H^{\wedge^{}_1\vee ^{*}_{i}}
\end{equation}
Thus, by Proposition \ref{propF} (2)  and \eqref{eq.s.5} we have
\begin{equation}
\label{eq.s.6}
F^{\wedge_{E'}\vee^*_{E'}\wedge^{}_1\vee^{*}_{1} } =
Rp'_{1*}p^{\prime \,!}_{2} (F^{\wedge_E}).
\end{equation}
Lastly, by Proposition 3.7.13 of \cite{KS90}  we have
\begin{equation}\label{eq.s.7}
\begin{split}
Rp^{}_{1,1*}& \,p^{\,!}_{1,2}((Rp^{}_{E',1*} \,p^{\,!}_{E',2}
(F^{\wedge_{E'}}))^{\wedge^{}_1})
\\
&
 =Rp^{}_{1,1*} \,p^{\,!}_{1,2}\,Rp^{}_{E',1*}((p^{\,!}_{E',2}
(F^{\wedge_{E'}}) )^{\wedge^{}_1})
\\
& =Rp^{}_{1,1*} \,p^{\,!}_{1,2}\,Rp^{}_{E',1*}\,p^{\,!}_{E',2}
(F^{\wedge_{E'}\wedge^{}_1})
= Rp^{}_{1*}p^{!}_2(F^{\wedge_E}).
\end{split}
\end{equation}
Therefore, by \eqref{eq.s.2}, \eqref{eq.s.6},   \eqref{eq.s.7} and the definition of $P^+$,
we have
\[
\xymatrix @C=1em @R=3ex{
F \ar[r] \ar[d]^-{\displaystyle\! \wr}& \tau^{!}\, R\tau_! F
\ar[d]^-{\displaystyle\! \wr}
\\
F^{\wedge_{E'}\vee^*_{E'}\wedge^{}_1\vee^{*}_{1} }  \ar@{-}[d]^-{\displaystyle\! \wr}\ar[r]
&
Rp^{}_{1,1*} \,p^{\,!}_{1,2}((Rp^{}_{E',1*} \,p^{\,!}_{E',2}
(F^{\wedge_{E'}}))^{\wedge^{}_1})  \ar@{-}[d]^-{\displaystyle\! \wr}
\\
Rp'_{1*}\,p^{\prime \,!}_{2} (F^{\wedge_{E'}}) \ar[r] &
Rp^{}_{1*}\,p^{!}_{2} (F^{\wedge_{E'}})
 \ar[r] & Rp^+_{1*}\,p^{+ \,!}_{2} (F^{\wedge_{E'}}) \ar[r]^-{+1} & {}.
}
\]
The commutativity follows from the constructions.
 \end{proof}
Therefore, we obtain the following:
\begin{teo}
Let $X$ be a real analytic manifold, and assume that the family
$\chi = \{M^{}_i\}_{i=1}^\ell$ of submanifolds in $X$ satisfies conditions \textup{H1},  \textup{H2} and \textup{H3}.
Set  $M:= \bigcap_{i=1}^\ell M^{}_i$.
Then for any $F \in  \bDb(k_{X_{\rm sa}})$, there exists the following distinguished triangle\textup{:}
\begin{equation}\label{tri1}
\nu^{}_{\chi}(F) \to \tau^{-1}R\varGamma^{}_Y(F) \tens \omega^{\otimes -1}_{M/X} \to
Rp^{+}_{1*}(p^{+}_{2})^{-1}\mu^{}_{\chi}(F) \tens \omega^{\otimes -1}_{M/X}  \xrightarrow{+1}.
\end{equation}
\end{teo}
\begin{proof}
By  the deinition of multi-microlocalization and Theorem \ref{FundTri}
we have (see also Remark \ref{!and!!})
\[
\nu^{}_{\chi}(F) \to  \tau^{!}  R\tau^{}_{!}\nu^{}_{\chi}(F)  \to
Rp^{+}_{1*}p^{+!}_{2}\mu^{}_{\chi}(F)  \xrightarrow{+1}.
\]
Since $\tau$ and $p^{+}_{2}$ are projections, we have
\begin{align*}
\tau^{!}& \simeq \tau^{-1}\tens  \omega^{}_{S_\chi/M}\simeq
\tau^{-1}\tens  \omega^{\otimes -1}_{M/X},
\\
p^{+!}_{2} &\simeq  (p^{+}_{2})^{-1}\tens  \omega^{}_{P^+/M}\simeq
 (p^{+}_{2})^{-1}\tens  \omega^{\otimes -1}_{M/X}.
 \end{align*}
Hence we prove the theorem.
\end{proof}
By Theorem \ref{teo: estimate cone},
under the identifications $T^*S^*_\chi = T^*S_\chi  =S_{\chi^*}$, we have
\[
\MS(\mu^{}_{\chi}(F))=\MS(\nu^{}_{\chi}(F))\subset C^{}_{\chi^*}(\MS (F) ).
\]
In particular we obtain
\begin{equation}\label{suppmu}
\supp\, \mu^{}_{\chi}(F) \subset S^*_{\chi} \cap  C^{}_{\chi^*}(\MS (F) ).
\end{equation}
Thus we obtain:
\begin{cor}\label{tri1a}
If $\dot{S}^*_{\chi}\cap  C^{}_{\chi^*}(\MS (F) )=\emptyset$,
then
\[
\nu^{}_{\chi}(F) \earrow\tau^{-1}R\varGamma^{}_Y(F) \tens \omega^{\otimes -1}_{M/X}\,.
\]
\end{cor}

Now, let $\chi=\{Y^{}_i\}_{i=1}^\ell$, and assume that each $Y^{}_i$  and $Y:= \bigcap_{i=1}^\ell Y^{}_i$ are  complex submanifolds of $X$.
As usual, let $\mathcal{D}_X$ be the sheaf of holomorphic differential operators on $X$.
Let $\M$ be a coherent $\mathcal{D}^{}_X$-module, and  $\ch \M $ the characteristic variety of $\M$.
Then,  for  $F=\rh^{}_{\mathcal{D}^{}_X}\!(\M,\OO^{}_X)$, it is known  that $\MS(F)=\ch \M$.
From   \eqref{tri1},  we have
\begin{multline*}
\rh^{}_{\mathcal{D}^{}_X}\!(\M,\nu^{}_{\chi}(\OO^{}_X))
\to
\tau^{-1}\rh^{}_{\mathcal{D}^{}_X}\!(\M,R\varGamma^{}_Y(\OO^{}_{X}))
\tens \omega^{\otimes -1}_{Y/X}
\\
 \to
Rp^{+}_{1*}(p^{+}_{2})^{-1}
\rh^{}_{\mathcal{D}^{}_X}\!(\M,\mu^{}_{\chi}(\OO^{}_{X}))\tens \omega^{\otimes -1}_{Y/X}  \xrightarrow{+1}.
\end{multline*}
Let $f\colon Y \hookrightarrow X$ be the canonical embedding.
We can define the following  natural mappings
associated with $f$:
\[
T^* Y \xleftarrow{f^{}_d} Y \smashoperator{\ftimes_{X}} T^*X
\xrightarrow{f^{}_\pi} T^*X.
\]
We define the \textit{inverse image} of $\M$ by
\[
\boldsymbol{D}f^* \!
\M  :=
\OO^{}_{Y}\smashoperator{\tens_{\hspace{1em} f^{-1}\OO^{}_X}^{L}}
 f^{\,-1} \! \M.
  \]
Assume  that  $Y$ is
non-characteristic for  $\M$; that is, $T^*_Y X \cap \ch \M\subset  T^*_XX$.
Then, it is known that  $\boldsymbol{D}f^* \!  \M$ is identified with
$Df^* \!  \M
:=H^{0}\boldsymbol{D}f^* \!  \M$, and $Df^* \! \M$ is  a coherent $\mathcal{D}^{}_Y$-module.

\begin{teo}\label{NC}
Assume that  $Y$ is non-characteristic for $\M$. Then
\begin{align*}
\rh^{}_{\mathcal{D}^{}_X}\!(\M,\nu^{}_{\chi}(\OO^{}_X))
& \earrow
\tau^{-1}\rh^{}_{\mathcal{D}^{}_Y}\!(Df^* \!\M,\OO^{}_{Y})
\\
&\simeq \tau^{-1}f^{\,-1}\rh^{}_{\mathcal{D}^{}_X}\!(\M,\OO^{}_{X}).
\end{align*}
\end{teo}
\begin{proof}
By the  non-characteristic condition and Cauchy-Kovalevskaja-Kashiwara theorem, we obtain the  following isomorphisms:
\begin{align*}
\rh^{}_{\mathcal{D}^{}_X}\!(\M,R\varGamma^{}_Y(\OO^{}_{X})) \tens \omega^{\otimes -1}_{Y/X}
& \simeq
\rh^{}_{\mathcal{D}^{}_Y}\!(Df^* \!\M,\OO^{}_{Y})
\\
&\simeq f^{\,-1}\rh^{}_{\mathcal{D}^{}_X}\!(\M,\OO^{}_{X}).
\end{align*}
Hence by Corollary  \ref{tri1a}, we may show $\dot{S}^*_{\chi}\cap   C^{}_{\chi^*}(\ch \M)=\emptyset$.
Assume that there exists a point
\[
 (x^{(0)};\, \xi^{(1)},\dots,\xi^{(\ell)}) \in
\dot{S}^*_{\chi}\cap   C^{}_{\chi^*}(\ch \M).
\]
 Then by Theorem \ref{thm: estimate sequences} there exist sequences
\begin{align*}
\{(x^{(0)}_{n},x^{(1)}_n,\dots,x^{(\ell)}_{ n};\xi^{(0)}_{n},\xi^{(1)}_n,\dots,\xi^{(\ell)}_{ n})\}_{n=1}^\infty
&\subset \ch \M,
\\
\{(c^{(1)}_{n},\dots,c^{(\ell)}_{n})\} _{n=1}^\infty
&\subset (\RP)^\ell,
\end{align*}
 such that:
\[\left\{
    \begin{aligned}
     \lim_{n\to \infty} & c^{(j)}_{n} = \infty, \qquad (j=1,\dots,\ell), \\
     \lim_{n\to \infty} &(x^{(0)}_{n},x^{(1)}_nc^{(\JJ_1)}_{n},\dots,x^{(\ell)}_{n}c^{(\JJ_\ell)}_n
;\xi^{(0)}_nc_n,\xi^{(1)}_{n}c^{(\JJ_1^c)}_n,\dots,\xi^{(\ell)}_{ n}c^{(\JJ_\ell^c)}_n)
\\
& = (x^{(0)},0,\dots,0;0,\xi^{(1)},\dots,\xi^{(\ell)}),
    \end{aligned}
  \right.
\]
where $c^{}_n :=\prod_{j=1}^\ell c^{(j)}_{n}$, $\JJ^c_j:=\{1,\dots,\ell\}\smallsetminus\JJ_j$,
and $c^{(J)}_n := \prod_{j \in J}c^{(j)}_n$ for any $J \subseteq \{1,\dots,\ell\}$.
In particular we have $\lim\limits_{n\to \infty} (x^{(1)}_n,\dots,x^{(\ell)}_{n},\xi^{(0)}_nc^{}_n) = (0,\dots,0,0)$.
Since $(\xi^{(1)},\dots,\xi^{(\ell)}) \ne 0$, we may assume that
\[t^{}_n:= |(\xi^{(0)}_n,\xi^{(1)}_{n},\dots,
\xi^{(\ell)}_{ n})|>0.
\]  We consider the sequence  $\{(x^{(0)}_{n},x^{(1)}_n,\dots,x^{(\ell)}_{ n};
t^{-1}_n(\xi^{(0)}_n, \xi^{(1)}_{n},\dots,\xi^{(\ell)}_{ n}))\}_{n=1}^\infty \subset \ch \M $.
By extracting subsequence,  we may assume that there exists $\zeta^{}_0\ne 0$ such that
\[
\lim\limits_{n\to \infty} t^{-1}_n(\xi^{(0)}, \xi^{(1)}_{n},\dots,\xi^{(\ell)}_{ n})= \zeta ^{}_0\, .
\]
 Choose $1\leqslant k \leqslant \ell$ with $
\lim\limits_{n\to \infty} \xi^{(k)}_n c^{(\JJ_k^c)}_n =\xi^{(k)}\ne 0$.
Since $\lim\limits_{n\to \infty}  c^{(j)}_{n} = \infty $, we have
\[
\lim\limits_{n\to \infty} \dfrac{\xi^{(0)}_n}{\,\xi^{(k)}_n\,}
= \lim\limits_{n\to \infty} \dfrac{\xi^{(0)}_nc^{}_n}{\,\xi^{(k)}_n c^{(\JJ_k^c)}_n c^{(\JJ_k)}_n\,}=0.
\]
Therefore, we see that  $\zeta ^{}_0 =(0, \zeta^{}_{01},\dots, \zeta^{}_{0\ell})\ne 0$. Hence
\begin{align*}
\lim_{n\to \infty}&(x^{(0)}_{n},x^{(1)}_n,\dots,x^{(\ell)}_{ n}; t^{-1}_n(\xi^{(0)}, \xi^{(1)}_{n},\dots,\xi^{(\ell)}_{ n})
)
\\
&=
(x^{(0)},0,\dots,0; \zeta^{}_0) \in \dot{T}^*_YX \cap \ch \M ,
\end{align*} which  contradicts
the non-characteristic condition. Thus $\dot{S}^*_{\chi}\cap   C^{}_{\chi^*}(\ch \M) =\emptyset$.
Hence we obtain the desired result.
\end{proof}

\begin{es}
Take a  point $p=(q;\,\xi) \in \underset{X,1\leq j\leq \ell}{\times}T_{Y_j}\iota(Y_j)$, and set
$x_0:= \tau(p)\in Y$. Recall the notation in Remark \ref{stalk-nu}.
Then under  the assumption  of Theorem \ref{NC},
every solution  to $\mathcal{M}$ defined on  $\operatorname{Cone}_\chi(p,\,\epsilon)$
 for some $\epsilon > 0$ extends automatically  a  solution   defined on a full neighborhood of $x_0$.
\end{es}
Next, we consider the real cases. Let $M$ be  a real analytic manifold, and $\chi=\{N^{}_i\}_{i =1}^{\ell}\subset M$.
Assume that each $N^{}_i$  and $N:= \bigcap_{i=1}^\ell N^{}_i$ are  real analytic  submanifolds of $M$.
We consider the multi-normal deformation $\widetilde{M}^{}_\chi$ along $\chi$.
Let $X$ be the complexification of $M$,  and $Y$  the complexification of $N$ in $X$.
Let $\iota \colon M \hookrightarrow X$ the canonical embedding.
Let $\B^{}_M$ be the sheaf of hyperfunctions on $M$.
 Then by \eqref{tri1} we obtain
\begin{multline*}
\rh^{}_{\mathcal{D}^{}_X}\!(\M,\nu^{}_{\chi}(\B^{}_M))
\to
\tau^{-1}\rh^{}_{\mathcal{D}^{}_X}\!(\M,R\varGamma^{}_N(\B^{}_{M}))
\tens \omega^{\otimes -1}_{N/M}
\\
 \to
Rp^{+}_{1*}(p^{+}_{2})^{-1}
\rh^{}_{\mathcal{D}^{}_X}\!(\M,\mu^{}_{\chi}(\B^{}_{M}))\tens \omega^{\otimes -1}_{N/M}  \xrightarrow{+1}.
\end{multline*}
For  any conic subset $A \subset T^*X$ we  can define
$\iota^{\#}(A):= T^*M \cap C^{}_{\smash{T_M^*}X}(A)$ (\cite[Definition 6.2.3]{KS90}).
Note that  $(x^{}_0;\xi^{}_0)
\in \iota^{\#}(A) $ if and only if  there exists   a  sequence
$\{(x^{}_\nu+\sqrt{-1}\, y^{}_\nu;\xi^{}_\nu+\sqrt{-1}\, \eta^{}_\nu)\}_{\nu=1}^\infty \subset A$ such that
\[
\smashoperator{\lim_{\nu\to \infty}}\,
(x^{}_\nu+\sqrt{-1}\, y^{}_\nu;\xi^{}_\nu)  = (x^{}_0;\xi^{}_0),\quad  \smashoperator{\lim_{\nu\to \infty}}\,|y^{}_\nu|\,|\eta^{}_\nu| = 0.
\]
\begin{teo}
Assume that   $N\hookrightarrow M$ is hyperbolic for
$\M$\textup{;} that is,
\begin{equation}\label{hypM}
\dot{T}^*_NM \cap \iota^{\#}(\ch \M)=\emptyset.
\end{equation}
 Then
\[
\rh^{}_{\mathcal{D}^{}_X}\!(\M,\nu^{}_{\chi}(\B^{}_M))
\earrow
\tau^{-1}\rh^{}_{\mathcal{D}^{}_Y}\!(Df^* \!\M,\B^{}_{N}).
\]
\end{teo}
\begin{proof}
By \eqref{hypM}, we see that
$Y$ is non-characteristic for $\M$ on a neighborhood of $N$.
By the non-characteristic division theorem, we have
\[
\rh^{}_{\mathcal{D}^{}_X}\!(\M,R\varGamma^{}_N(\B^{}_{M}))
\tens \omega^{\otimes -1}_{N/M}\simeq
\rh^{}_{\mathcal{D}^{}_Y}\!(Df^* \!\M,\B^{}_{N}).
\]
By Corollary 6.4.4 of \cite{KS90},
we have
\[
\MS(\rh^{}_{\mathcal{D}^{}_X}\!(\M,\B^{}_{M})) \subset
 \iota^{\#}\MS(\rh^{}_{\mathcal{D}^{}_X}\!(\M,\OO^{}_{M}))
= \iota^{\#}(\ch \M).
\]
As in the proof of Theorem \ref{NC} we have
\[
\supp(
\rh^{}_{\mathcal{D}^{}_X}\!(\M,\mu^{}_{\chi}(\B^{}_{M}))) \cap \dot{S}^*_{\chi}
\subset
\dot{S}^*_{\chi}\cap   C^{}_{\chi^*}( \iota^{\#}(\ch \M)) =\emptyset.
\]
and we obtain the desired result.
\end{proof}
 We shall give another  result. Let us take $\chi =\{N,M\}$ with $N\subset M \subset X$ (Takeuchi's case).
Then $S_\chi= T_N M \ftimes_M T_MX$ and $S^*_\chi=T^*_N M \ftimes_M T^*_MX \simeq T^*_{N  \ftimes_M T^*_MX}T^*_MX$.
Under the notaion of \cite{Ta96},
we set
\begin{align*}
\nu_{NM}:= \nu_{\chi} \colon \bDb(k_{X})& \to  \bDb(k_{T_N M \ftimes_M T_MX}),
\\
\nu\mu_{NM}:= \nu_{\chi}^{\wedge^{}_{2}} \colon \bDb(k_{X})& \to  \bDb(k_{T_N M \ftimes_M T^*_MX}),
\\
\mu_{NM}:= \mu_{\chi} \colon \bDb(k_{X})& \to  \bDb(k_{T^*_N M \ftimes_M T^*_MX}).
\end{align*}
As usual,  let
$\E^{}_X$ be the sheaf of \textit{ring of microdifferential operators} on $T^*X$ and
$\{\E^{}_X(m)\}_{m \in
\mathbb{Z}}^{}$ the usual \textit{order filtration} on
$\E^{}_X$.
Let    $U$ be  a $\mathbb{C}^\times$-conic   open subset of $\dot{T}^*X$,
and  $\varSigma$ a $\mathbb{C}^{\times}$-conic involutory closed  analytic subset of $U$.
 Set $
\mathcal{I}^{}_\varSigma :=\{P \in \E^{}_X(1)|^{}_U;\,\sigma^{}_1(P)|{}^{}_\varSigma \equiv 0\}$ and
$ \E^{}_\varSigma := \bigcup_{m \in\mathbb{N}^{}_0}\mathcal{I}^{m}_\varSigma$.
 Here $\sigma^{}_m(P)$ denotes the \textit{principal symbol} of
$P \in \E ^{}_X(m) $, and we set $\mathcal{I}^{0}_\varSigma := \E ^{}_X(0)|^{}_U $.
Namely, $\E^{}_\varSigma\subset  \E^{}_X |^{}_U $
is a  subring generated by $\mathcal{I}^{}_\varSigma$.
\begin{df}
(1) Let    $U$ be  a $\mathbb{C}^\times$-conic   open subset of $\dot{T}^*X$,
and  $\varSigma$ a $\mathbb{C}^{\times}$-conic involutory closed  analytic subset of $U$.
Let  $\mathfrak{M}$ be a coherent $\E^{}_X$-module defined on $U$.

(a)  An  $\E^{}_\varSigma$ submodule  $\mathfrak{L}$ of $\mathfrak{M}$
is called  an  \textit{$\E^{}_\varSigma$-lattice} of $\mathfrak{M}$ if  $\mathfrak{L}$
is  $\E ^{}_X(0)$-coherent
 and  $\E^{}_X\mathfrak{L} =\mathfrak{M}$.

(b)  We say that
$\mathfrak{M}$ \textit{has regular singularities along} $\varSigma$  if
for any  $p^*\in U$,  there exist an open neighborhood $V$ of $p^*$ and  an
$\E^{}_\varSigma$-lattice $\mathfrak{L} \subset \mathfrak{M}|{}^{}_{V}$.

(2)
Let $\varSigma$ be a $\mathbb{C}^{\times}$-conic involutory closed analytic subset of
$\dot{T}^*X$, and  $\M$ a coherent $\mathcal{D}^{}_X$-module.
Then we say that $\M$
\textit{has regular singularities along} $\varSigma$  if so does   $\E^{}_X
\smashoperator{\tens_{\pi^{-1} \mathcal{D}^{}_X}}\pi^{-1} \! \M$.
\end{df}
We
impose  the
\begin{con}\label{conRS}
(1) $\varLambda \subset \dot{T}^*X $  is   a    $\mathbb{C}^\times$-conic closed regular involutory  complex
submanifold,

(2) $\M$ has regular singularities
along $\varLambda$,

(3) $\dot{T}^*_NM \cap \iota^{\#}(\varLambda)=\emptyset$.
\end{con}
By Condition \ref{conRS} (2),  we have $\ch \M\subset \varLambda\sqcup \supp \M$.
Hence by virtue of  Condition \ref{conRS} (3), we see that
$Y$ is non-characteristic for $\M$ on neighborhood of $N$.
Let $\db_M$ be  the sheaf of distributions on $M$.
\begin{teo}
Assume Condition \ref{conRS}. Then
\[
\rh^{}_{\mathcal{D}^{}_X}\!(\M,\nu^{}_{\chi}(\db_M))
\earrow
\tau^{-1}\rh^{}_{\mathcal{D}^{}_Y}\!(Df^* \!\M,\db^{}_{N}).
\]
\end{teo}
\begin{proof}
Consider
\begin{multline*}
\rh^{}_{\mathcal{D}^{}_X}\!(\M,\nu^{}_{\chi}(\db^{}_M))
\to
\tau^{-1}\rh^{}_{\mathcal{D}^{}_X}\!(\M,R\varGamma^{}_N(\db^{}_{M}))
\tens \omega^{\otimes -1}_{N/M}
\\
 \to
Rp^{+}_{1*}(p^{+}_{2})^{-1}
\rh^{}_{\mathcal{D}^{}_X}\!(\M,\mu^{}_{\chi}(\db^{}_{M}))\tens \omega^{\otimes -1}_{N/M}  \xrightarrow{+1}.
\end{multline*}
Under Condition \ref{conRS}, we have
\begin{align*}
\rh^{}_{\mathcal{D}^{}_X}\!(\M,R\varGamma^{}_N(\db^{}_{M}))
\tens \omega^{\otimes -1}_{N/M} &\simeq
\rh^{}_{\mathcal{D}^{}_X}\!(\M,\varGamma^{}_N(\db^{}_{M}))
\tens \omega^{\otimes -1}_{N/M}
\\
&\simeq
\rh^{}_{\mathcal{D}^{}_Y}\!(Df^*\!\M,\db^{}_{N}),
\end{align*}
and
\[
\MS(\rh^{}_{\mathcal{D}^{}_X}\!(\M,\db^{}_{M})) \subset \iota^{\#}(\varLambda\sqcup \supp \M).
\]
%
%
As in the proof of Theorem \ref{NC} we have
\[
\supp(
\rh^{}_{\mathcal{D}^{}_X}\!(\M,\mu^{}_{\chi}(\db^{}_{M}))) \cap \dot{S}^*_{\chi}
\subset
\dot{S}^*_{\chi}\cap   C^{}_{\chi^*}(\iota^{\#}(\varLambda\sqcup \supp \M)) =\emptyset,
\]
and this entails that $Rp^{+}_{1*}(p^{+}_{2})^{-1}
\rh^{}_{\mathcal{D}^{}_X}\!(\M,\mu^{}_{\chi}(\db^{}_{M}))=0$,
and we obtain the desired result.
\end{proof}

\begin{teo}
Assume Condition \ref{conRS}. Then
\[
\rh^{}_{\mathcal{D}^{}_X}\!(\M,\nu^{}_{\chi}(\mathcal{C}^\infty_M))
\earrow
\tau^{-1}\rh^{}_{\mathcal{D}^{}_Y}\!(Df^* \!\M,\mathcal{C}^\infty_{N}).
\]
\end{teo}
\begin{proof}
Consider
\begin{multline*}
\rh^{}_{\mathcal{D}^{}_X}\!(\M,\nu^{}_{\chi}(\mathcal{C}^\infty_M))
\to
\tau^{-1}\rh^{}_{\mathcal{D}^{}_X}\!(\M,R\varGamma^{}_N(\mathcal{C}^\infty_{M}))
\tens \omega^{\otimes -1}_{N/M}
\\
 \to
Rp^{+}_{1*}(p^{+}_{2})^{-1}
\rh^{}_{\mathcal{D}^{}_X}\!(\M,\mu^{}_{\chi}(\mathcal{C}^\infty_{M}))\tens \omega^{\otimes -1}_{N/M}  \xrightarrow{+1}.
\end{multline*}
Under Condition \ref{conRS}, we have
\begin{align*}
\rh^{}_{\mathcal{D}^{}_X}\!(\M,R\varGamma^{}_N(\mathcal{C}^\infty_{M}))
\tens \omega^{\otimes -1}_{N/M} &\simeq
\rh^{}_{\mathcal{D}^{}_X}\!(\M,\W^{\infty}_{M,N})
\\
&\simeq
\rh^{}_{\mathcal{D}^{}_Y}\!(Df^*\!\M,\mathcal{C}^\infty_{N}),
\end{align*}
and
\[
\MS(\rh^{}_{\mathcal{D}^{}_X}\!(\M,\mathcal{C}^\infty_{M})) \subset \iota^{\#}(\varLambda\sqcup \supp \M).
\]
%
Here $\W^{\infty}_{M,N}$ is  the sheaf of Whitney functions on $N$.
Thus the proof is same as in Theorem \ref{NC}.
\end{proof}

\addcontentsline{toc}{section}{
\end{document}
\textbf{References}}
\begin{thebibliography}{16}


 \vspace{0.5cm}


\bibitem{De96} J. M. Delort, Microlocalisation simultan\'ee et probll\`eme de Cauchy ramifi\'e, Compositio Math, $\bf{100}$, pp. 171-204 (1996).


\bibitem{Du79} A. Dufresnoy, Sur l'op\'erateur $d''$ et les fonctions diff\'erentiables au sens de Whitney, Ann. Inst. Fourier, Grenoble, $\bf{29}$ N. $\bf{1}$ pp. 229-238 (1979).


\bibitem{Gr55} A. Grothendieck, Produits tensoriels topologiques et espaces nucl\'eaires, Mem. Amer. Math. Soc. $\bf{16}$ (1955).


\bibitem{HP} N. Honda, L. Prelli, Multi-specialization and multi-asymptotic expansions, Advances in Math. $\bf{232}$, pp. 432-498 (2013).


\bibitem{Ka88} A. Kaneko, Introduction to hyperfunctions, Mathematics and its Applications, Kluwer Academic Publishers Group (1988).


\bibitem{KKK} M. Kashiwara, T. Kawai, T. Kimura, Foundations of algebraic analysis, Princeton Math. Series $\bf{37}$, Princeton University Press (1986).


\bibitem{KS90} M. Kashiwara, P. Schapira, Sheaves on manifolds, Grundlehren der Math. $\bf{292}$, Springer-Verlag, Berlin (1990).


\bibitem{KS96} M. Kashiwara, P. Schapira, Moderate and formal cohomology associated with constructible sheaves, M\'emoires Soc. Math. France $\bf{64}$ (1996).


\bibitem{KS01} M. Kashiwara, P. Schapira, Ind-sheaves, Ast{\'e}risque $\bf{271}$ (2001).


\bibitem{KT01} H. Koshimizu, K. Takeuchi, On the solvability of operators with multiple characteristics, Comm. Partial Diff. Equations $\bf{26}$, pp. 1691-1720 (2001).


\bibitem{Pr08} L. Prelli, Sheaves on subanalytic sites, Rend. Sem. Mat. Univ. Padova Vol. $\bf{120}$, pp. 167-216 (2008).


\bibitem{Pr11} L. Prelli, Conic sheaves on subanalytic sites and Laplace transform, Rend. Sem. Mat. Univ. Padova Vol. $\bf{125}$, pp. 173-206 (2011).


\bibitem{Pr13} L. Prelli, Microlocalization of subanalytic sheaves, M\'emoires Soc. Math. France $\bf{135}$ (2013).


\bibitem{ST94} P. Schapira, K. Takeuchi, D\'eformation normale et bisp\'ecialisation, C. R. Acad. Sci. Paris Math. $\bf{319}$, pp. 707-712 (1994).


\bibitem{Ta96} K. Takeuchi, Binormal deformation and bimicrolocalization, Publ. RIMS, Kyoto Univ. $\bf{32}$ pp. 277-322 (1996).


\bibitem{Uc95} M. Uchida, A Generalization of Bochner's Tube Theorem in Elliptic Boundary Value Problems, Publ. RIMS, Kyoto Univ. $\bf{31}$ pp. 1065-1077 (1995).


\end{thebibliography}
